\documentclass[11pt]{amsart} 
\DeclareFontFamily{U}{matha}{\hyphenchar\font45}
\DeclareFontShape{U}{matha}{m}{n}{
      <5> <6> <7> <8> <9> <10> gen * matha
      <10.95> matha10 <12> <14.4> <17.28> <20.74> <24.88> matha12
      }{}
\DeclareSymbolFont{matha}{U}{matha}{m}{n}
\DeclareFontFamily{U}{mathx}{\hyphenchar\font45}
\DeclareFontShape{U}{mathx}{m}{n}{
      <5> <6> <7> <8> <9> <10>
      <10.95> <12> <14.4> <17.28> <20.74> <24.88>
      mathx10
      }{}
\DeclareSymbolFont{mathx}{U}{mathx}{m}{n}

\DeclareMathSymbol{\obot}         {2}{matha}{"6B}
\DeclareMathSymbol{\bigobot}       {1}{mathx}{"CB}

\usepackage{makecell}
\usepackage[pdfauthor={Congling Qiu}, 
        pdftitle={???}%
      dvips
      ]{hyperref}
\hypersetup{
  bookmarksnumbered=true,
  citecolor=black,
  pagecolor=black, 
  urlcolor=black,  
}
 
\usepackage[english]{babel}
\usepackage[toc,page]{appendix}

  \usepackage{multirow}

\usepackage[utf8]{inputenc}

\usepackage{xtab}
\usepackage{amsmath, amssymb
}
\usepackage{mathrsfs}
\usepackage[all]{xy}
\usepackage{extarrows}

\usepackage{enumerate}
\usepackage{mathtools,booktabs}
\usepackage{color}
\usepackage[nameinlink]{cleveref}

\usepackage{epstopdf} 
\usepackage{booktabs}

\numberwithin{equation}{section}

\theoremstyle{plain}
\newtheorem{proposition}{Proposition}[subsection]
\newtheorem{conj}[proposition]{Conjecture}
\newtheorem{cor}[proposition]{Corollary}
\newtheorem{lem}[proposition]{Lemma}
\newtheorem{thm}[proposition]{Theorem}
\newtheorem{prop}[proposition]{Proposition}

\theoremstyle{definition}
\newtheorem{defn}[proposition]{Definition}

\newtheorem{asmp}[proposition]{Assumption}

\theoremstyle{remark}
\newtheorem{rmk}[proposition]{Remark}

\numberwithin{equation}{section}

  \newcommand{\BA}{{\mathbb {A}}} \newcommand{\BB}{{\mathbb {B}}}
    \newcommand{\BC}{{\mathbb {C}}} \newcommand{\BD}{{\mathbb {D}}}
     \newcommand{\BF}{{\mathbb {F}}}
    \newcommand{\BG}{{\mathbb {G}}} \newcommand{\BH}{{\mathbb {H}}}

     \newcommand{\BP}{{\mathbb {P}}}
    \newcommand{\BQ}{{\mathbb {Q}}} \newcommand{\BR}{{\mathbb {R}}}
     
     \newcommand{\BV}{{\mathbb {V}}}
     
     \newcommand{\BZ}{{\mathbb {Z}}}

    \newcommand{\cA}{{\mathcal {A}}} 
     
    \newcommand{\cE}{{\mathcal {E}}} \newcommand{\cF}{{\mathcal {F}}}
     \newcommand{\cH}{{\mathcal {H}}}
     
    \newcommand{\cK}{{\mathcal {K}}} 
    \newcommand{\cM}{{\mathcal {M}}} \newcommand{\cN}{{\mathcal {N}}}
    \newcommand{\cO}{{\mathcal {O}}} \newcommand{\cP}{{\mathcal {P}}}
     
    \newcommand{\cS}{{\mathcal {S}}} 
     \newcommand{\cV}{{\mathcal {V}}}
     \newcommand{\cX}{{\mathcal {X}}}
    \newcommand{\cY}{{\mathcal {Y}}} \newcommand{\cZ}{{\mathcal {Z}}}

    \newcommand{\RM}{{\mathrm {M}}}

    \newcommand{\fa}{{\mathfrak{a}}}

    \newcommand{\wt}{\widetilde}\newcommand{\ol}{\overline}
    \newcommand{\wh}{\widehat}

    \newcommand{\pair}[1]{\langle {#1} \rangle}

    \newcommand{\incl}{\hookrightarrow}

    \newcommand{\bsl}{\backslash}

 \newcommand{\vep}{\varepsilon} \newcommand{\ep}{\epsilon}
  \newcommand{\vpl}{\varprojlim}
             
 \newcommand{\vil}{\varinjlim}  
 \newcommand{\lb}{\left(} \newcommand{\rb}{\right)}

 \newcommand{\ab}{{\mathrm{ab}}}
    \newcommand{\ad}{{\mathrm{ad}}}

\newcommand{\homology}{{\mathrm{hom}}}
    
    \newcommand{\Ch}{{\mathrm{Ch}}}

\newcommand{\coh}{{\mathrm{coh}}}

    \newcommand{\End}{{\mathrm{End}}}

    \newcommand{\Frac}{{\mathrm{Frac}}}

    \newcommand{\Gal}{{\mathrm{Gal}}} \newcommand{\GL}{{\mathrm{GL}}}

\newcommand{\rO}{{\mathrm{O}}}
\newcommand{\GO}{{\mathrm{GO}}}

    \newcommand{\Hom}{{\mathrm{Hom}}}
    
    \newcommand{\id}{{\mathrm{id}}}

    \newcommand{\inv}{{\mathrm{inv}}} \newcommand{\Inv}{{\mathrm{Inv}}}
    \newcommand{\Jac}{{\mathrm{Jac}}}

    \newcommand{\ord}{{\mathrm{ord}}} 
     \newcommand{\Pic}{\mathrm{Pic}}

   \newcommand{\bc}{{\mathrm{bc}}}

\newcommand{\nspl}{{\mathrm{nspl}}}

\newcommand{\Nm}{{\mathrm{Nm}}}
        \renewcommand{\Re}{{\mathrm{Re}}}

           \newcommand{\reg}{{\mathrm{reg}}}

    \newcommand{\supp}{{\mathrm{supp}}}

         \newcommand{\SL}{{\mathrm{SL}}}
    \newcommand{\Spec}{{\mathrm{Spec}}}     
      
    \newcommand{\SO}{{\mathrm{SO}}}
    
    \newcommand{\sgn}{{\mathrm{sgn}}}

    \newcommand{\ur}{{\mathrm{ur}}}  
    \newcommand{\Vol}{{\mathrm{Vol}}}

    \newcommand{\Tr}{{\mathrm{Tr}}}

\newcommand\supervisor[1]{\def\@supervisor{#1}}

\newcounter{elno}

\renewcommand{\cong}{\simeq}

 \author{Congling Qiu} 
 
\begin{document} 
\address{Department of Mathematics, Yale University, New Haven CT 06511, United States}
\email{congling.qiu@yale.edu}
\title{Modularity and Heights of CM cycles on Kuga--Sato varieties}

\maketitle 
\begin{abstract}
We  prove  a  higher weight general Gross--Zagier formula  for CM cycles on Kuga--Sato varieties over   modular curves of arbitrary levels.
To formulate and prove this result,  we  prove several results on the  modularity  of  CM cycles, in the sense that the Hecke modules they generate are semisimple modules
 whose  irreducible components are associated to  higher weight holomorphic cuspidal automorphic representations  of $\GL_{2,\BQ}$.
 These two types of results provide evidence toward two conjectures of Beilinson--Bloch.
The higher weight general 
  Gross--Zagier formula is proved using arithmetic relative trace formulas. The proof of the modularity  of  CM cycles is  inspired by  arithmetic theta lifting.

   \end{abstract}

 \tableofcontents   
 
 \addtocontents{toc}{\protect\setcounter{tocdepth}{1}}

   \section{Introduction}
   \subsection{Gross--Zagier formula and generalizations}

      Let $K$ be an  imaginary quadratic field.  On  modular curves,  there are  explicit algebraic points representing  elliptic curves with CM, i.e. complex multiplication, by $K$.   
They  were studied by  Heegner  in the  case of the modular curve $X_0(N)$   for some $N$, and  thus called  Heegner  points in this case.     
   In \cite{GZ}, Gross and Zagier established a remarkable  formula  
    that relates N\'eron--Tate  heights of   Heegner points  and     central derivatives of     base change  $L$-functions associated to holomorphic modular forms of weight 2 for  the modular curve $X_0(N)$. 
 Explicitly, 
for  a holomorphic   newform $f$ of weight 2, let $J_0(N)_f$ be the $f$-isotypical component ($f$-component for short) of the Jacobian of  the modular curve  $X_0(N)$.
For 
 a degree-0 divisor  $P$ of Heegner points  modified by the cusps, the Gross--Zagier formula says that 
 under  the Heegner condition on $K$ and $N$, 
the $f$-component $P_f\in J_0(N)_f$ 
has height
$$\pair{P_f,P_f}=c \cdot L'(f,K,1)$$
for an explicit nonzero constant $c$.
Since the first day of its establishment, the  Gross--Zagier formula  has been the  best evidence toward the Birch--Swinnerton-Dyer conjecture.

The Gross--Zagier formula has been generalized in many aspects. In the content of modular curves,  the two definite aspects are the  level aspect and  the weight aspect. We summarize some of  the relevant results in the following table, and explain them in the next paragraph. 
 \begin{table}[thb]
\centering
\caption{}
    \label{tab:my_label1}
\begin{tabular}{|l|c|c|}\hline
\diaghead{  \ \ \ \ \ \  \    \ \ \ \ \ \  \     }%
{Weight}{Level}&
 {$X_0(N)$ with Heegner condition}& {Arbitary modular curves}\\    \hline
2& Gross and Zagier (1986) & Yuan, S. Zhang, and W. Zhang (2013)\\    \hline

$2k>2$& S. Zhang (1997)&   \\    \hline
\end{tabular} \end{table}

  In the weight 2 case, Yuan, S. Zhang and W. Zhang \cite{YZZ} proved the   Gross--Zagier formula  at arbitrary levels. (In fact, they proved it even for
    all quaternionic  Shimura curves.) Moreover, their formula is in terms of general CM points and automorphic  forms, and thus deserves to be called the  general Gross--Zagier formula.
      It was   built on previous works of S. Zhang \cite{Zha1,Zha01}.
  In the higher  weight $2k>2$, 
 S. Zhang 
  \cite{Zha97} proved a  Gross--Zagier  type  formula for  
Heegner cycles  on Kuga--Sato varieties over $X_0(N)$ and     modular forms of weight $2k$. 
The purpose of this paper is to fill in the blank in   Table \ref{tab:my_label1}.

\subsection{Results} 
The main result  (Theorem \ref{stronghtderthm}) of this paper is  a  higher weight    analog of the 
 general Gross--Zagier formula    on Kuga--Sato varieties over   modular curves.      It serves as evidence toward the conjecture of Beilinson and Bloch \cite{Bei,Blo} that  is a higher dimensional generalization of the Birch--Swinnerton-Dyer conjecture. 
 
 The formulation of this main result is   automorphic 
 and  aligns  to  the one of  Yuan, S. Zhang and W. Zhang \cite{YZZ}.
To formulate   this result, we need two modularity results (Theorem \ref{strongmodularity} and Theorem \ref{strongermodularity}) for  the Hecke modules of CM cycles. 
 In this introduction, for simplicity, we only state  one modularity result (Theorem \ref{strongmodularityintro}) which is enough to formulate a weaker version (Theorem \ref{mainint}) of our main result. The formulation of this weaker version is closer to the original Gross--Zagier formula.

Let $X$ be the modular curve $X(N)$  over $\BQ$ where $N\geq 3$ so that  $X(N)$   is a fine moduli space. 
Let   $k\geq 2$  be  an  integer.  
The Kuga--Sato variety $ Y $ is   the canonical desingularization  of the $(2k-2)$-tuple product of  the universal generalized elliptic curve over $X$ (see \cite[2.2]{Zha97}).
For an elliptic curve $A$   with CM by $K$ that represents a CM point on $X$,   identify 
$A^{2k-2}$ as the fiber of $Y$ over this CM point and  define the CM cycle $Z(A)$   on $A^{2k-2}\subset Y$ as follows. 
   Fix an embedding of  $K$ to $\BC$.  
        Choose   a positive integer $D$ such that $\sqrt{-D}\in K$ and is an endomorphism  of $A$. 
          Let $\Gamma$ be the graph of   $\sqrt{-D} $. Let $\Gamma_0$ be the divisor $$\Gamma-A\times \{0\}-D\{0\}\times A$$ on $A\times A$.
Then $\Gamma_0^{k-1}$ is a $(k-1)$-cycle on  $A^{2k-2}$.  
 Define the CM cycle on $A^{2k-2}$ to be $$Z(A)=c\sum_{\sigma }\sgn(\sigma) \sigma^*\lb \Gamma_0^{k-1}\rb   , $$
where $\sigma$ runs over the symmetric group of $2k-2$ elements which acts on $A^{2k-2}$ naturally, and 
$c$ is a positive constant such that the self intersection  number of $Z(A) $ in $A^{2k-2} $ is $(-1)^{k-1}$. 
 Then the Chow cycle $[Z(A)]$ of $Z(A) $ does not depend on the choice of $D$ (see \cite[Proposition 2.4.1]{Zha97}). We understand $Z(A)$  as a cycle on $Y$ via $A^{2k-2}\subset Y$. Recall that
 CM points (and thus CM cycles)  are defined over $K^{\ab}$ (see \cite[3.1.2]{YZZ}).   Define
     $$[Z]=\int_{\Gal(K^\ab/K) } [Z(A)]^\tau   d\tau,$$
        where  
  $\Gal(K^\ab/K)$ is endowed with the Haar measure of  total volume 1. In fact, we  will allow  the twist of $[Z]$ by  a finite order  character of $\Gal(K^\ab/K)$ and use an equivalent adelic definition (Remark \ref{al(K}). 
  
  Let $S$ be a finite set  of prime numbers  containing all prime factors of $N$. Let  $\cH^S$ be the  unramified Hecke algebra outside $S$, which acts on the Chow groups of $Y$ via Hecke correpondences (see \ref{ion on cohomological Chow cycles}).      Let      ${CM}^S $    be   the   $\cH^S$-module generated by $[Z]$.
    Let  $\ol{CM} $    be the quotient of   ${CM} $         by the 
kernel of the   Beilinson--Bloch height pairing on $CM$ defined in \S \ref{ and the height pairing}.
 
       \begin{thm}[Theorem \ref{strongmodularity}]\label {strongmodularityintro}
        Assume that $S$  has cardinality at least three, and 
        contains the prime  2  and all finite places of $\BQ$   ramified in $K$.  
      The  $\cH^S$-module $\ol{CM} $   is  semi-simple   whose  irreducible components are the $\cH^S$-modules associated to   weight  $2k$ holomorphic cuspidal automorphic representations  of $\GL_{2,\BQ}$ with trivial central character.

    \end{thm}

This theorem  is predicted by the conjecture of Beilinson and Bloch on the injectivity of the Abel--Jacobi map, see Remark \ref{modconjrmk}.

Let $\pi$ be a  weight  $2k$ holomorphic cuspidal automorphic representation  of $\GL_{2,\BQ}$ with trivial central character.
Then the $\pi$-component of $[Z]$ is well-defined by Theorem \ref{strongmodularityintro}, and denoted by $[Z]_\pi$. 
   \begin{thm}      \label{mainint}
  Let $\pi$ be generated by modular forms of weight $2k$ on $X$\footnote{Equivalently, 
 $\pi$ has nonzero invariant vectors by the adelic principal congruence subgroup of $\GL_{2,\BQ}$ of level $N$.}.  Assume $L(1/2,\pi_K)=0$.   
 There is a nonzero constant  $c_\pi$ such that 
 $$\pair{[Z]_\pi,[Z]_\pi}=  c_\pi \cdot L' (1/2 ,\pi_K). $$                      \end{thm}   
                   
               This is a special case of Corollary \ref    {uncond}  of our  main result  (Theorem \ref{stronghtderthm}).
                In  S. Zhang \cite{Zha97} over $X_0(N)$,  
                   the assumption   $L(1/2,\pi_K)=0$ is implicit by the Heegner condition, which implies that the root number  of $L(s,\pi_K)$ is $-1$. 

We also remark that the definition of Kuga--Sato varieties, CM cycles,  Theorem \ref{strongmodularityintro} and Theorem \ref{mainint} apply to any 
fine modular curve $X$ (for Theorem \ref{strongmodularityintro}, 
  assume that $X$ has maximal level structures outside $S$). Indeed, we can dominate $X$ by some $X(N)$  with $S$ containing all prime factors of $N$. The results for  $X(N)$ imply the ones for $X$.    If $X$ is not a fine moduli space,  one may follow \cite[4.1]{Zha97}.

                                \subsection{Legacy from previous works} 
Our   proof of Theorem \ref{strongmodularityintro} is strongly influenced by the  previous works in    Table \ref{tab:my_label1}.

A key step in Gross and Zagier  \cite{GZ} for $X_0(N)$
 is the following   explicit  \textit{modularity of generating series of heights}. 
 Let $T_m$ be  the standard $m$-th Hecke operator in $\Pic\lb X_0(N)^2\rb
  $,  as a correspondence from $X_0(N)$ to $X_0(N)$. If   $m$ is coprime to $N$, the N\'eron-Tate height pairing 
$$\pair{P, T_m P}$$ is the $m$-th coefficient of an explicit  holomorphic  modular form of weight 2.

 In weight $2k>2$,   S. Zhang \cite{Zha97} followed Gross and Zagier's approach, and  proved the  modularity of generating series of heights with $Z$ replacing $P$. 
  A consequence of this result is an analog of Theorem \ref{strongmodularityintro} in the setting of loc. cit.. 
   We shall call Theorem \ref{strongmodularityintro} the  \textit{modularity  of Hecke modules of Heegner/CM cycles}. Analogous modularity  for Heegner/CM points follows from the Abel--Jacobi theorem for divisors on curves for free.

 We establish  a modularity  result of generating series of heights  \eqref{concl}. 
 However, our result is weak from two aspects. First, our result is  only for some ``regular" Hecke operators (see Assumption \ref{asmp2'}). 
 This can be remedied, as  S. Zhang computed some  complement Hecke operators.  
 Second,  our relevant modular forms, more precisely, automorphic forms,   live on an adelic group with many connected components.  Our modularity  result  is only on one component (see  Remark \ref{clear}).  
 The same issue appears in Yuan, S. Zhang and W. Zhang \cite{YZZ}. They overcame it using another  modularity result, proved in a separate work of them \cite{YZZ1}. However, we do not  (need to) overcome this issue in our work.

Before we   explain how to 
prove the modularity  of Hecke modules   from the modularity  of
 generating series, we summarize some key words in the proofs of the above modularity results.
    \begin{table}[thb]
\centering
\caption{}
    \label{tab:my_label2}
\begin{tabular}{|l|c|c|}\hline
Modularity of &
 {Generating series}& {Hecke module}\\    \hline
Gross--Zagier (1986) & Explicit computation & Abel--Jacobi theorem\\    \hline
YZZ  (2013)& Another modularity in YZZ  (2009) &Abel--Jacobi theorem\\    \hline
S. Zhang (1997) &Explicit computation  &  Jacquet--Langalnds--Shimizu \\    \hline
Current work &Weak version   &   Numerical  interpretation\\    \hline
\end{tabular} \end{table}

 From a representation theoretic viewpoint, S. Zhang's   proof  is essentially an arithmetic  version of 
  the Jacquet--Langalnds--Shimizu correspondence from $\GL_2$ to $\GL_2$ itself \footnote{
  The general Jacquet--Langalnds--Shimizu correspondence is   from  the unit group of a quaternion algebra to $\GL_2$.
  In this setting, the quaternion algebra the matrix algebra and the unit group is also $\GL_2$!  For us, we will  use a quaternionic notation in our proof in \ref{remove}.}.  
By only having a weak version, we can not establish  such a correspondence. However,
we can interpret the modularity  of Hecke modules of CM cycles in terms of the vanishing of some height pairings (Proposition \ref{strongmodularity1}). 
 This numerical statement can be tackled by our weak version of the modularity  of generating series of heights.

 Finally, we remark that the approach of  Gross and Zagier has now been further developed to the theory of
  arithmetic theta lifting  by Kudla \cite{Kud}.  
It is worth mentioning that recently Li and Liu \cite{LL1}\cite{LL2} used     arithmetic theta lifting to prove an analog of  the Gross--Zagier   formula  for higher dimensional unitary Shimura varieties.

      \subsection{Relative trace formulas}

To prove  the higher weight    analog of the 
 general Gross--Zagier formula,     instead of   arithmetic theta lifting (as in Yuan, S. Zhang and W. Zhang \cite{YZZ}), we   use   the   arithmetic variants of
   Jacquet's  relative trace formulas \cite{Jac87,JN}.  In this arithmetic   relative trace formula approach, 
  another modularity result (Theorem \ref{strongermodularity}) is  also vital. This modularity result is for the action of the full Hecke algebra on CM cycles.

The arithmetic   relative trace formula approach  was proposed by W. Zhang \cite{Zha12} to attack the arithmetic Gan--Gross--Prasad conjecture, which is  an exact generalization of the Gross--Zagier formula to unitary Shimura varieties.   Subsequently,  we used this method to prove the general Gross--Zagier formula over function fields \cite{Qiu}.       Moreover,   Yun and Zhang \cite{YZ1,YZ2} used  a
  geometric  version  of the arithmetic   relative trace formula approach to prove   an analog of  the Gross--Zagier   formula  for higher order derivatives.

 
 \subsection{Generalizations} 
 The   analog  of S. Zhang's work  for general Shimura curves has not been established yet. 
 This is the main  obstruction to the generalizations of our results to Shimura curves.
 A key ingredient to  the   analog  of S. Zhang's work  has been established by
    Wen \cite{Wen}. 
  We hope to return to   Shimura curves   in the  near future.

  \subsection{Structure of the paper}
  In \S \ref{FixN}, we  first define   Kuga--Sato varieties, Hecke operators and   CM cycles.
Then we state  our  modularity  results for the Hecke action on CM cycles   in \S \ref{Modularitysec},
 and  our higher weight    analog of the 
 general Gross--Zagier formula in  \S \ref {Height and derivatives}. 
 
  In \S \ref {The automorphic distributions},
we  first review the    relative trace formulas and related backgrounds.
Then we  establish  an arithmetic relative trace identity 
 in \S \ref {Height distribution}.  Finally in \S \ref{proofs},
we 
 prove the higher weight    
 general Gross--Zagier formula, providing  our  modularity  results.

  In \S \ref{Weak},   we first recall some basics of theta series and Eisenstein series. Then we prove an arithmetic mixed Siegel--Weil formula in 
  \S\ref{Local comparison II}. 
   In \S \ref{remove}-\S\ref{removeer}, we prove the modularity 
    of CM cycles.   
      \subsection{Some notations and conventions}\label{nots}

A local field $F$ is assumed to be equipped with  an absolute value $|\cdot|_F$.
Let $|\cdot|_p=  |\cdot|_{\BQ_p}$ be that assigns $1/p$ to $p\in \BQ_p$.  On $\BR$, we use the usual absolute value. On $\BC$, the absolute value  is the square of the usual one.
 Let $\BA$ be the ring of adeles of $\BQ$.
 The symbol $|a|$ for $a\in \BA$ means the product of all local absolute values of $a$.

For a set of places  $S$ of $\BQ$  and a decomposable adelic object $X$  
over $\BA$, we use $X_S$ (resp. $X^S$) to denote the $S$-component (resp. component away from $S$) of $X$ if the decomposition of $X$ into the product of  $X_S$  and $X^S$ is clear from the context. 
       If $S=\{v\}$, we write $X_v$ (resp. $X^v$) for $X_S$ (resp. $X^S$). For example,  $\BA^{\infty}$  is the  ring of finite adeles.

  We use $\cS_c(X)$ resp. $\cS(X)$ to denote the space of compactly supported smooth  functions resp. Schwartz functions on a reasonable space $X$.
  We will only use the following  cases. 
     If $X$ is a real manifold, we will only use $\cS_c(X)$, except when $X=\BR^n$, in which case the definition  of  $\cS(X)$ is classic. 
   If $X$ is locally profinite, $\cS_c(X)=\cS(X)$ and is the space of compactly supported locally constant functions.  
 
 For a group $G$ and a function $f$ on $G$, let $$f^\vee(g):=f(g^{-1}).$$
If $G$ is  a  locally profinite group, 
    the convolution on $\cS_c(G)=\cS(G)$ is $$f* f_1(h)=\int_{G}f(g)f_1(g^{-1} h) dg.$$

      \subsection*{Acknowledgements}     
      The author  thanks  Yifeng Liu, Xinyi Yuan and     Shouwu Zhang for their  help. 
      The author is grateful to the referees for their  careful   reading of the earlier versions of the paper and their valuable comments.
       The research   is partially supported by the NSF grant DMS-2000533. 

\addtocontents{toc}{\protect\setcounter{tocdepth}{2}}
  \section{CM cycles   on Kuga--Sato varieties} \label{FixN}
  
  We first define   Kuga--Sato varieties and Hecke operators on Chow cycles on Kuga--Sato varieties. Then define CM cycles as well as their height pairings.
Finally we  state our theorems on the modularity of CM cycles, and the higher weight general 
  Gross--Zagier formula.

  \subsection{Kuga--Sato varieties and Hecke operators}  
   \subsubsection{Kuga--Sato varieties}
    For a positive integer $N$, 
let  $X(N)^\circ$ be the moduli stack of  elliptic curves $A$ over  $\BQ$-schemes with  
full level $N$ structure $ (\BZ/N\BZ)^2\cong A[N]$. 
For simplicity, let $N\geq 3$ so that $X(N)^\circ$ is representable 
by a  smooth curve over $\BQ$, which is connected  but not geometrically connected. 
It might be helpful to remind the reader the complex uniformization  of the  $\BC$-points of $X(N)^\circ$ as a Riemann surface:
\begin{equation}
\label{cplx}
 X(N)^\circ(\BC)\cong \GL_2(\BQ)_{>0} \bsl  \BH\times \GL_2 (\BA^\infty)^\times /U(N) .
 \end{equation}
 Here    $\GL_2(\BQ)_{>0} \subset \GL_2(\BQ)$ consists of elements with positive  determinants, $ \BH\subset \BC$ is the upper half plane and is equipped with the action of $\GL_2(\BR)_{>0} $ by fractional linear transformation,
   and  $U(N)=1+N\RM_2(\wh \BZ)$ is the corresponding principal congruence subgroup of $\GL_2(\BA^\infty)$ of level $N$, where  $\wh \BZ\subset \BA^{\infty}$  is the ring of integral adeles.

The modular curve $X(N)$ be the smooth
 compactification  of $X(N)^\circ$, which is the moduli space of generalized elliptic curves. 
 Let $E\to X(N)$ be  the universal generalized elliptic curve 
so that $E|_{X^\circ}\to  {X^\circ}$ is  the universal elliptic curve. 

Let    $k\geq 2$ be  an  integer. Let $Y(N) ^\circ=(E|_{X(N)^\circ})^{2k-2}$, the $(2k-2)$-tuple product of $E|_{X(N)^\circ}$  over $X(N)^\circ$.
The Kuga--Sato variety $ Y(N)$ is   the canonical desingularization  of  $E^{2k-2}$ \cite[2.2]{Zha97}. 
Indeed, the
 desingularization  procedure  only happens at the cusps  so that $ Y (N)^\circ= Y(N) |_{X(N)^\circ}$.

Below, we will use some results in \cite{Zha97}, where the author takes a geometrically connected component of $X(N)$ (while using the notation $X(N)$ in loc. cit.). The discussion there can be carried on to our setting  accordingly.

   \subsubsection{Group and semi-group action}  \label{Hecke action}

      We will define Hecke operators on Chow cycles by mimicking the Hecke operators in the setting of automorphic  forms, rather than  classical modular forms. Thus we need an adelic group action. It is defined on some infinite level modular curves, and related Kuga--Sato varieties.
    
   Let $\BG \subset  \GL_2(\BA^\infty)$ be the subset of elements with integral coefficients. 
      Let  $$X_\infty^\circ=\vpl_{N} X(N)^\circ,\ Y_\infty^\circ =\vpl_{N} Y(N)^\circ, $$
   where the transition morphisms are finite flat.
  We define a   left $\GL_2(\BA^\infty)  $-action on $X_\infty^\circ$ and  a  left $\BG  $-action on $Y_\infty^\circ$
  following Drinfeld  \cite[5.D]{DriEll1}.

  For  $g\in \BG$,  $g$ defines an endomorphism of $(\BQ/\BZ)^2$. 
   Let $S$ be a $\BQ$-scheme and $A$  an elliptic curve over $S$ together with an infinite level structure $\psi:(\BQ/\BZ)^2\to A(S)$, i.e. for every positive integer $n $, the restriction of $\phi$ to $\BZ[\frac{1}{n}]$ is a (Drinfeld) $\Gamma(n)$-structure \cite[3.1]{KM}. This gives an $S$-point of $X_\infty^\circ$.
    Let $H\subset A$ be the sum of $\psi(\alpha)$'s with $\alpha$ running over the kernel of the endomorphism  $g$ on  $(\BQ/\BZ)^2$, which is a subgroup scheme of $A$ \cite[Proposition 1.11.2]{KM}.
Define $\psi_1$ so that the following diagram  commutes:
\begin{equation}\label{gisog}     \xymatrix{ 
 & \ar[d] ^{g}  \ar[r]^{\psi} (\BQ/\BZ)^2  &   A(S)\ar[d]    \\
 &  (\BQ/\BZ)^2\ar[r]^{\psi_1}  & A/H(S)   }.
\end{equation}
   Here the right vertical arrow is the natural quotient morphism.
    
    Let   $E_\infty$ be  the universal elliptic curve over $X_\infty^\circ$ with an infinite level structure. Applying the above construction, we get a quotient 
    elliptic curve $E'_\infty$ over $X_\infty^\circ$ with an infinite level structure, and 
    a quotient 
    map $E_\infty\to E'_\infty$ over $X_\infty^\circ$. The elliptic curve $E'_\infty$ over $X_\infty^\circ$ is the pullback of $E$ via  a  morphism $ X_\infty^\circ\to X_\infty^\circ$, which we define   as the    
     action of $g$ on $X_\infty^\circ$.  For  $g \in\BG\cap\BQ^\times $,  the action is trivial.
Thus we have an action of  $ \GL_2(\BA^\infty)$  on $X_\infty^\circ$ by  letting $\BQ^\times$ act trivially.

\begin{rmk} \label{opposite}  
The complex uniformization   \eqref{cplx} gives
$$
 X_{\infty}^\circ(\BC)\cong \GL_2(\BQ)_{>0} \bsl  \BH\times \GL_2 (\BA^\infty)^\times.
$$
Then $g$ acts as
the  ``right multiplication"   by  $g^{-1}  $ (compare with  \cite[3.1.1]{YZZ}).
 \end{rmk}     

The composition of the above morphisms $E_\infty\to E'_\infty\to E_\infty$ defines 
 a morphism $g: E_\infty\to E_\infty$ for $g\in \BG$.   
 However,  the action of $g \in\BG\cap\BQ^\times $ is not trivial. Indeed, if $g\in \BZ$, it acts trivially on $X_\infty^\circ$, and  by multiplying  by $g$ on each fiberal elliptic curves. 
 Since   $Y_\infty^\circ=E_\infty^{2k-2}$,   we have an action of $\BG$ on $Y_\infty^\circ$. Again, the action of $g \in\BG\cap\BQ^\times $ is not trivial. 
 
    \subsubsection{Hecke action on cohomological Chow cycles}  \label{ion on cohomological Chow cycles} 
Let $\Ch^k(Y(N))$ be   the spaces of codimension $k$ Chow cycles on $Y(N)_{\ol\BQ}$ with $\BC$-coefficients. 
We define  a right action of $ \GL_2(\BA^\infty)$ on $$ \vil _{N}\Ch^k(Y(N))$$ via pullback (this is the reason for calling  $ \vil _{N}\Ch^k(Y(N))$ the group of cohomological Chow cycles).

For   given $N$ and $g\in \GL_2(\BA^\infty)$, the above construction defines a morphism   $g_{N',N}:Y(N')^\circ\to Y(N)^\circ$ for a large enough multiple $N'$ of $N$.
Let $\Gamma_{g_{N',N}}$ be the Zariski closure   of the  graph of $g_{N',N}$ in $Y(N') \times Y(N) $.
For $z\in \Ch^k(Y(N))$,  let $g ^* z \in \Ch^k(Y(N'))=\Ch_{k-1}(Y(N'))$ be the pullback of $z$ by $\Gamma_{g_{N',N}}$.
The image of $g^*z$ in $ \vil _{N}\Ch^k(Y(N))$ does not depend on the choice of $N'$.
Define the natural right action of $g$ on $ \vil _{N}\Ch^k(Y(N))$ to be 
  \begin{equation*}R^{\coh}_g(z) = |\det (g)|^{k-1}  g^*z.\label{normalization}\end{equation*}
 Then  $R^{\coh}_g$ is the identity for   $g \in\BG\cap\BQ^\times $.
Let $\BQ^\times$ act trivially on  $\vil _{N}\Ch^k(Y(N))$. Thus we have  a right action  of $\GL_2(\BA^\infty)$ on $\vil _{N}\Ch^k(Y(N))$. 
 However, in view of Remark \ref {opposite} (as well as the conventions for automorphic forms), it is better to use  the opposite left action
    \begin{equation*}L^{\coh}_g  = R^{\coh}_{g^{-1}}.\label{normalization}\end{equation*} 
       The induced action of          $f\in  \cS(\GL_{2}(\BA^\infty))$ on $\vil _{N}\Ch^k(Y(N))$
         is $$L^{\coh}_f: =\int_{g\in\GL_{2}(\BA^\infty) }f(g)L^{\coh}_g  dg.$$
 
      \begin{rmk}  \label{yongshangl}                
A more well-known definition of the Hecke action  (in the unramified case) is as follows \cite[p 122]{Zha97}.
                                        For  $i\in \BZ_{\geq 0}$,   let 
                                $A_{p^i}$ be the characteristic function of $\GL_2(\BZ_p) \begin{bmatrix} p^i&0\\
 0&1\end{bmatrix}\GL_2(\BZ_p)$.      For a positive integer $n$ coprime to $N$,  let 
                    $A_n=\bigotimes_{p|n}A_{p^{\ord_p n}}$.    
                               Let $Y (N,n)^\circ $   be the noncuspidal locus of the Kuga--Sato variety of elliptic curves with the canonical full level $N$ structure, and  a subgroup of order $n$.  Let $Z ^\circ$ be the quotient of $Y(N,n) ^\circ$ by the  $(2k-2)$-tuple product of the   subgroup of order $n$ of the universal elliptic curve (and remember the full level $N$ structure).  
               Then we have natural morphisms
               $$Y(N)^\circ\leftarrow Y(N,n)^\circ\rightarrow   Z ^\circ\rightarrow  Y(N)^\circ.$$
               Let  $C$ be the  pushforward   of the fundamental cycle of $Y(N,n)^\circ$ to $Y(N) \times Y(N) .$
   Then up to the choice of the measure on $\GL_{2}(\BA^\infty)$,  the Hecke action $R(A_n)$  coincides with the 
               pullback   by    $n^{-(k-1)}C$.

     Let $\xi$ be a cusp  form of $\SL_2(\BZ)$ of level $N$ and weight $2k$, and $\phi_\xi$ the associated cusp form of $\GL_{2,\BQ}$ as in \cite[(3.4)]{Gel}.
           The pullback  of  $\xi$  by $n^{-(k-1)}C$ \cite[1.11]{Kat} and the usual Hecke action of $A_n$ on $\phi_\xi$ are compatible  \cite[Lemma 3.7]{Gel}. 
\end{rmk} 
 \subsubsection{Hecke action on homological Chow cycles}  
 We define   a left action of $\GL_2(\BA^\infty)$ on $$\vpl _{N}\Ch^k(Y(N))$$  via push-forward 
  (this is the reason for calling  $\vpl _{N}\Ch^k(Y(N))$ the group of homological Chow cycles).

  Let $g\in \BG$ act on $ \vpl _{N}\Ch^k(Y(N))$ by 
  \begin{equation*}
  L^{\homology}_g: = |\det (g)|^{k-1}  g_* .
  \end{equation*}  
     Then  $L^{\homology}_g$ is the identity for   $g \in\BG\cap\BQ^\times $.
Let $\BQ^\times$ act trivially on  $\vpl _{N}\Ch^k(Y(N))$.  Thus we have  a left action  of $\GL_2(\BA^\infty)$ on $\vpl _{N}\Ch^k(Y(N))$. 
For $  f\in  \cS(\GL_{2}(\BA^\infty))$, we have the induced action           
 \begin{equation}
  \label{Lhomf}
L^{\homology}_f: =\int_{g\in\GL_{2}(\BA^\infty) }f(g)L^{\homology}_g  dg
  \end{equation} 
on $\vpl _{N}\Ch^k(Y(N))$.

               \subsubsection{Relation}  \label{endof}
    
    Consider the natural inclusion  \begin{align}\label{iota}\iota: \vil _{N}\Ch^k(Y(N))&\to  \vpl _{N}\Ch ^k(Y(N))\\
    \nonumber      z=\{z_N\}_N &\mapsto \{\Vol(U(N) )z_N\}_N).\end{align}   
Then   the restriction of $L^{\homology}_g$  via $\iota$ is $L^{\coh}_{g}$.

          \subsection{CM cycles and height pairings} 
   
           \subsubsection{CM cycles}\label{CM cycles}
      
          Let $K$ be an imaginary quadratic field with a fixed embedding to $\BC$. For an elliptic curve $A$  over an integral domain $R$ of  characteristic 0 with CM by $K$,  define the CM cycle $Z(A)$  of $A$ as follows.
        Choose   a positive integer $D$ such that $\sqrt{-D}\in K$ and is an endomorphism  of $A$. 
                  Let $\Gamma$ be the graph of   $\sqrt{-D} $. Let $\Gamma_0$ be the divisor $$\Gamma-A\times \{0\}-D\{0\}\times A$$ of $A\times A$.
Then $\Gamma_0^{k-1}$ is a $(k-1)$-cycle on  $A^{2k-2}$. Define the CM cycle on $A^{2k-2}$ to be $$Z(A)=c\sum_{\sigma }\sgn(\sigma) \sigma^*\lb \Gamma_0^{k-1}\rb   , $$
where $\sigma$ runs over the symmetric group of $2k-2$ elements which acts on $A^{2k-2}$ naturally, and 
$c$ is a positive constant such that the self intersection  number of $Z(A)_{\Frac R}$ in $A^{2k-2}_{\Frac R}$ is $(-1)^{k-1}$.
 Then the Chow cycle  of $Z(A)_{\Frac R}$ does not depend on the choice of $D$ (see \cite[Proposition 2.4.1]{Zha97}).

         Let $\GL(\BA^\infty)$ act on $X^\circ_\infty$ as in \S \ref {Hecke action}.   
        Fix an embedding $K\incl   \RM_{2,\BQ}$ of $\BQ$-algebras.
   So we have      an embedding  $\BA_K^{\infty,\times} \incl \GL(\BA^\infty)$.   
                 Fix a point $P_0\in X^\circ_\infty(\ol \BQ)^{K^\times}$, which is a CM point.             For $g\in  \GL_2(\BA^\infty)$, let $P_g$ be the image of $g^{-1}\cdot P_0  \in X_\infty$  to  $X(N)$ by $X_\infty\to X(N)$.  Then $P_g$ is a CM point.
       Let $A=E_{\ol \BQ}|_{P_g}$ so that $Y(N)_{\ol \BQ}|_{P_g}=A^{2k-2}$.
                 Let the cycle $Z_g$  on $Y(N)_{\ol \BQ}$ be the pushforward of the  CM cycle  $Z(A)$ on $A^{2k-2}$ under $A^{2k-2}\to Y(N)_{\ol \BQ}$ .
                 Then $[Z_g]$ does not depend on the choice of $D$, and 
               is cohomologically trivial \cite[3.1]{Zha97}. 

       Let $\Omega$ be a   character of $K^\times\bsl \BA_K^\times$ trivial on $ K_\infty^\times$. 
Define    the $\Omega$-isotypic CM cycle to be 
          $$Z_\Omega=\int_{K^\times\bsl \BA_K^{\infty,\times } } Z_t \Omega(t) dt,$$
          where we endow $K^\times\bsl \BA_K^{\infty,\times }$ with the Haar measure such that   
           \begin{equation}
\label{VolK}
      \Vol(K^\times\bsl \BA_K^{\infty,\times })=1.
\end{equation}

Let $Z^k(Y(N))$  be the space of codimension $k$ cycles on $Y(N)_{\ol\BQ}$ with $\BC$-coefficients.
  As $N$ varies, 
  $Z_{g  },Z_\Omega$ are  compatible under pushforward.
   Thus we consider them as an elements in 
 $\vpl _{N}Z^k(Y(N)).$

For a cycle $Z$, let $[Z]$ be its Chow cycle, so that  $[Z_{g  }],[Z_\Omega]$ are   elements in 
 $\vpl _{N}\Ch^k(Y(N)).$ Then $$L^{\homology}_{h^{-1}}[Z_{g  }]=[Z_{gh  }].$$

 \subsubsection{Hecke translation}
         
        For $f\in \cS(\GL_{2}(\BA^\infty))$    right $U(N)$-invariant, define  the translation  of $Z_{\Omega}$ by $f$  to be 
                                           \begin{align}
\label{ZTZ}
Z_{\Omega,f}=& \frac{1}{\Vol(U(N))}\int_{K^\times\bsl \BA_K^{\infty,\times }  } \int_{  \GL_2(\BA ^\infty)  }  f(g) Z_{tg}  dg   \Omega(t) dt.\\
=& \int_{K^\times\bsl \BA_K^{\infty,\times }  } \sum_{g\in  \GL_2(\BA ^\infty) /U(N)}  f(g) Z_{tg}    \Omega(t) dt
\end{align}
                                               
                                                        As $N$ varies, 
    $Z_{\Omega,f}$ is compatible under pullback.
  Thus we consider  $Z_{\Omega,f}$ as an element in 
 $\vil _{N}Z^k(Y(N)).$           
 \begin{rmk}The   normalization factor  $ \frac{1}{\Vol(U(N))}  $ is 
 to achieve the compatibility as $N$ varies.  The   normalization factor    will be more complicated than  for general Shimura curves. See \cite[(3.4.4)]{YZZ} or \cite[3.1.1]{Qiu} if we pretend $k=1$.
  \end{rmk}
 Now we interpret $Z_{\Omega,f}$ by Hecke action. 
Under \eqref{iota}, we have
 \begin{equation}
\label{ZTZ}
\iota\lb  [Z_{\Omega,f}]\rb=L^{\homology}_{f^\vee} [Z_\Omega ].
\end{equation}

                      The following relations will be useful later:  for $g\in  \GL_2(\BA ^\infty) $, \begin{equation}\label{ZTZ00}L^{\coh}_g [ Z_{\Omega,f} ] =[Z_{\Omega, f^g}]  ,\end{equation} 
where $f^g(h)=f(hg )$, and  
                        \begin{equation}\label{ZTZ0}
                        L^{\coh}_{f_1} [ Z_{\Omega,f} ]=  [ Z_{\Omega,f*f_1^\vee}].
                        \end{equation}

                  \begin{rmk}\label{al(K}
                  CM points (and thus CM cycles)  are defined over $K^{\ab}$ (see \cite[3.1.2]{YZZ}).   Regarding $\Omega$ as a character of $\Gal(K^\ab/K)$ by the class field theory, then we can redefine
          $$[Z_\Omega]=\int_{\Gal(K^\ab/K) } [Z_1]^\tau  \Omega(\tau) d\tau$$
          (and           $ [Z_{\Omega,f}]$  similarly),
          where  
  $\Gal(K^\ab/K)$ is endowed with the Haar measure of  total volume 1.
\end{rmk}

 \subsubsection{Integral models}\label{Kintmodel}

 To define height pairings, we need integral models of Kuga--Sato varieties defined as follows.
 Let $N$ be  the product of two relatively prime integers which are $\geq 3$ until \ref{Height pairing}.
 
Let the morphism $\pi:\cE\to \cX$ of    regular, flat, and projective $\BZ $-schemes 
be as in  \cite[Theorem 2.1.1]{Zha97},
such that 
  $\cX_{\BQ }\cong X(N) $   and 
  $\cE$ is a generalized elliptic curve over  $\cX$ whose generic fiber is $E $.
   Let  $\cY$ be  the canonical desingularization  of  $\cE^{2k-2}$  over $\cX$ \cite[2.2]{Zha97}.
The desingularization  procedure  only happens at the cusps in the following sense.
 By the cusps of $\cX$, we mean   the Zariski closure of 
 $X(N)\bsl X(N)^\circ$, the closed subscheme of
 the cusps on the generic fiber. 
  Let $\cX^\circ$ be the noncuspidal locus of $\cX $   and $\cY ^\circ=\cY |_{\cX^\circ}$.
  Then $\cE|_{\cX^\circ}\to  {\cX^\circ}$ is smooth and  $\cY ^\circ=(\cE|_{\cX^\circ})^{2k-2}$.

For    a number field  $F$, we construct  regular  proper models $\cX'$ and $\cY'$ of $X_F$ and $Y_F$ over $ \Spec \cO_F$ as follows.
Let $\cX'$ be the minimal desingularization of  $\cX_{\cO_F}$. 
Let $\cN$ be a neighborhood  of the cusps of $\cX$ which is smooth over $\BZ$ \cite[Theorem 8.6.8]{KM}. 
Then   
\begin{equation}
\label{nsloc}\cX'|_{\cN}=\cN_{\cO_F} .
\end{equation}
Let $\cE'=\cE\times_ \cX\cX'$. 
As the non-smooth locus of  $\cE\to \cX$  is supported at  the cusps, it is contained in $\cE|_{\cN}$.
So by \eqref{nsloc},  the non-smooth locus of  $\cE'\to \cX'$  is contained in $\lb \cE|_{\cN} \rb_{\cO_F}\to \cN_{\cO_F}$.
Then  the explicit desingularization  procedure  in \cite[2.2]{Zha97} still applies  to   $\cE'^{2k-2}$. The desingularization   $\cY'$
 satisfies  $\cY'_{F}\cong Y(N)_F$ and  comes with  a natural morphism $\cY'\to \cY(N)_{\cO_F}$. 
On the non-cuspidal locus $\cX'^\circ=\cX'|_{\cX^\circ}$, 
 its preimage  $\cY '^\circ$ in $\cY $ 
 satisfies 
$$(\cY'^\circ\to \cX')=(\cY^\circ\times_{\cX^\circ}\cX'^\circ\to \cX'^\circ)=\lb(\cE|_{\cX^\circ})^{2k-2}\times_{\cX^\circ}\cX'^\circ\to \cX'^\circ\rb.$$  
In particular, 
$\cY'|_{x'}=\cE|_{x}^{2k-2}\times_x x'$ for $x'\in \cX'$ over noncupidal  $x\in \cX$.
 
 In fact,   we only care about $\cX'^\circ$ as our  integral CM   cycles below sit in this locus.
               \subsubsection{Height pairings and  integral CM   cycles} \label{ and the height pairing}

                    We define the Beilinson--Bloch height pairing   \cite{Bei,Blo} using arithmetic intersection pairing \cite{GS}.    We need to define  integral CM   cycles.
                    
                    Consider the CM elliptic curve $E|_{P_g}$ corresponding to the CM point $P_{g}\in X(N)(\ol\BQ)$ that is defined over $F$. Then 
       by CM theory \cite[Theorem 11.2]{Sil} and reduction theory  \cite[Proposition 5.4, 5.5]{Sil} of elliptic curves,     $E|_{P_g}$  has reductions outside the cusps.
         Then the Zariski closure    $\ol P_g$ of the CM point $P_{g}$ in $\cX_{\cO_F}$ is an $\cO_F$-point away from the cusps, and   defines a CM elliptic curve $A=\cE|_{ \ol P_g}$. 
So we have  the CM cycle $Z(A)$ on 
   $\cY_{\cO_F}|_{ \ol P_g}=A^{2k-2}$  (see \S \ref
 {CM cycles}).
   Let   $ \cZ_{g}$ be the pushforward   of  $Z(A)$ to $\cY_{\cO_F}$. Let $ \cZ_{g}'$ be the flat pullback  of $ \cZ_{g}$  to $\cY'$ (defined in \S \ref{Kintmodel}). 
               As the desingularization  procedure in \S \ref{Kintmodel} only happens at the cusps and   $ \cZ_{g}'$ is supported away from the cusps, $ \cZ_{g}'$ is compatible if we vary $F$. More precisely, let $F_1,F_2 $ be two choices of $F$ and let $ \cZ_{g,1}', \cZ_{g,2}' $ be defined accordingly. If $F_2$ is an extension of $F_1$, then $\cZ_{g,2}'$ is the natural pullback of $\cZ_{g,1}'$.

  Let $Z_{g_1},Z_{g_2}$ be  two CM   cycles. By taking      $F$ large enough, we have  two CM   cycles $ \cZ_{g_1}'$ and $ \cZ_{g_2}'$
both defined over $\cY'/\cO_F$.
                 Then the restrictions of $ \cZ_{g_1}'$ and $ \cZ_{g_2}'$
                  to the irreducible components of the special fibers of $\cY'$ are cohomologically trivial \cite[3.1]{Zha97}.
                                 So  the Beilinson--Bloch height pairing   \cite{Bei} 
 is unconditionally defined, and given by  an arithmetic intersection pairing \cite{GS}:
      \begin{equation}\label{BB}\pair{Z_{g_1},Z_{g_2}}_{Y }=\frac{1}{[F:\BQ]}(-1)^k(\cZ_{g_1}',G_{g_1})\cdot (\cZ_{g_2}',G_{g_2}).\end{equation}
      Here $G_{g_i}$ is the ``admissible"  Green current of $Z_{g_i}  $  in  the sense of the first paragraph of  \cite[p 125]{Zha97}.
          Note that if we replace one of $\cZ_{g_1}'$ and $\cZ_{g_2}'$ by its horizontal part, we still get the same height pairing through the arithmetic intersection pairing.
          And $\pair{Z_{g_1},Z_{g_2}}_{Y(N)}$ is independent of choice of $F$ by the projection formula and the compatibility discussed  in the end of the last paragraph\footnote{This is a special case of a more general fact \cite[Lemma 1.5]{Kun}.}. 
Moreover, by  \cite{Bei},   the height pairing between CM cycles factors through their images in        $\Ch^k(Y(N) )$.

                      \subsubsection{Vary the level}  \label{Height pairing}    
                                     Let us first discuss the general case, and then CM cycles.  
                                  We want to define a   height pairing
                                         $$\pair{\cdot,\cdot}:\vil _{N}\Ch^k(Y(N))_0\times \vpl _{N}\Ch^k(Y(N))_0\to \BC.$$ 
                                         Here $\Ch^k(Y(N))_0$ denotes   the subspace of $\Ch^k(Y(N))$ of cohomologically trivial cycles.
                                         For
                           $(x_N)\in \vil _{N}\Ch^k(Y(N))_0$  and $(y_N)\in \vpl _{N}\Ch^k(Y(N))_0$, 
                     by the projection formula for height pairing  \cite{Bei},
                            the   height pairing   $\pair{x_N,y_N}_{Y(N)}$,  
                             if well-defined (this is conjecturally the case) for $N$ large enough, does not depend on the choice of $N$.      Let $$\pair{(x_N),(y_N)}=\pair{x_N,y_N}_{Y(N)}.$$
             This in particular gives the height pairing  between  $x_N,y_N$ for  any $N$.

                           Then by the discussion in \S \ref{ and the height pairing}, 
                      the   height pairing     $\pair{\cdot,\cdot}$ is    well-defined  on the subspaces of CM cycles  in $\vil _{N}\Ch^k(Y(N))_0\times \vpl _{N}\Ch^k(Y(N))_0$. 
                                In particular,  
                        if $f$ is right $U(N)$-invariant, then
                                \begin{align}  \label{NHP}H(f)=\pair{  Z_{\Omega,f }, Z_{\Omega^{-1}}}  = \pair{  Z_{\Omega,f },Z_{\Omega^{-1}}} _{Y(N)}  \end{align}  is defined and 
does not depend on the choice of $N$. 
So   the normalized height pairing  between cohomological cycles \begin{align*}  \pair{  Z_{\Omega,f _1},Z_{\Omega^{-1},f_2}}^\sharp =   \Vol(U(N)) \pair{ Z_{\Omega,f _1},Z_{\Omega^{-1},f_2}} _{Y(N)}:= \pair{ Z_{\Omega,f _1}, \iota\lb Z_{\Omega^{-1},f_2}\rb} _{Y(N)}  \end{align*}  
does not depend on the choice of $N$ (as long as $f$ is right $U(N)$-invariant). 
Moreover,  by  \eqref{ZTZ}, the projection formula, and  \eqref{ZTZ0}, we have   
\begin{equation}\label{ydy} \begin{split} \pair{  Z_{\Omega,f _1},Z_{\Omega^{-1},f_2}}^\sharp&=  \pair{   [Z_{\Omega,f _1}] ,L^{\homology}_{f_2^\vee}[ Z_{\Omega^{-1}}]} _{Y(N)}\\
&= 
 \pair{ L^{\coh}_{f_2} [Z_{\Omega,f_1}], [Z_{\Omega^{-1}}]}\\
 &= 
 \pair{   [Z_{\Omega,f_1*f_2^\vee}], [Z_{\Omega^{-1}}]}. \end{split}  \end{equation}

    \begin{lem}\label{lem:eta}
    Let $\eta$  be the Hecke character of $\BA^\times $ associated to $K$ by the class field theory, and consider it as an automorphic function on $\GL_2(\BA)$ by composition with $\det$.
Assume that $\eta $ and $f$ are right $U(N)$-invariant. 
Then $ H(f)= H(f\eta)$.

       \end{lem}
       
    \begin{proof} 
    
By definition, $    \pair{Z_{g_1},Z_{g_2}}_{Y(N)}\neq0$ only if they are on the same geometrically connected components
of $Y(N)$, equivalently, $\det(g_1)\in \det(g_2) F^\times \det(U(N))$.
Thus, in the definition \eqref{Lhomf} and \eqref{ZTZ} of $
Z_{\Omega,f}$, only $g\in F^\times \det(U(N)) \Nm(\BA_E^\times)$ contributes to 
$ H(f)$. On these $g$'s, $\eta(g)=1$ by definition. So $ H(f)= H(f\eta)$.
       \end{proof}

                 \subsection{Modularity}\label{Modularitysec}    
                We  first formulate   the modularity conjecture for CM cycles. Then we prove some  cases of it.

                 \subsubsection{The conjecture}
                    Consider the following subspace of $ \vil _{N  }Z^k(Y(N)):$
$$ CM(\Omega ): =\{  Z_{\Omega,f} : f \in \cS(\GL_2(\BA^\infty))   \}.$$ 
      Let  $\ol{CM}(\Omega)$    be the quotient of  ${CM}(\Omega )$    by the 
kernel of the normalized height pairing \eqref{NHP} with   ${CM}(\Omega^{-1} )$.  
Define   $\ol{CM}(\Omega^{-1})$   similarly.
 By \eqref{ZTZ00},   the left $\GL_2(\BA^\infty)$-action $L^\coh$ on  $\vil_N\Ch^k(Y(N))$
 stabilizes  $\ol{CM}(\Omega)$ and  $\ol{CM}(\Omega^{-1})$, and thus
   induces   left $\GL_2(\BA^\infty)$-actions on them.
   From now on, we fix this left action.
   Under the left action, the central character of $\ol{CM}(\Omega)$ is $\Omega$.
  
Let $\cA $ be 
  the  set of cuspidal automorphic representations $\pi$ of $\GL_{2,\BQ}$ such that 
  \begin{itemize}\item $\pi_\infty$ is the holomorphic discrete series of weight $2k$,
  \item $\pi$ has  central character $\Omega^{-1}|_{\BA^\times}  $.
    
  \end{itemize}  
    Under the second condition, $  \pi_K\otimes \Omega$ is self-dual.
In particular, the $L$-function    $$L (s ,\pi_K\otimes \Omega) =L (s ,\wt\pi_K\otimes \Omega^{-1})$$
     with  root number  $$\ep (1/2 ,\pi_K\otimes \Omega) =\ep (1/2,\wt\pi_K\otimes \Omega^{-1})\in \{\pm1\}.$$

For $\pi\in \cA$,  let $\pi^\infty$ be the finite component. Let $\wt \pi^\infty$ be the admissible dual of $\pi^\infty$. 
Let 
\begin{align*}\ol{CM}[\pi]&=\Hom_{\GL_2(\BA^\infty)}(\wt \pi^\infty,\ol{CM}(\Omega)), \\
 \ol{CM}[\wt \pi]&=\Hom_{\GL_2(\BA^\infty)}( \pi^\infty,\ol{CM}(\Omega^{-1})).
 \end{align*}
     
We have  
natural embeddings of $\GL_2(\BA^\infty)$-modules 
\begin{equation}\label{mod0conj1}
\begin{split}\bigoplus_{\pi\in \cA}\wt \pi^\infty\otimes \ol{CM}[\pi]&\incl\ol{CM}(\Omega),\\\bigoplus_{\pi\in \cA}  \pi^\infty\otimes \ol{CM}[\wt\pi]&\incl\ol{CM}(\Omega^{-1}) .
\end{split}\end{equation}
 We propose the following modularity conjecture, following S. Zhang \cite{Zha19}.
\begin{conj} \label{modconj1} The embeddings  \eqref{mod0conj1} are surjective.
\end{conj}

 \begin{rmk} \label{moW}
  The reader should think of  $\ol{CM} $ as  a subgroup of the ``Mordell-Weil group" of the submotive of $\{Y(N)\}_N$ that is spanned by CM cycles.  
  We remind the reader of   the weight 2 case in \cite[p. 78]{YZZ}, there is a decomposition of abelian varieties up to isogeny similar to \eqref{mod0conj1}:  $$  \Jac (X(N)) =\bigoplus \wt \Pi ^{U(N)}\otimes_{\End(\Pi)} A_\Pi,$$ 
where
$\Pi$ runs over irreducible admissible representations of $\GL_2(\BA^\infty)$ over $\BQ$ such that $\Pi_\BC$ is a direct sum of the finite components of Galois Conjugate cuspidal automorphic representations of weight 2,  $\wt \Pi ^{U(N)}$ is the space of $U(N)$-invariants,
and  the $L$-function $L(s,A_\Pi)=L(s-1/2,\Pi)$. Here we consider $  \Jac (X(N))(\BQ)_\BC$'s  as
Chow groups of $X(N)$'s so that it is equipped with a left Hecke action by  the subalgebra of bi-$U(N)$-invariant functions in $\cS\lb\GL_{2}\lb\BA^{S,\infty}\rb\rb$
  as above. This decomposition respects the  Hecke-actions on both sides.
 
  \end{rmk}
  
\begin{rmk} \label{modconjrmk}
It is a conjecture  (attributed to Beilinson and Bloch, see  \cite{Jan}) that  for a smooth projective variety $Y$ over a number field $F$,
 the Abel--Jacobi map from the Chow group of cohomologically trivial codimension-$d$ cycles with $\BQ_l$-coefficients
  $$\Ch^d(Y)_{0,\BQ_l}\to H^1\lb \Gal(\ol F/F),H^{2d-1}(Y_{\ol F}, \BQ_l)\rb$$ 
should be 
injective \cite{Bei}. 
Conjecture \ref{modconj1} is implied by this conjecture as follows.

On   $Y(N)$, there is the natural  action of $\Delta_k: =  \{\pm 1\}\wr S_{2k-2}= \{\pm 1\}^{2k-2}\rtimes  S_{2k-2}$ as in \cite[2.3]{Zha97} with the obvious  compatibility as $N$ varies. 
Then  by the functoriality and conjectural injectivity  of the Abel--Jacobi map, we have 
a $\Delta_k\times \GL_2(\BA^\infty)$-equivariant embedding 
$$\vil_N \Ch^k(Y(N))_{0,\BQ_l}\to \vil _N H^1\lb \Gal(\ol \BQ/\BQ),H^{2k-1}(Y(N)_{\ol \BQ}, \BQ_l)\rb.$$
Let $\vep$ be the character of  $\Delta_k$ that is   the product map on  $ \{\pm 1\}^{2k-2}$, and is the sign character on $S_{2k-2}$. 
By \cite[Lemma 2.4.3]{Zha97},    the image of  $ CM(\Omega ) $  in  $\vil_N \Ch^k(Y(N))_{0,\BQ_l}$ is contained in the $\vep$-component. 
By 
\cite[Theorem 2.3.1]{Zha97} (due to Scholl), the $\vep$-component of $$\vil _N H^1\lb \Gal(\ol \BQ/\BQ),H^{2k-1}(Y(N)_{\ol \BQ}, \BQ_l)\rb\otimes_{\BQ_l} \BC$$ is 
a   direct sum of  elements in $\cA$. Then so are $\Ch^d(Y)_{0}$ (with $\BC$-coefficients) and its subquotient
$ CM(\Omega )$.  Then Conjecture \ref{modconj1}  follows.
\end{rmk}
 
\subsubsection{Unramified modularity} \label{Unramified Hecke action}
 We establish  a version of Conjecture  \ref{modconj1} for modules of   unramified Hecke algebras, which we now define.

For a  finite set $S$ of finite places of $\BQ $, let  $\cH^S\subset  \cS\lb\GL_{2}\lb\BA^{S,\infty}\rb\rb$ be the usual unramified Hecke algebra, i.e. the
sub-algebra   of bi-$\GL_2(\wh \BZ^S)$-invariant functions.
 Let  $$ CM(\Omega  ,\cH^S) =\{  Z_{\Omega, f } :    f\in \cS(\GL_2(\BA_S))\otimes \cH^S\}.$$
Let  $ \ol{CM}(\Omega  ,\cH^S)$ be the image of $ CM(\Omega  ,\cH^S)$  in $\ol{CM}(\Omega)$. 
Then $ \ol{CM}(\Omega  ,\cH^S)$ is a $\cH^S$-submodule of $\ol{CM}(\Omega)$.

  Let $\cA^{\GL_2(\wh \BZ^S)  }\subset \cA$ be 
  the  subset of  representations with nonzero $\GL_2(\wh \BZ^S)$-invariant vectors. 
    For $\pi\in  \cA^{\GL_2(\wh \BZ^S)  }$, let $L_{\pi^S}$ (resp.  $L_{\wt \pi^S}$) be the character  of $\cH^S$ on the 1-dimensional space of $\GL_2(\wh \BZ^S)  $-invariant vectors in $ \pi^{S,\infty}$ (resp. $ \wt\pi^{S,\infty}$).
   Let  
\begin{equation*}\begin{split}
\ol{CM}[ {\pi^S}]&=\Hom_{\cH^S}(L_{\wt \pi^S},\ol{CM}(\Omega,\cH^S)),\\
\ol{CM}[ {\wt \pi^S}]&=\Hom_{\cH^S}(L_{  \pi^S},\ol{CM}(\Omega^{-1},\cH^S)).
\end{split}
\end{equation*}
 
    By     the strong multiplicity one theorem \cite{PS}, $L_{\pi^S}$'s are pairwise non-isomorphic.   
So  the   following 
natural maps of $\cH^S$-modules  are embeddings
\begin{equation}\begin{split}
\bigoplus_{\pi\in \cA^{\GL_2(\wh \BZ^S)  }}L_{\wt \pi^S}\otimes \ol{CM}[\pi^S]&\incl\ol{CM}(\Omega,\cH^S),\\ \bigoplus_{\pi\in \cA^{\GL_2(\wh \BZ^S)  }} L_{  \pi^S}\otimes \ol{CM}[\wt\pi^S]&\incl\ol{CM}(\Omega^{-1},\cH^S) .\end{split}
\label{mod0conj1S}
\end{equation}

       \begin{thm}  \label{strongmodularity}     Assume that $S$  has cardinality at least three, and 
        contains the prime  2  and all finite places of $\BQ$   ramified in $K$.                 Then the embeddings  \eqref{mod0conj1S} are surjective.
       
      \end{thm}
      We will prove Theorem \ref{strongmodularity} in \S \ref{remove}.
       \begin{rmk}    
                As one may expect,  the assumption that $S$ contains 2  comes from local theta lifting, see Lemma \ref{not2}.           
            
             \end{rmk}

      \subsubsection{Full modularity} Let us reformulate  Conjecture  \ref{modconj1} to make it easier to attack.
       We use Theorem \ref{strongmodularity} for the reformulation.
      
      For $S\subset S'$,   we  naturally have  $\ol{CM}(\Omega,\cH^S)\subset \ol{CM}(\Omega,\cH^{S'})$. Choose any isomorphism  $L_{\wt \pi^{S'}} \cong L_{\wt \pi^S} $ (of 1-dimensional spaces).    
Then $f\in  \ol{CM}[{\pi^S}]$,  understood as a linear  map from $L_{\wt \pi^{S'}}$  to $ \ol{CM}(\Omega,\cH^{S'})$, is   also $\cH^{S'}$-equivariant.    Thus we have  an embedding  $\ol{CM}[{\pi^S}]\incl   \ol{CM}[{\pi^{S'}}]$  which gives an inclusion
       \begin{equation} 
          L_{\wt \pi^S}\otimes \ol{CM}[{\pi^S}]\subset L_{\wt \pi^{S'}}\otimes \ol{CM}[{\pi^{S'}}] 
       \label{mod0conj1SS'}
\end{equation} 
in $\ol{CM}(\Omega).$
          Note that the inclusion is independent of the choice of  the  isomorphism  $L_{\wt \pi^S}\cong L_{\wt \pi^{S'}}$.
               Let
      $$  \ol{CM} (\pi)=\vil_S L_{\wt \pi^S}\otimes \ol{CM}[{\pi^S}]\subset \ol{CM}(\Omega),$$ which is stable by the $\GL_2(\BA^\infty)$-action. 
           Define $  \ol{CM} (\wt\pi)\subset \ol{CM}(\Omega^{-1}) $ similarly. 
                 As   $\ol{CM}(\Omega)=\vil_S  \ol{CM}(\Omega  ,\cH^S)$   by definition,    we have 
\begin{equation} \ol{CM}(\Omega)= \bigoplus_{\pi\in \cA} \ol{CM} (\pi) \label{omelim}\end{equation} 
by Theorem \ref{strongmodularity}.
  So     Conjecture \ref{modconj1} is implied by  the truth of the following conjecture for all $\pi\in \cA$.

       \begin{conj}\label{modconj}Under the embedding \eqref{mod0conj1}, we have      $$  \ol{CM} (\pi)=\wt\pi^\infty\otimes \ol{CM} [\pi]$$
        in  $\ol{CM}(\Omega)$.
The corresponding equality for $  \ol{CM} (\wt\pi)$ also holds.
                 \end{conj} 
        Moreover,    if Conjecture \ref{modconj1} holds,  by     the strong multiplicity one theorem \cite{PS}, $L_{\wt \pi^S}\otimes \ol{CM}[{\pi^S}]$ sits in $ \wt\pi^\infty\otimes \ol{CM} [\pi]$.  Thus  Conjecture \ref{modconj1} is in fact equivalent to         
           Conjecture \ref{modconj} for all $\pi\in \cA$.

Our main result  on Conjecture \ref{modconj} is as follows.
                    \begin{thm}  \label{strongermodularity}  
         Assume that   $L (1/2 ,\pi_K\otimes \Omega) =0$.  Then   Conjecture \ref{modconj} holds.
        In particular, if  
         $\ep (1/2 ,\pi_K\otimes \Omega) =-1, $ then   Conjecture \ref{modconj} holds.
 
              \end{thm}
                 The theorem is proved in \S \ref          {removeer}.
              \subsubsection{Rank} 
              To motivate the next result, recall that
              the theorem of Tunnell \cite{Tun} and  Saito \cite{Sai}  implies the following result.
 \begin{prop}  \label{stupid}Assume      Conjecture \ref{modconj}. Then         $ \ol{CM}(\pi)=\{0\}$ if    $\ep (1/2 ,\pi_{K_p}\otimes \Omega_p) \Omega_p(-1)= -1$ for some  $p<\infty$.
\end{prop}

Wihtout Conjecture \ref{modconj}, 
we prove the following unconditional result.
     \begin{thm}      \label{Vani}     
  Assume  that $\ep (1/2 ,\pi_{K_p}\otimes \Omega_p)\Omega_p(-1) = -1$ for more than one  $p<\infty$. Then   $ \ol{CM}(\pi)=\{0\}$ and  $ \ol{CM}(\wt\pi)=\{0\}$.  In particular, Conjecture \ref{modconj} holds.
                                 \end{thm} 
                            The theorem is proved in \S \ref          {removeer}.

                                    By Proposition \ref{stupid}, if  $ \ol{CM}(\pi)\neq\{0\}$, we should expect  $\ep (1/2 ,\pi_{K_p}\otimes \Omega_p) \Omega_p(-1)= 1$ for  every  $p<\infty$. As $\ep(1/2,\pi_{K_\infty}\otimes\Omega_\infty)\Omega_\infty(-1)=-1$,          $\ep (1/2 ,\pi_K\otimes \Omega) =-1$ and thus $L (1/2 ,\pi_K\otimes \Omega) =0$.  
           Motivated by  the analog in the weight 2 case (see Remark \ref{moW}), we predict the following.
         
  \begin{conj} \label{rankconj1}  
   For $\pi\in \cA$,  $\dim\ol{CM}[\pi]\leq 1$.   The equality holds  if and only
   if 
   $\ep (1/2 ,\pi_{K_p}\otimes \Omega_p) \Omega_p(-1)= 1$ for  every  $p<\infty$ and 
    $L' (1/2 ,\pi_K\otimes \Omega) \neq 0$.

\end{conj}
The ``if" part is proved in Corollary \ref{BBB}. The ``only if" part may be proved similarly assuming Conjecture \ref{modconj}, see Proposition \ref{stupid}. However, the claim $\dim\ol{CM}[\pi]\leq 1$ seems out of reach. Assuming  injectivity  of the Abel--Jacobi map, the subspace of $\dim\ol{CM}[\pi]$ coming from  CM cycles on 
Kuga-Sato varieties over $X_0(N)$'s have dimension at most 1, by Nekov\'a\v{r}'s theorem \cite{Nek}.
  
 \subsection{Height  pairing and derivatives}\label{Height and derivatives}

Now we  state the higher weight analog of   the general Gross--Zagier formula  of Yuan, S. Zhang, and W. Zhang.
We need to fix some notations first. 

\subsubsection{Local periods}\label{Local periods}
Recall that $\eta$ is the Hecke character of $\BQ^\times $ associated to $K$ by the class field theory.

Fix  the following measures. For a finite place $p$, endow $\BQ_p^\times$ with the Haar measure such that $\Vol(\BZ_p^\times)=1$. This gives the measure on $K_p^\times/\BQ_p^\times$ for $p$ split in $K$.  
For $p$ unramified in $K$, endow $K_p^\times/\BQ_p^\times$  with the Haar measure such that $\Vol(K_p^\times/\BQ_p^\times)=1$.
For $p$ ramified in $K$, endow $K_p^\times/\BQ_p^\times$  with the Haar measure such that $\Vol(K_p^\times/\BQ_p^\times)=2|D_p|_p^{1/2}$ where $D_p$ is the local discriminant. See also \cite[1.6.2]{YZZ}.

            Let $\pi_p$ be an infinite-dimensional irreducible unitary admissible representation of $\GL_2(\BQ_p)$.   For $  \phi_p\in   \pi_p$ and $   {\wt\phi}_p\in  \wt \pi_p$,  
let    \begin{equation*}\alpha^\sharp _{\pi_p} ( \phi_p\otimes {\wt\phi}_p)=\alpha^\sharp _{\pi_p} ( \phi_p, {\wt\phi}_p) =\frac{L(1,\eta_p)L(1,\pi_p,\ad)}{\zeta_{\BQ_p}(2 )L(1/2,\pi_{p,K_p}\otimes\Omega_p)} \int_{K_p^\times/\BQ_p^\times} \lb \pi_p(t) \phi_p,{\wt\phi}_p\rb \Omega_p(t)d t  . \end{equation*} 
    Then 
   $ \alpha^\sharp _{\pi_p}( \phi_p\otimes {\wt\phi}_p) =1$ if  $K_p,\Omega_p,\pi_p$ are unramified, $\phi_p,{\wt\phi}_p$ are $\GL_2(\BZ_p)$-invariant, and $(\phi_p,{\wt\phi}_p)=1$.
Let $$\alpha _{\pi^\infty}^\sharp= \prod_{p <\infty }    \alpha_{\pi_p}^\sharp.$$

 \subsubsection{Yuan-Zhang-Zhang type formula I}\label{Height and derivatives1}

For $f \in  \cS(\GL_2(\BA^\infty)) $, let $z_{\Omega,f }$ be      the image of $Z_{\Omega,f}$ in $\ol {CM}(\Omega)$.   
 For $\pi\in \cA$, let $z_{\Omega,f,\pi}\in   \ol{CM} (\pi)$ be the   $ \pi $-component of $z_{\Omega,f}$, defined  by applying  \eqref{omelim}. 
 
We give an explicit    description of    $z_{\Omega ,f  , \pi}$ and  a consequential lemma that   will only be used in later sections. The reader may skip them and go to  Lemma \ref{modconjlem}.
  Choose $S$ large enough such that  the conditions in Theorem \ref{strongmodularity}    hold and 
  $$f=f_S\otimes f^S \in \cS(\GL_2(\BA_S))\otimes \cH^S.$$ Then  $z_{\Omega ,f  , \pi}$  
   is the projection of   $z_{\Omega ,f }$ to  $ L_{\wt \pi^S}\otimes \ol{CM}[{\pi^S}]$ under Theorem \ref{strongmodularity}.   
  The projection is $\cH^S$-equivariant. 
   In particular,    as 
   $z_{\Omega ,f }=L_{f^{S,\vee}}^\coh (z_{\Omega ,f_S })$ (see \eqref{ZTZ0}), we have
    \begin{equation} z_{\Omega ,f  , \pi} =L_{\wt \pi^S}(f^{S,\vee}) z_{\Omega,f_S ,\pi}= L_{  \pi^S}(f^{S }) z_{\Omega,f_S ,\pi}.
    \label{zpif}\end{equation}
 More explicitly, 
 choose a finite subset $\cA_1$ of $\cA $ containing $\pi$  such that  of $z_{ \Omega  , f }$    lies in
 the sum of the $L_{ \wt \pi_1^S}$-components of $\ol {CM}(\Omega ,\cH^S)$ over all $\pi_1\in \cA_1$. The existence of
 $\cA_1$ is assured by
 Theorem \ref{strongmodularity}.
   Choose $f_1 \in \cH^S$ which acts   as identity  on   $ \pi_1$   if $\pi_1=  \wt\pi$,
 and as 0  otherwise.  The existence of $f_1$  follows from the strong multiplicity one theorem and the density theorem of Jacobson and Chevalley for semisimple modules, see for example \cite[p 73]{YZZ}.
 Then 
     \begin{equation} 
     z_{\Omega,f,\pi} =L_{f_1 }^\coh (z_{\Omega,f}) =z_{\Omega,f* f_1^{\vee}  }  .
     \label{zpif1}\end{equation}
 Here the second ``=" is \eqref{ZTZ0}.
 A consequential lemma  is as follows.
 \begin{lem} 
\label{lem:eta2}
 If  $\det  (\supp f )\subset  \Nm(\BA_K^{\infty,\times})$, then   $$ \pair{ z_{\Omega,f,\pi},z_{\Omega^{-1} }}= \pair{ z_{\Omega,f,\pi\otimes \eta},z_{\Omega^{-1} }}.$$

 \end{lem}
\begin{proof}  
 Let $f_1$ be as above the lemma. Then 
 $z_{\Omega,f,\pi\otimes \eta}=z_{\Omega,f* (f_1^{\vee}  \eta)}$. 
  As $\det  (\supp f )\subset  \Nm(\BA_K^{\infty,\times})$, it is easy to see that 
$f* (f_1^{\vee}  \eta)=(f* f_1^{\vee} ) \eta$. Now the
lemma follows from Lemma \ref{lem:eta}.
\end{proof}

Note that \eqref{omelim} is defined   independent of any conjecture. 
 Now we impose  Conjecture \ref{modconj}.
\begin{lem} 
\label{modconjlem}
 Assume that  Conjecture \ref{modconj} holds for $\pi$.
If the  image of $f^\vee$ in  $ \pi^\infty \otimes    \wt\pi^\infty\subset \Hom( \wt\pi^\infty, \wt\pi^\infty) $  via  the usual (left) Hecke action is $ 0$ (i.e.,  $f^\vee$ acts on $\wt\pi^\infty$ as 0), then $z_{\Omega,f,\pi}=0$.  
\end{lem}
\begin{proof}  
Choose $f_1\in\cS(\GL_2(\BA^\infty)) $ such that $f_1*f =f$. By \eqref 
{ZTZ0}, $z_{\Omega,f}=L_{f^\vee}^\coh (z_{\Omega,f_1})$. As each direct summand of \eqref{omelim} is  $\GL_2(\BA^\infty)$-stable (and thus Hecke stable),
$z_{\Omega,f,\pi}=L_{f^\vee}^\coh (z_{\Omega,f_1,\pi})$. The right hand side is 0 by Conjecture \ref{modconj} and the condition on 
$f^\vee$.
\end{proof}

 For  $\phi \in\pi ^\infty$ and ${\wt\phi} \in \wt \pi ^\infty$, let $f \in  \cS(\GL_2(\BA^\infty)) $  such that the  image of $f^\vee$ in  $ \pi^\infty \otimes    \wt\pi^\infty\subset \End( \wt\pi^\infty) $  via  the usual (left) Hecke action is $ \phi \otimes{\wt\phi}$.   Define 
 \begin{equation} z_{  \phi \otimes{\wt\phi}}=z_{\Omega,f ,\pi}\in  \ol{CM} (\pi),\label{fff}\end{equation}
 which does not depend on the choice of $f$ under Conjecture \ref{modconj} by Lemma \ref{modconjlem} (but depends on the measure on $\GL_2(\BA^\infty)$). 
  Similarly, we define  $z_{   {\wt\phi} \otimes { \phi}  } \in  \ol{CM} (\wt\pi)$.
    
  The following formula  is the counterpart  of the projector version of  the general Gross--Zagier formula  in weight 2 \cite[Theorem 3.15]{YZZ}.
          \begin{thm}      \label{stronghtder}   
            Assume that  Conjecture \ref{modconj} holds for $\pi$.
For $\phi \in\pi^\infty$ and ${\wt\phi} \in \wt \pi^\infty$,  we have    \begin{equation}   \label{stronghtder1}
\frac{ \pair{z_{   \phi \otimes{\wt\phi}},z_{\Omega^{-1} }}}{\Vol(\GL_2(\wh \BZ))}= \frac{   L' (1/2 ,\pi_K\otimes \Omega) }{2 L(1,\eta) ^2L (1,\pi,\ad)}  \frac{(2k-2)!  }{  (k-1)!\cdot   (k-1)!}      \alpha_{\pi^\infty}^\sharp( \phi \otimes{\wt\phi} ).
          \end{equation}  
                   \end{thm} 
                   All global $L$-functions (or zeta functions) in this paper are the complete ones containing the archimedean parts.
               The theorem is proved in \S \ref    {proofs}.

                   \begin{rmk}    \label{still}We will explain the following assertions, which one may expect, in \S \ref {proofs}
before the proof of the theorem.
                   \begin{itemize}
         \item[(1)]            The number  $ \frac{(2k-2)! }{  (k-1)!\cdot   (k-1)!}$ is part of a local factor   \eqref{oinf2}          at $\infty$.

        \item[(2)]   If we pretend $k=1$, the    formula should (formally) be the same with \cite[Theorem   3.15]{YZZ} (up to the choices of measures), see Remark \ref{k=1}.        
              \end{itemize}
                   
                          \end{rmk}
             
              \subsubsection{Yuan-Zhang-Zhang type formula II}\label{Height and derivatives2}

      We define a height pairing between $\ol{CM}[\pi]$ and $  \ol{CM}[\wt\pi]$, following S. Zhang \cite{Zha19}. 
   For $l\in \ol{CM}[\pi]$ and $\wt l\in \ol{CM}[\wt\pi]$, $\phi\in \pi^\infty$ and $\wt\phi\in \wt \pi^\infty$ with $(\phi,{\wt\phi})\neq 0$, let 
   \begin{equation}\label{pipair}\pair{l,\wt l}=\frac{\pair{l(\wt \phi),\wt l({\phi})}^\sharp}{(\phi,{\wt\phi})}.\end{equation}
 By Schur's lemma,  the pairing  \eqref{pipair} does not depend on the choices of $\phi , {\wt\phi}$.
 By definition, the pairing   \eqref{pipair} is non-degenerate.

Assume that  Conjecture \ref{modconj} holds for $\pi$.
For ${ \phi}\in  \pi$, let $$l_{ \phi} \in  \ol{CM}[\pi]=\Hom_{\GL_2(\BA^\infty)}(\wt\pi^\infty,\ol{CM}(\Omega)) 
$$ be defined by  $\wt \phi\mapsto z_{ \phi \otimes{\wt\phi}}$. 
Since $z_{  \phi \otimes{\wt\phi} }$'s span $\ol{CM} (\pi)$, every element in $ \ol{CM}[\pi]$ equals some $l_{\phi}$. 
Similarly, every element in $ \ol{CM}[\wt \pi]$ equals $l_{ \wt \phi}:\phi\mapsto z_{  {\wt\phi} \otimes { \phi} } $ for some $\wt \phi\in\wt \pi$.
 

    The following  theorem   is  
    a consequence of Theorem \ref{stronghtder}. 
  It is  in the form of   the general Gross--Zagier formula   \cite[Theorem 1.2]{YZZ}.  This is suggested by S. Zhang. 

    \begin{thm}      \label{stronghtderthm}     
Assume that  Conjecture \ref{modconj} holds for $\pi$.        
For $\phi \in\pi^\infty$ and ${\wt\phi} \in \wt \pi^\infty$,  we have    \begin{equation}   
\frac{ \pair{l_{  { \phi} },l_{\wt\phi }}}{\Vol(\GL_2(\wh \BZ))} = \frac{L' (1/2 ,\pi_K\otimes \Omega) }{L(1,\eta) ^2L (1,\pi,\ad)}  \frac{(2k-2)!  }{ 2 (k-1)!\cdot   (k-1)!}      \alpha_{\pi^\infty}^\sharp( \phi ,{\wt\phi} ).\label{eq:stronghtderthm}     
          \end{equation}  
                   \end{thm} \begin{rmk} 
              (1)   Remark \ref  {still} still applies, replacing \cite[Theorem   3.15]{YZZ}  in loc. cit.
    by  \cite[Theorem  1.2, 3.13]{YZZ}.
                 
 (2) Theorem \ref{stronghtderthm}  can be viewed as an evidence toward the conjecture of Beilinson and Bloch \cite[Conjecture 5.4]{Bei}, which is a higher dimensional generalization of the Birch--Swinnerton-Dyer conjecture.  Compare with Remark \ref{moW}.
         \end{rmk}
           
        Let us draw some unconditional corollaries  from Theorem \ref{stronghtderthm}   by using Theorem \ref{strongermodularity} to remove  the dependence on Conjecture \ref{modconj} in  Theorem \ref{stronghtderthm}.   
             \begin{cor}      \label{uncond}     
    If  $L (1/2 ,\pi_K\otimes \Omega) =0$, then   \eqref{eq:stronghtderthm}     holds. 
                                 \end{cor}

               \begin{cor}      \label{Vani2}     
   If $L (1/2 ,\pi_K\otimes \Omega) =0$ and $L '(1/2 ,\pi_K\otimes \Omega) =0$, then   $ \ol{CM}(\pi)=\{0\}$.  
                                 \end{cor}

                           Concerning   Conjecture \ref    {rankconj1}, we have the following.     
                                   \begin{cor}     \label{BBB}
  Assume  that  $L' (1/2 ,\pi_K\otimes \Omega) \neq 0$. If $\ep (1/2 ,\pi_{K_p}\otimes \Omega_p) \Omega_p(-1)= 1$ for all  $p<\infty$, then $\dim\ol{CM}[\pi]=1$.  Otherwise, $\ol{CM}[\pi]=\{0\}$.
                     \end{cor} 
                 
                 \begin{proof} As $L' (1/2 ,\pi_K\otimes \Omega) \neq 0$, $\ep (1/2 ,\pi_{K}\otimes \Omega) =-1$. Thus $L (1/2 ,\pi_K\otimes \Omega) = 0$.
                 By  Theorem \ref {strongermodularity}, Conjecture \ref{modconj} holds for $\pi$.  Thus the ``otherwise" part  follows from Proposition \ref{stupid}. Now assume  that  $\ep (1/2 ,\pi_{K_p}\otimes \Omega_p) \Omega_p(-1)= 1$ for all  $p<\infty$.
                    By   the theorem of Tunnell \cite{Tun} and  Saito \cite{Sai},   $\dim\Hom_{\BA_{K}^{\infty,\times}}  (\pi^\infty\otimes \Omega^\infty ,\BC)=1$. 
Let $l\neq 0$  be in this $\Hom$ space. Let $\phi_0\in  \pi^\infty$ such that $l(\phi_0)\neq 0$. Then by Theorem \ref{stronghtderthm} and the discussion above it, $l_\phi=\frac{l(\phi)}{l(\phi_0)}l_{\phi_0}$.    \end{proof}

      \section{Relative trace formulas and height pairing}\label{The automorphic distributions}
We use the arithmetic relative trace formula approach  to prove Theorem \ref{stronghtder}. 

There will be a relative trace formula  for each one of  two (types of) groups. One group is   $ \GL_{2,K}$. 
The other is a quaternionic group, i.e., the
unit group of  a quaternion algebra    over $\BQ$ containing $K$. 
We will further specify these data later when necessary.
Each relative trace formula is an identity between  the orbital integral side and the automorphic distribution side of this group. The two relative trace formulas are the two vertical equations in the diagram   \eqref{RTF} below. 
They are related by comparing their orbital integrals.
    \begin{equation} \xymatrix{
\textit{orbital integrals for a quaternionic group}  \ar@{=}[r]^ {\textit{\ \ \ \ \ \ \ comparison}} \ar@{=}[d]^{\textit{after summation}} &\textit{orbital integrals for }\GL_2\ar@{=}[d]^{\textit{after summation}} \\
\textit{automorphic distributions for a quaternionic group}   \ar@{=}[r]   &\textit{automorphic distributions for } \GL_2.}
 \label{RTF} \end{equation}
Call the resulted bottom identity a relative trace identity.

The structure of this section is as follows. We   introduce the trace formulas, including orbital integrals and automorphic distributions (plus local distributions) in the first three subsections.
We do   not  establish the two vertical equations in the diagram  \eqref{RTF}, but will refer to literature when we need them.
 The  arithmetic variants
and the comparison are then done in \S \ref {Height distribution}.  The   conclusion  of the comparison is an arithmetic relative trace identity  \eqref{ATI}. 
Finally in \S \ref{proofs},  we use the comparison to prove  the higher weight general Gross--Zagier formula.

\subsection{Orbital integrals}

We discuss orbits and (derivatives of) orbital integrals on two groups separately, and than compare them. 
The main result is Proposition \ref {transfer}. 

We start with a more general setup.
Let $F$ be a  field,  $E$    a separable  quadratic  field extension   of $F$,   $z\mapsto \bar z$ the Galois conjugation, and   $\Nm:E^\times\to F^\times$   the norm map.      
 (We omit the   case  $E=F\oplus F$, and refer to \cite[4.3]{Qiu}.)

 \subsubsection{Orbits I}   We first deal with the orbits for  $\GL_2$.
 
 Let $G=\GL_2(E)$. 
Let $\cV\subset G $ be the  subset of invertible Hermitian matrices over $F$ with respect to  $E$.  
Let $E^\times\times F^\times$ act on $\cV$ via $$(a,z)\cdot s= \begin{bmatrix}a&0\\
0&1\end{bmatrix}s\begin{bmatrix}\bar a&0\\
0& 1\end{bmatrix}z.$$
 Define  $$\Inv :    \cV\to \BP^1(F)-\{1\}=  (F^\times-\{1\}) \cup\{0,\infty\}    ,$$
which maps $A=\begin{bmatrix}a&b \\
 \bar b&d\end{bmatrix}$ to $$ \frac{ad}{b\bar b}=\frac{\det(A) }{b\bar b}+1.$$
Call $\gamma\in \cV$    regular\footnote{Here, for simplicity, we abuse the notion of ``regular"  for the more standard notion of ``regular semi-simple" for   reductive group actions.}  for this action if $\Inv (\gamma)\in F^\times-\{1\}$.    Let $\cV_{ \reg}\subset \cV$  be the regular locus.
The restriction of $\Inv$ to the regular  orbits $ E^\times\bsl \cV_\reg/ F^\times$ are bijective to  $F^\times-\{1\}$%

  To relate functions on $G$ and $\cV$, consider 
 $g\in G$ acting on $\cV$ by 
$g\cdot s: =gs  \bar g^t$, where $ \bar g^t$ is the Galois conjugate of the transpose of $g$. 
Let $H_0\subset G$ be  the unitary group associated to $$w=\begin{bmatrix}0&1\\
 1&0\end{bmatrix}\in \cV,$$ 
 i.e., $H_0$ is  the stabilizer of $w$ for the above action.
    If $F$ is a local field (archimedean or non-archimedean) so that
    $\cV$ is equipped with the corresponding topology, then $ G\cdot w\subset \cV$ is an open subspace. 
   If $H_0$ is given a Haar measure, for $f\in \cS_c(G )$, let  $ \Phi_{f}\in \cS_c(G\cdot w)$, such that  
    \begin{equation}\label{Phi}\Phi _f(g\cdot w  )=\int_{H_0 } f(gh) dh.\end{equation}   
    The map $f\mapsto \Phi_f$ is surjective  to $\cS_c(G\cdot w)$.
      If $F$ is a global field, the same definition applies to  $f\in \cS_c(G(\BA_F) )$.

 Let  $\lb1- \Nm (E^\times)\rb=\{1-\Nm(a):a\in E^\times\}$.  Then 
 \begin{equation*}
 \Inv(G\cdot w)=\lb1- \Nm (E^\times)\rb\cup \{\infty\}.
\end{equation*}

    We have the following easy lemma.
  \begin{lem} \label{Xue0}(1) The map $\Inv$ realizes $(G\cdot w)\cap \cV_\reg$ as a trivial (topological) $E^\times\times F^\times$-bundle over $\lb1- \Nm (E^\times)\rb-\{0\}$. 
   \end{lem}

 \subsubsection{Orbits II}  \label{matchorb}
 Then we discuss the orbits on a  quaternionic group.
 
 Let $B$ be a quaternion algebra    over $F$ containing $E$.
  There exists $j\in B $   such that $B=E\oplus Ej$,  $j^2=\ep\in F^\times$ and $jz=\bar zj$ for $z\in E$.
  Then the reduced norm on $B$  can be expressed by  $q(a+bj)=\Nm (a)-\Nm (b)\cdot\ep.$
Note that different choices of $j$ give the same $\ep$ in $F^\times/\Nm(E)^\times$. In fact, $F^\times/\Nm(E)^\times$ classifies  quaternion algebras    over $F$ containing $E$   (see  \cite[4.1]{Qiu}).
And $\ep=1$ if and only if $B=\RM_{2,F}$.

 Let $E^\times\times E^\times$ act on $B^\times$ by $$(h_1,h_2)\cdot \gamma=h_1^{-1}\gamma h_2.$$
Define an invariant for this action: $ \inv (a+bj)= \ep  \Nm (b)/\Nm (a).$  
 This  induces a  bijection $$\inv :   E^\times \bsl B^\times/E^\times  \cong  \lb \ep \Nm   (E^\times)-\{1\}\rb \cup\{0,\infty\}     .
 $$    
 Call  $\delta=a+bj$    regular for this action if $\inv\neq 0,\infty$, equivalently   $a  b  \neq 0$.  
  Let $B^\times_\reg\subset B^\times $  be the  regular locus. 
  We have the following easy lemma.
     \begin{lem} \label{qcond}For $g\in B^\times_\reg$, $\inv(g)\in (1- \Nm (E^\times))$ if and only if $q(g)\in \Nm(E^\times).$
     \end{lem}

\subsubsection{Orbital integrals I}  Now we deal with the orbital integrals for  $\GL_2$.

 Let $F$ be a local field (archimedean or non-archimedean).
 Let  $\eta$  be the quadratic character of $F^\times$   associated to  $E$ by the class field theory.
Let $\Omega$ be a continuous character of    $E^\times$  and   $\omega=\Omega|_{F^\times}$. 
   
   For  $\Phi\in  \cS_c(\cV)$, $\gamma\in \cV_\reg$, and $s\in \BC$, define the orbital integral  \begin{equation}O(s, \gamma,\Phi)= \int_{E^\times}  \int_{F^\times} \Phi\left(z\begin{bmatrix}a&0\\
0&1\end{bmatrix}\gamma\begin{bmatrix}\bar a&0\\
0&1 \end{bmatrix} \right)\eta\omega^{-1}(z)\Omega^{-1} (a) |a|_E^sd  zd  a. \label{int0}
\end{equation}
  The integral \eqref{int0} converges absolutely and defines a holomorphic  function of $s$.   
   For  $x\in F^\times  -\{1\}$, 
let $$\gamma(x) =  \begin{bmatrix}x&1\\
1&1\end{bmatrix}\in \cV_\reg .$$
Then $\Inv(\gamma(x))=x$.    
Define $$O(s, x,\Phi)=O(s, \gamma(x),\Phi),\ x\in F^\times-\{1\}.$$ 
Let $O'(0,x,\Phi)$ be the $s$-derivative  at $s=0$ (for the legitimacy of taking $s$-derivative, see \cite[Lemma 4.2.2]{Qiu}).
They are all smooth functions on $x$.

   \begin{lem}\label{Xue1}
   (1) There is a neighborhood of $1$ in $F$ on which
 $O(s, x,\Phi)$  (thus $O(0,x,\Phi)$ and $O'(0,x,\Phi)$) vanishes        for all $s$ simultaneously.

 (2)  If   $\Phi$ is further compactly supported on    $\cV_\reg$, 
 there is a compact subset of $F^\times-\{1\}$  on which $O(s, x,\Phi)$   (thus $O(0,x,\Phi)$ and $O'(0,x,\Phi)$) is supported   for all $s$ simultaneously. In particular,  $O(0,x,\Phi)$ and $O'(0,x,\Phi)$ extends to a Schwartz function on $F$ by extension by 0.

 (3)   For $\phi \in \cS_c\lb \lb1- \Nm (E^\times)\rb\rb $ compactly supported, there is   $\Phi\in \cS_c\lb(G\cdot w)\cap \cV_\reg\rb)$ such that $O(0, x,\Phi)=\phi(x)$,  both understood as  Schwartz  functions on $F$ by extension by 0.
  \end{lem}
\begin{proof}
 As $\Phi$ is compactly supported on $\cV$ and  $\Inv:\cV\to \BP^1-\{1\}$ is continuous, $\Inv$ maps the support of $\Phi$  to a compact set. 
  (1) follows.
(2) is similar. (3) is an immediate corollary of Lemma \ref{Xue0}.
\end{proof}

   \begin{lem} \label{lem:GLmatch}  
  (1) Let $ A(x) $ be  a smooth function defined on  a   neighborhood of $0\in F$.
 There is $\Phi\in \cS_c(G\cdot w)$ such that 
  \begin{itemize}
    
\item 
  both  $O(0,x,\Phi) , O'(0,x,\Phi)$
  extend to    smooth functions on $F$  with compact supports, and 
   
\item   $O(0,x,\Phi)=A(x) $  
  for $x\neq 0$ in a (smaller) neighborhood of $0$ in $\lb1- \Nm (E^\times)\rb$. 
    \end {itemize}
       
  (1')     Let $ A(x) $ be  a smooth function defined on  a   neighborhood of $0\in F$.
 There is $\Phi\in \cS_c(G\cdot w)$ such that 
  \begin{itemize}
    
\item   both 
 $O(0,x,\Phi) , O'(0,x,\Phi)$ have compact supports on $F$ and  are smooth outside 0, and 
 \item 
    for $x\neq 0 $ in a neighborhood of $0$ in  $\lb1- \Nm (E^\times)\rb$, 
   we have $O(0,x,\Phi)=\eta(x) A(x) $ and  
      $$O'(0,x,\Phi)=\eta(x) \lb -  \frac{1}{2}A(x) \log |x|_E+C(x) \rb,$$ 
      where $C(x)$ is a  smooth function.
    \end {itemize}

  (2)  Let $ A(x) $ be  a smooth function defined on    a   neighborhood   of $0$ in $E$.
      There is $\Phi\in \cS_c(G\cdot w)$ such that  
       \begin{itemize}
    
\item   both 
 $O(0,x,\Phi) , O'(0,x,\Phi)$ vanish  in a neighborhood of $0$   and  are smooth outside 0, and 
 \item     for  $x\in \lb1- \Nm (E^\times)\rb $  with large enough absolute value so that we can write
   $x$ as $-\Nm(b) $ with $b^{-1}$ in a small enough neighborhood   of $0$ in $E$, we have   $O(0,x,\Phi)=\Omega(b) A(b^{-1})$. 
      \end {itemize}
     \end{lem}
    \begin{proof} (1) Let $\phi\in \cS_c(G\cdot w)$ be  supported  in a small enough neighborhood of $ \begin{bmatrix}0&1\\
1&1\end{bmatrix}$ and nonzero at  $ \begin{bmatrix}0&1\\
1&1\end{bmatrix}$. A direct computation shows that $O(0,x,\phi)$ extends to a  smooth nonvanishing  function on a neighborhood of $0\in F$. Then the pullback of $A(x)/O(0,x,\phi)$, by the map $\Inv$,  to   a neighborhood of $\Inv^{-1}(\{0\})$ in $G\cdot w$ extends to a smooth function on $G\cdot w$.
Let $\Phi$ be the product of $\phi $ with this smooth extension, and we proved the first part of the lemma.
 A direct computation shows that  $O'(0,x,\phi)$ extends to a  smooth function on a neighborhood of $0\in F$. So does  $O'(0,x,\Phi)$

(1') Consider   $\phi\in \cS_c(G\cdot w)$  supported in a small enough neighborhood of$ \begin{bmatrix}1&1\\
1&0\end{bmatrix}$.  The argument is similar to the one in the proof of (1).

(2) Consider $\phi\in \cS_c(G\cdot w)$ supported  in a small enough neighborhood of $ \begin{bmatrix}-1&0\\
0&1\end{bmatrix}$. The argument is similar to the one in the proof of (1).
  \end{proof}
    
\subsubsection{Orbital integrals II} Then we deal with the orbital integrals on the  quaternionic group.

   For $f\in \cS_c (B^\times)$  and 
$\delta\in {B^\times_\reg} $, define the orbital integral  \begin{align*}O(  \delta,f):
=\int_{E^\times/F^\times}\int_{E^\times } f(h_1^{-1}\delta h_2) \Omega(h_1) \Omega^{-1}(h_2)  dh_2dh_1. 
\end{align*}
    For  $x  \in \ep  \Nm (E^\times) -\{1\}$ (where $\ep$ is as in \S \ref{matchorb}),  choose $b\in E^\times$ such that $x=\ep \Nm(b)$.
Let $$\delta(x) =1+bj\in B^\times\in B^\times_\reg.$$ 
 Then   $\inv (\delta(x))=x$.
 Define $$O(x,f)= O(\delta(x),f)  ,\ x  \in \ep  \Nm (E^\times) -\{1\}.$$ Easy to check that $O(x,f)$ does not depend the choice of $b$.

    We have the following characterization of $O (x,f)$  by \cite[p. 322, Proposition]{Jac87}  (which  contains an error: it swapped the behaviors of $O( x,f)$ for $x$ near 0 and near $\infty$).

\begin{prop} \label{Proposition 2.4, Proposition 3.3Jac862} 
Let $\phi$ be a function on $\ep \Nm (E^\times)-\{1\}$.  Then $\phi(x)=O( x,f)$    for some   $f\in \cS_c (B^\times)$ if and only if the following  conditions hold:
\begin{itemize}
    
\item[(1)]the function $\phi$  is smooth   on $\ep  \Nm(E^\times)-\{1\}$;

\item[(2)] the function $\phi$ vanishes  in a   neighborhood of  $1$;

\item[(3)]there exists a smooth function $A_1 $ defined on  a   neighborhood $U_1$ of $0$ in $F$,  such that $\phi(x)=A_1(x) $ 
  for $x\in U\cap \lb \ep \Nm (E^\times)-\{1\}\rb $.
  
\item[(4)] 
 there exists a smooth function $A_2 $  defined on  a   neighborhood $U_2$ of $0$ in $E$, such that $\phi(x)=\Omega(b)A_2(b^{-1}) $ for $x =\ep \Nm(b)$  with $b^{-1}\in U_2$; 

 \end{itemize}

      \end{prop} 
  
\subsubsection{Transfer} Now we compare  the orbital integrals on the two groups.

  \begin{defn} \label{matchingdef}  For $f\in \cS_c(B^\times)$,  $f'\in \cS_c(G )$ is a transfer of    $f $   if   $   O(0, x,\Phi_{f'})=O(x,f)$  for     $x\in \ep\Nm(E^\times)-\{1\}$,   and  $   O(0, x,\Phi)=0$ for $x  \in  F^\times  -\{1\}-\ep\Nm(E^\times).$

 \end{defn}  
 
We also have the following  sufficient condition for  a transfer  to exist.

     \begin{lem}\label{Xue} (1)  Let  $f\in \cS_c(B^\times_\reg)$   such that $O(x,f)$ vanishes outside  $\lb 1- \Nm (E^\times)\rb $, there is a    transfer   $\Phi\in \cS_c\lb(G\cdot w)\cap \cV_\reg\rb)$. 
     
     (2)
For $\Phi\in \cS_c\lb(G\cdot w)\cap \cV_\reg\rb)$, there is $e\in \cS_c(B^\times_\reg)$   such that $O(x,e)=O'(0,x,\Phi)$.

  \end{lem}
     \begin{proof} (1) Similar to 
     Lemma \ref{Xue1} (1) (2), $O(x,f)$ has compact support contained in $\lb 1- \Nm (E^\times)\rb $ (see also 
     Proposition \ref{Proposition 2.4, Proposition 3.3Jac862}  (2)).     The lemma follows from Lemma \ref{Xue1} (3) and   Proposition \ref{Proposition 2.4, Proposition 3.3Jac862} (1).
     
     (2)  By
     Lemma \ref{Xue1}   (2),  $O'(0,x,\Phi)$ has compact support contained in $\lb 1- \Nm (E^\times)\rb $.
     The rest of the proof is similar to (1) and omitted.
  \end{proof}

As $|F^\times/ \Nm (E^\times)|=2$, by our discussion in the beginning of \ref{matchorb}, there is a unique
    quaternion algebra $B^\dag$ over $F$ and containing $E$ that is non-isomorphic to $B$. Let $\ep^\dag\neq \ep\in F^\times/ \Nm (E^\times) $ be associated to $B^\dag$ as in \S \ref{matchorb}.
  
     \begin{prop}\label{transfer} 
     Let $f\in \cS_c(B^\times)$ such that $O(x,f)$ vanishes outside $\lb 1- \Nm (E^\times)\rb $. 
            Let $A_1(x) $ be as in Proposition \ref{Proposition 2.4, Proposition 3.3Jac862}  (3).
           
     (1) There is a transfer  $f'\in \cS_c(G )$  of    $f $.   
     
          (2) 
  We can choose a  transfer   $f'$ as in (1),
   $f^\dag\in \cS_c(B^{\dag,\times})$   and a neighborhood $U$ of $0$ in $F$ 
     on which $A_1$ is defined, 
      such that for  $x\in U \cap \ep^\dag   \Nm (E^\times) $,   
          we  have \begin{equation} \frac{1}{2}A_1(x) \log |x|_E = -O'(0,x,\Phi_{f'})+O(x,f^\dag) .\label{cnm2}\end{equation}

(3)   Assume that 
$A_1$ extends to
a smooth     function   on $F$  such that
 \begin{itemize}
 
  \item[(a)] $\supp (A_1)\cap \ep^\dag   \Nm (E^\times)$ is bounded from above;
   \item[(b)] $1\not\in \supp (A_1)\cap \ep^\dag   \Nm (E^\times)$.

\end{itemize}
If   $-1\not\in \ep^\dag   \Nm (E^\times)  $, then \eqref{cnm2} holds for  all $x\in   \ep^\dag   \Nm (E^\times) $, i.e., we can remove $U$.

\end{prop}
       \begin{proof} (1) Apply 
       Lemma \ref{lem:GLmatch} (1)  resp. (1') to $A(x):=A_1(x)$ and get $\Phi_1$ resp. $\Phi'_1$. 
       Let $A_2(x)$ be as in Proposition \ref{Proposition 2.4, Proposition 3.3Jac862}  (4).
       Apply 
       Lemma \ref{lem:GLmatch} (2) to $A(x):=A_2(x)$ and get $\Phi_2$. 
       Let $\Phi=\Phi_1-\eta(\ep^\dag)\Phi'_1+\Phi_2$.          By Lemma \ref{lem:GLmatch} (1) (2) and Proposition \ref{Proposition 2.4, Proposition 3.3Jac862}, we have 
       $$\phi:=O(x,f)-O(0,x,\Phi)\in \cS_c\lb \lb1- \Nm (E^\times)\rb\rb.$$
       Apply  Lemma \ref{Xue1} (3) to $\phi$ to get $\Phi_3$. Let  
        $f'\in \cS_c(G )$  such that $\Phi_{f'}=\Phi+\Phi_3$ (recall  the map $f\mapsto \Phi_f$ is surjective  to $\cS_c(G\cdot w)$).
Then $f'$ is a transfer   of    $f $.  

(2) By Lemma \ref{lem:GLmatch}  (1)(2),
$O'(0,x,\Phi_{f'})+2A_1(x) \log |x|$ is a smooth function for $x$ in  a neighborhood $U$ of $0$ in $\lb 1- \Nm (E^\times)\rb$. Choosing $U$ small enough, by Proposition \ref{Proposition 2.4, Proposition 3.3Jac862}, for  $x\in U\cap  \ep^\dag   \Nm (E^\times) $, we have $O'(0,x,\Phi_{f'})+2A_1(x) \log |x|$ equals $O(x,f^\dag)$ for some $f^\dag\in \cS_c(B^{\dag,\times})$.

(3) Let us prove (3)  by modifying $f^\dag$ in (2). 
In (2), we may assume that 
the support of $O(x,f^\dag)$ is  bounded from above. Moreover,  it does not contain $1$ by Proposition \ref{Proposition 2.4, Proposition 3.3Jac862}. If    $-1\not\in \ep^\dag   \Nm (E^\times)  $,  then  $ \lb 1- \Nm (E^\times)\rb\cap \ep^\dag   \Nm (E^\times) $ 
is bounded.
Then by   Lemma \ref{Xue1} (1) and conditions (a) and (b), the  support of $ O'(0,x,\Phi_{f'})+2A_1(x) \log |x|-O(x,f^\dag) $, as a function on $\ep^\dag \Nm (E^\times) $
is a compact  subset of  $\ep^\dag \Nm (E^\times)  $ not containing $1$. 
By Proposition \ref{Proposition 2.4, Proposition 3.3Jac862}, there is $f^{\dag\dag}\in \cS_c(B^{\dag,\times})$ such
 that  
$$O'(0,x,\Phi_{f'})+2A_1(x) \log |x|-O(x,f^\dag)=O(x,f^{\dag\dag})$$ on 
$\ep^\dag \Nm (E^\times) $.
Replace $f^\dag$ by $f^\dag+f^{\dag\dag}$.
  \end{proof}

          \subsection{Automorphic distributions}\label{Automorphic distributions}
  Let $F$ be a global   field, and let $E$ be a separable quadratic  field  extension.
            Let $\Omega$ be a  character     $E^\times\bsl \BA_E^\times$,   $\omega=\Omega|_{\BA_{F}^\times}$, and  $\omega_E=\omega\circ\Nm$. 
       
   \subsubsection{General linear side}    Let $G=\GL_{2,E}$.
Let $A$ be the diagonal torus   of $G$ and   let $Z$  be the center. Let $H\subset G$ be the  similitude unitary group associated to $w$, and 
 let $\kappa$ be the  similitude character. 
 
   For
$f'\in  \cS_c(G(\BA_E))$, define a  kernel function on $G(\BA_E)\times G(\BA_E)$:
\begin{equation*}\label{autker}K(x,y) =\int_{Z(E)\bsl Z(\BA_E) }\left(\sum_{g\in G(E)} f'(x^{-1}gzy)  \right)\omega_E^{-1}(z)dz.\end{equation*}

 For  $a= \begin{bmatrix}a_1&0\\ 0&a_2\end{bmatrix}\in A(\BA_E) $, 
 let  $\Omega(a)=\Omega(a_1 \ol {a_2})$
  and  let 
 $|a_E=|a_1/a_2|_E$.  (Compare with \cite[4.4.1]{Qiu}, where there is a typo.)
 For $s\in \BC$,
formally define the   distribution $O(s,\cdot)$ on $G(\BA_E)$ by assigning to $f' \in \cS_c( G(\BA_E))$ the integral 
\begin{align}O( s, f')=  \int_{Z(\BA_E)A(E)\bsl A(\BA_E)} \int_{Z(\BA_E)H(F)\bsl   H(\BA_F)} K (a,h)\Omega(a) \eta\omega^{-1}(\kappa(h) ) |a|_E^sdh da\label{45} .
\end{align} 
The invariance of the integrand under the inner (resp. outer) $Z(\BA_E)$ follows from the definition of $K(x,y) $ and that the restriction of $\omega\circ\kappa$ (resp. $\Omega$) to $Z(\BA_E)\cong \BA_E^\times$ is $\omega_E$.

We always assume  the following assumption.
 \begin{asmp}\label{freg} Assume that $\Phi_{f'}(g)=0$ for $g\in   \BA_E^\times (\cV-\cV_\reg)  \BA_F^\times  .$
\end{asmp}
\begin{lem}
 The integral $O( s, f')$ in \eqref{45}    converges absolutely
 under Assumption \ref{freg}.  And 
 we have a decomposition
  \begin{align*}
  O'( 0, f')&=\sum_{x\in F^\times-\{1\}}  O'(0,x,\Phi_{f' }).
   \end{align*}    
   \end{lem}
   \begin{proof}

Let us establish the equation in the lemma formally, and refer the convergence to   \cite[Lemma 4.4.2]{Qiu}.
Unfolding $K(a,h)$, applying the definition of $\Phi_{f'}$ and then using $\kappa:H(\BA_F)/H_0(\BA_F)\cong\BA_F^\times $, we have 
\begin{align*}
O( s, f')&=  \int_{Z(\BA_E)A(E)\bsl A(\BA_E)} \int_{ H(F)\bsl   H(\BA_F)} \left(\sum_{g\in G(E)} f'(a^{-1}gh)  \right) \Omega(a)
 \eta\omega^{-1}(\kappa(h)  )|a|^sdh da\\
 &=  \int_{Z(\BA_E)A(E)\bsl A(\BA_E)}\int_{  F\bsl   \BA_F^\times}\sum_{\gamma\in \cV} \Phi_{f'}\lb(a^{-1}z)\cdot \gamma\rb
 \Omega(a)
 \eta\omega^{-1}(z  )|a|^sdz da.
\end{align*}
By Assumption \ref{freg},  only regular $\gamma$ contributes to the  inner sum. The inner sum is now 
$$ \sum_{\gamma\in E^\times\bsl  \cV_\reg/F^\times}\sum_{(a',z')\in E^\times\times F^\times}\Phi_{f'}\lb(a^{-1}z)\cdot (a'^{-1},z')\cdot \gamma\rb
$$
Identifying 
$Z(\BA_E)A(E)\bsl A(\BA_E)\cong E^\times\bsl \BA_E^\times$ via $\begin{bmatrix}a_1&0\\ 0&a_2\end{bmatrix}\mapsto a_1/a_2$, and change the order of the summation,  the equation in the lemma  follows. 
\end{proof}

 Further 
 assume that $f'$ is a pure tensor (so is $\Phi_{f'}$).
Then we have a decomposition
  \begin{align}O'( 0, f')&=\sum_{x\in F^\times-\{1\}} \sum_{v}O'(0,x,\Phi_{f',v})O(x ,\Phi_{f'}^v) , \label{decder} \end{align}    
 where the sum is over the set of places of $F$. The sum \eqref{decder} converges absolutely.
  \subsubsection{Quaternion side} \label{quaternion side}
Let  $B$ be  a quaternion algebra  over $F$.
For  $f\in  \cS_c(B^\times(\BA_F))$, define a kernel function on $B^\times(\BA_F)\times B^\times(\BA_F)$: 
\begin{equation} k (x,y)=\sum_{g\in B^\times} f(x^{-1}gy).\label{kxy}\end{equation}
 Define  a distribution  $O(\cdot)$ on $B^\times(\BA_F)$ by  assigning to $f\in \cS_c( B^\times(\BA_F))$ the integral $$O(  f)=\int_{E^\times  \bsl \BA_E^\times / \BA_F^\times} \int_{E^\times  \bsl \BA_E^\times  }
k(h_1,  h_2) \Omega(h_1)\Omega^{-1}(h_2) dh_2dh_1.$$
 This integral  converges absolutely under the following assumption.
 \begin{asmp}\label{freg'} 
Assume that  $f$ vanishes on $\BA_E^\times (B^\times-B^\times_\reg) \BA_E^\times $.
\end{asmp}
We always assume Assumption \ref{freg'}.
Then we have a decomposition \begin{align*} 
O(   f) =\sum_{x\in\ep  \Nm ( E^\times)-\{1\}}O(x,f ),
\end{align*}
where $\ep$ is as in \S \ref{matchorb}, and $$O( x,f)=\int_{  \BA_E^\times / \BA_F^\times}\int_{  \BA_E^\times  }  f(h_1^{-1}\delta(x) h_2)\Omega(h_1)\Omega^{-1}(h_2) dh_2dh_1 .$$

  \subsubsection{The case of  an incoherent transfer}
   \label{Decomposition under the pure  matching condition}
   
  Let $\BB$ be  a quaternion algebra   over $\BA_F$ such that for an odd number of places $v$ of $F$,    $\BB_v$  is the division quaternion algebra over $F_v$ (such $\BB$ is called   incoherent). 
Assume that $f\in \cS_c(\BB^\times)$ is  a pure tensor. 
Assume that   $f'\in \cS_c(G(\BA_E))$ is a pure tensor and  a transfer of $f$ (at every place), and satisfies Assumption \ref{freg}.
Let $\Phi=\Phi_{f'}$.

We  rearrange the decomposition of  $O'( 0, f')$ in \eqref{decder} according to the decomposition  $$F^\times-\{1\}=\coprod 
 \inv( B ^\times_\reg) ,$$
 where the union is  over all   quaternion algebras over $F$ containing $E$ as an $F$-subalgebra.
Let $B$ be such a quaternion  algebra, and  $x\in  \inv( B ^\times_\reg) $.  
 If $B $ and $ \BB$ are not isomorphic at more than one place, then by the transfer condition,    for every place $u$, the   product
$$O(0,x ,\Phi^u):=\prod_{v\neq u} O(0,x ,\Phi_v) $$ 
contains at least one  local component 
 with value 0 so that  $O(0,x ,\Phi^u)=0$.
 Otherwise, $B$ is as follows.
   For a place $v $ of $F$, let 
  $ B(v)$ be the unique   quaternion algebra over $F$ such that $\BB^v\cong B(v)(\BA_F^v)$ (so that $\BB_v\not\cong B(v)_v$), whose existence is assured by the Hasse principle.
   Moreover, if $B(v)$ contains $E$ as an $F$-subalgebra, then  $v$  is  nonsplit in $E$. (Indeed, otherwise, $E_v=F_v\oplus F_v$ and can not be contained in both $\BB_v$ and $B(v)_v$.)
Let   $O(x,f^v) $  be the orbital integral defined by 
regarding   $f^v$ as a function on $B^\times(\BA_F^v)$.
Then by the transfer condition, for every  place $u$ of $F$, $O(0,x,\Phi^u)\neq 0 $ only if $u=v$.   In this case, $ O(0,x,\Phi^v)=O(x,f^v)$.

To sum up, we have the following lemma.
\begin{lem}\label{decpure} Let $\Xi_\nspl $ be the  set of places of $F$ nonsplit in $E$.   There is a decomposition  
$$O'( 0,f')=\sum_{v\in \Xi_\nspl}\sum_{x\in\inv ( B(v)^\times_\reg) }  O'(0,x,\Phi_v) O(x,f^v).$$
 \end{lem}

  \begin{defn}\label{decpuredef}Let  $O'( 0,f')_v$ be the summand corresponding to $v$   in the equation in Lemma \ref{decpure}.
   \end{defn}

        \subsubsection{Globalization}
        Let  $B$ be  a quaternion algebra  over $F$. 
         \begin{lem} \label{Globalization} Let $v_1,...,v_n$ be distinct  places of $F$ such that $B_{v_i}$ is a division algebra, and  for  $i=1,...,n$, let  $ \pi_i$ be  an irreducible representation of  $B_{v_i}^\times$. 
         Assume that $\Hom_{E_{v_i}^\times} (\pi_i\otimes \Omega_{v_i},\BC)\neq 0$ (in particular, $\pi_i$ has central character $\omega_{v_i}^{-1}$). Then
         there is an automorphic representation $\pi$ of $B ^\times$ with central character $\omega^{-1},$
          and $\phi\in \pi$ such that  
         $\pi_{v_i}\cong \pi_i$ for $i=1,...,n$, and 
         $$\int_{E^\times  \bsl \BA_E^\times / \BA_F^\times}\phi(h) \Omega(h ) dh\neq 0.$$ 
 \end{lem}
 \begin{proof}
 The proof mimics the argument  in \cite{HM}. 
  We apply the relative trace formula associated to the automorphic distribution in \S \ref{quaternion side}.
  In particular, we choose a pure tensor $f$ in \S \ref{quaternion side} such that the   twisted average of $f_{v_i}$   by $\omega_{v_i}$ along the center
  is a  matrix coefficient of  $\pi_{v_i}$.  Moreover, we choose another place $v_0$ such that $\supp f_{v_0}\subset B^\times_{v_0,\reg}$.
  The rest of the proof is the same as in \cite{HM}. 
 \end{proof}
       
  \subsection{Local distributions}\label{local relative trace formula}
The   automorphic distributions $O( s, f')$ (and thus $O( 0, f'),$ $ O'( 0, f')$) and $O( f)$ above can be decomposed into  sums over automorphic representations. Each summand is further a product (when $f,f'$ are pure tensors) of local distributions. For transfers $f,f'$, the local distributions for $O( 0, f')$ and $O( f)$ can be compared. 
 We will not dive into the decomposition, but will refer to previous works. Here we only introduce the local distributions.

Let $F$ be a local field and 
 $E/F$   a separable  quadratic field extension.
Let $\Omega$ be a unitary character of 
$E^\times$, $\omega=\Omega|_{F^\times}$, and $\omega_E=\omega\circ\Nm$.
Let $\psi$ be a nontrivial additive  character of $F$, and $\psi_E=\psi\circ \Tr$ where $\Tr:E\to F$ is the trace map.
(We omit the   case  $E=F\oplus F$, and refer to \cite[6.4]{Qiu}.)

\subsubsection{Local distributions on $\GL_{2,E}$}\label{Local distributions on}
Let $\sigma $ be an infinite dimensional  irreducible unitary     representation of $G=\GL_{2,E}$ with   central character $\omega_E^{-1}$. 
Let $W(\sigma,\psi_E)$ be the $\psi_E$-Whittaker model of $\sigma$.
For $W\in W(\sigma,\psi_E)$, define the local  Rankin-Selberg period
 \begin{equation}\lambda (s,W)=\int_{E^\times}W \left(\begin{bmatrix}x&0\\ 0&1\end{bmatrix}\right )|x|_E^s\Omega(x)d  x\label{lambda}\end{equation}  
 and  the local base change period 
 \begin{equation}\cP (W)=\int_{F^\times}W \left(\begin{bmatrix}x&0\\ 0&1\end{bmatrix}\right ) \eta\omega(x)d  x\label{cp}. \end{equation}  
Then  the integral \eqref {cp}  converges (see for example \cite[p.52 Remark]{JN}).   
 Assume that $\sigma$ is tempered, then the local  Rankin-Selberg integral \eqref {lambda}  converges for $\Re(s)>-1/2$.   
Moreover, when the data are ``unramified", $\lambda (s,W)=L(1/2,\pi_E\otimes \Omega)$, see \cite[(8.1)]{Qiu}.

  Fix some (unique up to scalar)  $G$-invariant inner product (see the remark after   Proposition \ref{the proof is more important for us} below). For $f\in \cS_c(G)$, define 
$$ I_\sigma(s, f)= \sum_W \lambda (s,\pi(f)W) \overline {\cP(   W)}$$
where the sum is over an orthonormal basis of $W(\sigma,\psi_E)$


 \subsubsection{Quaternion side}
 
Let $B$ be a quaternion algebra over $F$ containing $E$.
 Let  
  $\pi $ be an irreducible unitary     representation of $B^\times$  with   central character $\omega^{-1}$.   
   For $u\in \pi$ and $v\in \wt \pi$, define    $$\alpha_\pi(u,  v)=\int_{E^\times/F^\times} (\pi(t) u,  v) \Omega(t)dt .$$  
 By abuse of notation, for $f\in \cS_c(B^\times)$, let    \begin{equation}\alpha  _{\pi } (f ) = \sum_{u} \alpha _{\pi } \lb \pi (f )u,\wt u\rb.\label{asharp}\end{equation} 
where 
  the sum is over an orthonormal   basis $\{u\}$ of $\pi $, and $\{  \wt u \}$ is the dual basis of $\wt \pi$.

\begin{prop}   \label{the proof is more important for us}
Assume that  $F$ has characteristic 0, and $\alpha_\pi\neq 0$. Then 
        for a transfer $f'\in \cS_c(G')$ of $f $, we have  \begin{align*} I_{\pi_E}(0,f')=c_\pi  \alpha_{\pi}(f ),\end{align*} 
        where $c_\pi$ is an explicit nonzero constant.\end{prop}
        \begin{proof}  
      The proposition follows from the globalization argument in  \cite[Proposition 5.7.1]{Beu}   and the proof of \cite[Proposition 5]{JN} (see also \cite[Proposition 6.3.1]{Qiu}).   
                \end{proof}
The constant $c_\pi$ in Proposition \ref{the proof is more important for us} 
depends on the choices of  the $G$-invariant inner product on $W(\sigma,\psi_E)$ and 
 the measures. In the non-archimedean case, it  can be read from   the equation in \cite[Proposition 6.3.3]{Qiu}.
It has the same shape in the  archimedean case by the same proof.
  We will not need the constant $c_\pi$ explicitly, rather refer to \cite{Qiu} later.
 
 
     \subsubsection{A quaternionic test function at $\infty$}\label{oinfsec}
       Let $\BD$ be the unique division quaternion algebra over $\BR$, and fix  an embedding  $\BC\incl \BD$. 
In the notation of \ref{matchorb}, we choose
 $\ep=-1$. (Then $\BD=\BC\oplus\BC j$ is the usual way of expressing the Hamilton quaternion algebra.)
       
       Let $\rho_{2k}$ be the irreducible $(2k-1)$-dimensional representation of $\BD^\times $ with trivial central character whose Jacquet-Langlands correspondence to $\GL_2(\BR)$ is the holomorphic discrete  series of weight $2k$. Fix an invariant inner product  $(\cdot,\cdot)$ on $\rho_{2k}$, and let $e$ be the unique unit vector which is $\BC^\times$-invariant (see \cite{Gro1}). 
  Let $f_\infty\in \cS_c(\BD^\times)$ be such that    \begin{equation} 
         \label{oinf1}\int_{\BR^\times} f_\infty(zg) d  z= (\rho_{2k}(g)e,e).   \end{equation}


 Then by the discussion in \cite[6.2]{Qiu}, we have the first equality of the following
       \begin{align*} \begin{split} 
  \alpha_{\rho_{2k}}(f_\infty)=\Vol(\BC^\times/\BR^\times)\int_{\BD^\times} f_\infty(g) (\rho_{2k}(g)e,e) dg    =\frac{\Vol(\BC^\times/\BR^\times)}{d_{\rho_{2k}}}.
          \end{split}     \end{align*}
        Here $d_{\rho_{2k}}$ is the formal degree of ${\rho_{2k}}$ and the second equality is by definition. It is a standard fact that 
        $  d_{\rho}\Vol(\BD^\times/\BR^\times)=\dim \rho$ for any    irreducible  representation $\rho$ of the compact group $\BD^\times/\BR^\times$.
      So   \begin{align}\label{eq:alinf} \begin{split} 
  \alpha_{\rho_{2k}}(f_\infty)=\frac{\Vol(\BC^\times/\BR^\times)\Vol(\BD^\times/\BR^\times) }{2k-1}.
          \end{split}     \end{align}

        Let  $P_{k-1}(t)$ be the $(k-1)$-th Legendre polynomial, i.e. the multiple of 
$\frac{d^{k-1}}{dt^{k-1}}(t^2-1)^{k-1}$ whose value at 1 is 1.
Let $\Omega=1$.
       By \cite[Lemma 4.14]{FW} and a direct computation, we have 
              \begin{equation} 
         \label{oinf}O(\delta, f_{\infty})=   P_{k-1}\lb\frac{1+\inv (\delta)}{1-\inv(\delta)}\rb  {\Vol(\BC^\times/\BR^\times)^2} 
       \end{equation}
       for $\delta\in \BD^\times_\reg$.  We remind the reader that  there is a sign mistake in the statement of \cite[Lemma 4.14]{FW}, which is easy to   spot by checking \cite[Lemma 4.8]{FW} (the proof of \cite[Lemma 4.14]{FW} is based on  \cite[Lemma 4.8]{FW}).
       
                     \subsection{Local comparison I}\label{Height distribution}    
          
 We compare the    height  distribution  $$H(f )  = \pair{  Z_{\Omega,f },z_{\Omega^{-1}}} $$
 on $\GL_2(\BA^\infty)$ with the automorphic distribution on $\GL_{2}(\BA_K)$.   The conclusion is  an arithmetic relative trace identity  \eqref{ATI}.
 
 The proof of \eqref{ATI} consists of 10 steps, each is a subsubsection. 
We briefly sketch them as follows:
 \begin{itemize}
 
  \item[(1)] recall/set up notations;
   \item[(2)] decompose $H(f ) $ into a sum of local heights $H(f )_v $'s over places of $\BQ$; 
   \item[(3,4)]  express the local height  $H(f )_p ,p<\infty,$ into a sum of  intersection numbers using a description of the integral model;  
   \item[(5$\&$6)]  for $p$ nonsplit in $K$, decompose the intersection numbers using formal uniformization of the integral model, and rewrite $H(f )_p$ in a  form like a sum of orbital integrals (with the local orbital integrals at $p$ replaced by intersection numbers);
     \item[(7$\&$8$\&$9)] 
     choose test functions and finish the comparison for $p<\infty$;
         \item[(10)]   deal with the infinite place.

\end{itemize}

  \subsubsection{Notations} \label{Notations}
   
We will need the following   notations for quaternion algebras.  
 For a quaternion algebra $B$ over a field $F$, we always use $q$ to denote the  
 reduced norm map. (If $B$ is the matrix algebra, then $q $ is the determinant, and we will consistently use $q$ to replace $\det$ from now on.) 
 Let $E$ be a separable quadratic field extension of $F$.
 For an embedding of $E\incl B$, 
 define an invariant  on $B^\times$ for the bi-$E^\times$-action
        \begin{align*}
        \lambda: B^\times&\to \BQ,\\
        \lambda(a+bj)&=\frac{q(bj)}{q(a+bj)},\nonumber\end{align*}
        where the notation is the same as in \S \ref{matchorb}.
        The relation with the invariant  defined in \S \ref{matchorb} is 
  $$ \lambda(\delta )=\frac{ -\inv (\delta)}{1-\inv(\delta)}  .$$

  We will use quaternion algebras over $\BQ$. For a  finite place $p$ of $\BQ$, let $B(p)$ be the unique   quaternion algebra over $\BQ$ such that $B(p)_v$ is division only for $v=p$ and $v=\infty$.
  Let $B(\infty) $ be the matrix algebra $ \RM_{2,\BQ}$.

Also need Legendre functions.
 Let $P_{k-1}(t)$ be the $(k-1)$-th Legendre polynomial, i.e. the multiple of 
$\frac{d^{k-1}}{dt^{k-1}}(t^2-1)^{k-1}$ whose value at 1 is 1.   
Let  $Q_{k-1}$ be the Legendre function of the second kind, i.e.,
$$Q_{k-1}(t)=\int_{u=0}^\infty (t+\sqrt{t^2-1}\cosh u)^{-k} du,\  t>1.$$
Then   (see \cite[p. 294, (5.7)]{GZ}) \begin{equation}\label{PQrelation}Q_{k-1}(t)=\frac{1}{2}P_{k-1}(t)\log\frac{t+1}{t-1}+(\mbox{polynomial in } t),\end{equation}
and \begin{equation}\label{Qinfty} Q_{k-1}(t)=O(t^{-k}),\ \text{for }t\to \infty.\end{equation}

Finally,
 let  $N$ be a positive integer  as in \S \ref{FixN}  and let $ U=U(N)$, the corresponding principal congruence subgroup of $\GL_2(\BA^\infty)$. Let us rewrite $H(f)$.
Let $$c_K=  \Vol(K^\times\bsl \BA_K^{ \times } /\BA^{ \times } )\Vol(K^\times\bsl \BA_K^{ \times }/\BR^\times  ).
$$ 
 By definition and   a simple change  of variable $h=t_1^{-1}t_2g$, we have   
  \begin{equation}    \begin{split}  \label{eq:H1}
H(f)_ p
=   &  \frac{ 1}{c_K}     \int_{ K^\times\bsl \BA_K^{\times}/\BA^{ \times}}\int_{  K^\times\bsl  \BA_K^{ \times}/\BR^\times}       \sum _{h\in  \GL_2(\BA^\infty)/U}    f   (h )     \\
 &     \pair{ {Z_{t_1 h }} ,  {Z_{t_2}}}_{Y(N)}\Omega^{-1}( t_{2 })\Omega(t_{1 })   dt_{2 }d t_{1 }\\
 =   &  \frac{ 1}{c_K}     \int_{ K^\times\bsl \BA_K^{\times}/\BA^{ \times}}\int_{  K^\times\bsl  \BA_K^{ \times}/\BR^\times}       \sum _{g\in  \GL_2(\BA^\infty)/U}    f   (t_1^{-1}t_2g )      \\
 &     \pair{ {Z_{t_2g }} ,  {Z_{t_2}}}_{Y(N)}\Omega^{-1}( t_{2 })\Omega(t_{1 })   dt_{2 }d t_{1 }.
  \end{split}     \end{equation}

\subsubsection{Local-global decomposition}

 Let $S$ be a finite set of finite places of $\BQ$ which contains all places of $\BQ$   ramified in $K$ or  ramified for $\Omega$, and all prime factors of $N$. 
   Fix $f_S=\otimes_{p\in S} f_p \in \cS_c(\GL_2(\BA_S))$ where each $f_p \in \cS_c(\GL_2(\BQ_p))$ is right $U_p$-invariant.
Let   $f=f_S \otimes f^S$ where   $f^S\in \cH^S$.  Such test functions will be enough for our   proof of Theorem    \ref{stronghtder} in \S \ref{proofs}.

For a given $f_S$, let us consider the defining fields  and height pairings of the involved CM cycles as $f^S\in \cH^S$ varies.
Let $F$ be   the finite abelian extension of $K$ corresponding to 
\begin{equation*}\BA_K^{\infty,\times}\bigcap \bigcap_{g_S\in \supp f_S} g_SUg_S^{-1}\label{CFT}\end{equation*}
by the class field theory.
Then by the CM theory, 
all CM cycles
 $Z_{tg_S}$'s, $t\in \BA_K^{\infty,\times}$ and $g_S\in \supp f_S$, are defined over   $F $.  
Moreover,  for $g^S\in  \GL_2(\BA^{S,\infty})$, $Z_{tg_Sg^S}$ is defined over a finite extension $F'$ of $F$ unramified at 
$S$.
So we can define height pairing  $\pair {Z_{tg_Sg^S}, Z_{t'}}$, $t'\in \BA_K^{\infty,\times}$,
as in \S \ref{ and the height pairing}  using the integral model  $\cY'$ of $Y(N)$ over $\Spec \cO_F$ which is smooth  outside $S$.

Fix an  $f_S$ satisfying  the following  assumption (we will  further specify such an $f_S$ in \S \ref{proofs}).     

  \begin{asmp} \label{asmp2}   For every $p\in S$,   we have 
   $\supp f_p\subset \GL_2(\BQ_p)_\reg$, the regular locus for the $K_p^\times\times K_p^\times$-action.  
  
\end{asmp} 

Then   $Z_{\Omega,f_S} $ and $Z_{\Omega^{-1}} $ do not intersect.
So we have the decomposition   of the height pairing  into local heights over places of $\BQ$: 
 \begin{align}  H(f) =\sum_{v  }H(f)_v .\label{Hdec}\end{align}     
 Here the local height at $v$ comes from all places of $F$ over $v$.   
  
    \subsubsection{Integral CM cycles and desingularization}
   Let $\cY'\to \cX'$ be the  integral model of $Y(N)\to X(N)$ as in \S \ref{Kintmodel}. 
   Recall that a cycle on   $\cY'$ or $ \cX'$   is called    vertical if it is supported on the special fibers. It is called
   horizontal if it has no vertical component, equivalently, it is the Zariski closure of its generic fiber.
 For $t\in \BA_K^{\infty,\times}$,  
  let $ \cZ_{t}'$ be  the  integral CM cycle  on $\cY'$ as in
 \ref{ and the height pairing} and we   describe it as follows. 
 
  Let $p$ a prime numer and $F_p=F\otimes_{\BQ}\BQ_p$, which is a product of finite extensions of $\BQ_p$. Let $\cO_{F_p^\ur}$ be the product of the integer rings of the completed maximal unramified extensions of the components of $F_p$. We work on $\cO_{F_p^\ur}$ for later convenience.
  Let     $P_{t,0}$ (resp. $P'_{t,0}$)  be a reduction of the CM point  $P_{t}$   in    $\cX_{\cO_{F_p^\ur}}$ (resp., $\cX'_{\cO_{F_p^\ur}}$) .  Let $\cY'_{P'_{t,0}}$ be the fiber of $\cY'$ over $P'_{t,0}\in \cX'_{\cO_{F_p^\ur}}$. Then as we have seen in \S \ref{Kintmodel},
  the natural map $\cY'\to \cY $ gives the identification 
  \begin{equation}\label{CY'CE}
  \cY'_{P'_{t,0}} =\cE|_{P_{t,0}}^{2k-2} .
\end{equation}
  Let $Z_{t,0}=Z\lb \cE|_{P_{t,0}} \rb$ be the CM cycle as in
  \ref{CM cycles}.
 Then 
\begin{equation}\label{Local structure at finite places}
  \cZ_{t}'|_{\cY'_{\cO_{F_p^\ur}}}=\ol {Z_{t}}+Z_{t,0}\times D_t,
  \end{equation}  where $\ol {Z_{t}}$ is  the  horizontal part of $ \cZ_{t}'$   and $D_t $ is the base change to $\cO_{F_p^\ur}$ of a divisor  of $\cX'$
 supported on  the exceptional divisor of the desingularization $\cX'\to \cX_{\cO_F}$ at  $p$. 
 The special fiber  of $\ol {Z_{t}}$ sits in $\cY'_{P'_{t,0}} =\cE|_{P_{t,0}}^{2k-2}$, and  is identified as  $Z_{t,0}$.
 We have a similar description for $Z_{tg}$.
Then we have   \begin{align}     H(f)_ p
=   i(f)_p+j(f)_p ,  \label{Hdec2},
    \end{align}
 where  $i(f)_p$ comes from the intersections of horizontal cycles,
  $j(f)_p$ comes from the intersections between horizontal and vertical cycles.

 As   $\cX$ is smooth over $\BZ$ at  ordinary loci (i.e., the loci of points representing ordinary elliptic curves over finite fields,
 see \cite{KM} or \cite{Zha01}),
the desingularization of $\cX_{\cO_F}$ (to get $\cX'$) only happens at supersingular locus.  
   So vertical parts of integral CM cycles are supported on   supersingular loci.

  \subsubsection{Split $p$}\label{split0}
   Let $p$ be split in $K$. 
   Under Assumption \ref{asmp2},  the underlying CM points of $Z_{\Omega,f_S} $ and $Z_{\Omega^{-1}} $  do not have  any  common reduction (see  \cite[8.4]{YZZ} or \cite[9.5.1]{Qiu}). So  $i(f) _p=0$.
 Since  the CM points have ordinary reductions at such $p$ (see for example \cite{GZ}) while  the   vertical parts  of $\cZ_g'$
is supported over supersingular loci, 
  $j(f)_p=0$.

\subsubsection{Nonsplit $p$, $i$-part}\label{ipartm} 
 Let $p$ be nonsplit in $K$. 
 By definition and \eqref{eq:H1}, we have   
  \begin{align*}  
   i(f)_ p
=   &  \frac{ 1}{c_K}     \int_{ K^\times\bsl \BA_K^{\times}/\BA^{ \times}}\int_{  K^\times\bsl  \BA_K^{ \times}/\BR^\times}       \sum _{g\in  \GL_2(\BA^\infty)/U}    f   (t_1^{-1}t_2g )         (\ol {Z_{t_2g }} \cdot\ol {Z_{t_2}}) _p\Omega^{-1}( t_{2 })\Omega(t_{1 })   dt_{2 }d t_{1 }\cdot \log p
      \end{align*}
Here $\ol {Z_{t_2g }} $ and $\ol {Z_{t_2}}$
are horizontal parts of $ {\cZ'_{t_2g }} $  and ${\cZ'_{t_2}}$ in $\cY'_{\cO_{F_p^\ur}}$ and $ (\ol {Z_{t_2g }} \cdot\ol {Z_{t_2}}) _p$  is $1/[F:\BQ]$ times their intersection number.

 To proceed, we need some preparations for computing  $ (\ol {Z_{t_2g }} \cdot\ol {Z_{t_2}}) _p$.  
 Note that  the CM points have supersingular reductions at such $p$ (see for example \cite{GZ})
  Let $\cM_{U_p}$ be the supersingular Lubin-Tate formal deformation space of level $U_p$ \cite[9.3.2]{Qiu}, which is defined over the integer rings of the completed maximal unramified extension of $\BQ_p$. Then  the corresponding base change of the formal  neighborhood of the supersingular locus in $\cX$ is given as follows:
 $$  B(p)^\times \bsl \cM_{U_p}\times B(p)^\times(\BA ^{p,\infty}) /U^p.$$ 
  Let $\cM_{U_p}'$ be  the minimal desingularization of the base change of $\cM_{U_p}$ to   $\cO_{F_p^\ur}$.

Let $t\in K_p^\times $ act on $B(p)_p^\times$  by  right multiplication  by $t^{-1}$, and act on $\GL_2(\BQ_p)$  via  left multiplication by $t$.
 Then there is a natural map from the contracted product  $B(p)_p^\times\times ^{K_p^\times}\GL_2(\BQ_p)$ to  $\cM_{U_p}$ parametrizing CM liftings \cite[5.5]{Zha01}. 
So, we have a map from $B(p)_p^\times\times ^{K_p^\times}\GL_2(\BQ_p)$ to  $\cM'_{U_p}$ by strict transform.
Abusing notation, we  use $(\delta,g)\in B(p)_p^\times\times ^{K_p^\times}\GL_2(\BQ_p)$    to denote the corresponding CM point  in   $\cM'_{U_p}$.
The   multiplicity function $m_p(\delta,g)$    on $B(p)_p^\times\times ^{K_p^\times}\GL_2(\BQ_p)-\{(1,1)\}$  is defined as $1/[F:\BQ]$ times  the intersection multiplicity between $(\delta,g)$ and $(1,1)$
 (see   \cite[Definition 9.3.6]{Qiu} or   \cite[8.2.1]{YZZ}).   It is a smooth function.
 Note that $m_p$ in fact depends on $U_p$ though we do not indicate this dependence in the notation.

  We recall the following properties of the multiplicity function (see    \cite[9.3.3]{Qiu}).
  \begin{lem}\label{mnonv}
 (1) If $m_p(\delta,g)\neq0$, then $q(\delta)q( g)\in q(U_p)$.
 
 (2) We have $  m_p(\delta^{-1} ,g^{-1} )= m_p(\delta,g ).  $
 
 \end{lem}

As in \S \ref{Hecke action}, first consider  $g\in\BG$. 
         Let $\psi_g:\cE|_{P_{t_2g}}=E|_{P_{t_2g}}\to \cE|_{P_{t_2}}=E|_{P_{t_2}}$  be the unique morphism that fits into the right downward arrow in \eqref{gisog}  with $A=\cE|_{P_{t_2g}}$ there.
(Here we use $g$ rather than $g^{-1}$ due to that $P_{t_2g}$ is defined ``via the left action by $g^{-1}$", see \S \ref{CM cycles}.) Let $\psi_{g,0}$ be its reduction. 
Consider
$$\Hom(\cE|_{P_{t_2,0}},\cE|_{P_{t_2g,0}})_\BQ \cong \End(\cE|_{P_{t_2,0}} )_\BQ \cong B(p)$$
where the first isomorphism is $\phi\mapsto \psi_{g,0} \circ \phi$ and the second is well-known.
 For $\delta\in   B(p)^\times$, let $\phi(\delta)\in \Hom(\cE|_{P_{t_2,0}},\cE|_{P_{t_2g,0}})_\BQ$ be its preimage.  
  By    \cite[Proposition 3.3.1]{Zha97} and \cite[Lemma 9.3.9]{Qiu}, if  $P_{t_2g}\neq P_{t_2}$, then 
 \begin{align} (\ol {Z_{t_2g }} \cdot\ol {Z_{t_2}})_p=(-1)^{k} \sum_{\delta\in   B(p)^\times} m_p(t_{2,p}^{-1}\delta t_{2,p},g_p ^{-1}) 1_{ U^p}((t_2^{p,\infty} g ^p)^{-1}  \delta t_2 ^{p,\infty}) \lb \phi(\delta)^{k-1,*} Z_{t_2g,0}\cdot Z_{t_2,0}\rb_{\cY'_{P_{t_2,0}}}. \label{011181}
    \end{align}  Here   $\lb \phi(\delta)^{k-1,*} Z_{t_2g,0}\cdot Z_{t_2,0}\rb_{\cY'_{P_{t_2,0}}}$ is the intersection number on $\cY'_{P_{t_2,0}}$ (understood by the isomorphism \eqref {CY'CE}).
 By  \cite[4.5.3]{Zha97} (for  the CM cycle $Z_{t_2,0}$, $\sigma=\id$ and $\fa=(1)$ in the notations in loc. cit., and the proof is purely over $\ol{\BF_p}$),
we have 
\begin{equation}(\delta^{k-1,*} Z_{t_2,0}\cdot Z_{t_2,0})_{\cY'_{P_{t_2,0}}}= \lb-q(\delta)\rb ^{k-1}P_{k-1}(1-2\lambda(\delta )).\label{ZZtt}\end{equation}
Then since $\psi_{g,0}^{k-1,*} Z_{t_2,0}=|q(g)|^{-(k-1)}Z_{t_2g,0}$ (the reduction of the corresponding equation on the generic fiber), 
$$(\phi(\delta)^{k-1,*} Z_{t_2g,0}\cdot Z_{t_2,0})_{\cY'_{P_{t_2,0}}}= \lb-q(\delta)\cdot |q(g)| \rb ^{k-1}P_{k-1}(1-2\lambda(\delta )) $$

By Lemma \ref{mnonv},  the nonzero contribution  in \eqref{011181} is only from $\delta $ such that  $q(\delta)\in q(t_2^{ \infty} g    t_2 ^{\infty,-1})q(U)$.   Since $B(p)_\infty$ is division,   $q(\delta)>0$. So $q(\delta)=|q(g)|^{-1}$.
Thus, we have
   \begin{align} \begin{split} 
   i(f)_ p
=   &  \frac{-1}{  c_K}     \int_{ K^\times\bsl \BA_K^{\times}/\BA^{ \times}}\int_{  K^\times\bsl  \BA_K^{ \times}/\BR^\times}       \sum _{g\in  \GL_2(\BA^\infty)/U}    f   (g )  \sum_{\delta\in   B( p) ^\times }
     \\   &        P_{k-1}(1-2\lambda(\delta ))  m_p( t_{1,p } ^{-1}  \delta t_{2,p }  ,   g_p^{-1} )  1_{ U^p}((t_1^{p,\infty} g ^p)^{-1}  \delta t_2 ^{p,\infty})   \Omega^{-1}( t_{2 })\Omega(t_{1 })   dt_{2 }d t_{1 }\cdot \log p
 \label{11181}
  \end{split}     \end{align}
  Note that  $m_p( \cdot   ,   g_p^{-1} ) $  only depends on $g_p U_p$ by  Lemma \ref{mnonv} (2). So the sum  over $g$ is well-defined.
  

 
Let $f_\infty$ be  as in \eqref{oinf1}. Changing the order of summation in \eqref{11181}, we formally have 
 \begin{align} \begin{split}
   i(f)_ p
     =&  \frac{-1}{c_K} \sum_{\delta\in K^\times\bsl B( p)_\reg^\times/K^\times} O(\delta,f^p\otimes f_\infty)  \\ 
      &     \int_{ K^\times_ p/\BQ_ p^\times}\int_{ K^\times_ p}    \sum _{g\in  \GL_2(\BQ_p)/U_ p} f_{ p}(g )    m_p( t_{1 } ^{-1}  \delta t_{2 }  ,   g^{-1} )     \Omega_ p^{-1}( t_{2 })\Omega_ p(t_{1 })   dt_{2 }d t_{1 }\cdot \log p\label{1120}.
  \end{split}     \end{align}
Here we use the orbital integrals of $ f_\infty$ computed in \eqref{oinf}. Moreover, the equality  is verified by   Assumption \ref{asmp2} and 
 Fubini's theorem, see \cite[Lemma 9.3.10]{Qiu}

    \subsubsection{Nonsplit $p$, $j$-part}\label{nonj}
Let $p\in S$ which is nonsplit in $K$.
Let  $\cP$ be   the union of the exceptional divisors of $\cX'$ over $p$.
Let $\cV$ be 
the exceptional divisor of $  \cM_{U_p}'$, which  inherits  a $B(p)_p^\times$-action from the $B(p)_p^\times$-action on $ \cM_{U_p}'$. Then      \begin{equation} \cP \cong B(p)^\times\bsl   \cV \times \GL_2(\BA^{p,\infty} )/ U^p.\label{vsing}\end{equation} 
For a  vertical divisor  $ C $ of  $  \cM_{U_p}'$ supported  on $ \cV$ and $g\in \GL_2(\BA^{ p,\infty })$, let $[C,g]$ be the corresponding divisor of $\cP$   via \eqref{vsing}. 
Define a function $l_{ C}$  on $ B(p)_p^\times\times ^{K_p^\times}\GL_2(\BQ_p)$ 
  by  letting $l_{ C}(\delta,g)$ be the  intersection number  of $C$ and   $(\delta,g)$   in $  \cM_{U_p}'$.  
 
 There is a natural map  $  \cM_{U_p}\to \BZ$ given by the degree of the quasi-isogeny in the definition of  the deformation space $\cM_{U_p}$ \cite[9.3.2]{Qiu}. 
 Thus we have a map $  \cM'_{U_p}\to \BZ$.
It induces    a   map    $ \cV\to \BZ $.
  \begin{lem}[{\cite[Lemma 9.3.19]{Qiu}}]  \label{9123}The function $l_C$ satisfies the following properties:
  \begin{itemize}
 
  \item[(1)] for $h \in B(p)_p^\times\times ^{K_p^\times}\GL_2(\BQ_p)$, $l_{C}(h)\neq 0$ 
  only if  the image of the support of
  $  C$ in $\BZ$  by the map $       \cV\to \BZ$  contains 
  the image of $h$  by $B(p)_p^\times\times ^{K_p^\times}\GL_2(\BQ_p) \to \BZ$;
  
 \item[(2)]  for $b\in B(p)_p^\times $ and $h \in B(p)_p^\times\times ^{K_p^\times}\GL_2(\BQ_p) $, $l_{bC}(bh)=l_C(h)$; 
  
 \item[(3)]  $l_C$    is smooth.
\end{itemize}
     \end{lem} 
    
     Let $D_t$ be as in \eqref{Local structure at finite places}.
Write $D_1$ as $[C,1]$ for a    divisor  $ C $ of  $  \cM_{U_p}'$ supported  on $ \cV$. Then $D_{t_2}=[t_{2_p}C,t_2^p] $.
  By Lemma \ref{9123} (2),  a similar process of unfolding  as in \eqref{11181}
 and refolding
as in \eqref {1120}, we have
  \begin{align} \begin{split}
   j(f)_ p
      = & \frac{-1}{c_K} \sum_{\delta\in K^\times\bsl B( p)_\reg^\times/K^\times} O(\delta,f^p f_\infty)  \\ 
      &     \int_{ K^\times_ p/\BQ_ p^\times}\int_{ K^\times_ p}    \sum _{g\in  \GL_2(\BQ_p)/U_ p}f_{ p}(g )    l_C( t_{1 } ^{-1}  \delta t_{2 }  ,   g^{-1} )     \Omega_ p^{-1}( t_{2 })\Omega_ p(t_{1 })   dt_{2 }d t_{1 }\cdot \log p\label{1120j}.
  \end{split}     \end{align}

     \subsubsection{Test functions}\label{Mps} 
  For $p\in S$,  besides  Assumption \ref
  {asmp2}, further 
  assume  that  
  $q(\supp f_p )\subset \Nm(  K_p^{ \times}),p\in S$. Then by Lemma \ref{qcond}, we can use Lemma \ref{Xue} (1).
     Let $f'_p $  be a transfer of $f_p$ given by Lemma \ref{Xue} (1).
     
     For $v=\infty$, $q(\supp f_\infty )\subset \Nm(  K_\infty^{ \times}) $ holds.
   So  we can use Proposition   \ref{transfer} (for now Proposition   \ref{transfer} (1) is enough, and we will specify $f'_\infty$ in \S \ref{locinf}).
 
 Assume that $f^S$ has an unramified transfer  $ f'^S\in \cS_c\lb\GL_2\lb\BA_K^{S,\infty}\rb\rb$  given by   the fundamental lemma \cite[Proposition 7.3.1]{Qiu} (see also   \cite[Section 4]{Jac87}  \cite[Proposition 3]{JN}).
Explicitly, let $\cH^S_E$ be unramified-outside-$S$ Hecke algebra of $\GL_{2,E}$ (defined in the same way as $\cH^S$).
     So we have the base change homomorphism $\bc:\cH_E^S\to \cH^S$
 (see \cite[Section 1]{Lan}).  
  Then $f^S$ is a multiple of  $\bc(f'^S)$, depending on the measures.
  Moreover, as  $q(\supp \bc(f'^S )\subset \Nm(  \BA_K^{S\cup{\infty}, \times}) ,$ 
  $q(\supp  (f  )\subset \Nm(  \BA_K^{ {\infty}, \times}) .$ 
  
   Let $$f'=\bigotimes_{p\in S} f'_p\bigotimes f'_\infty\bigotimes f'^S.$$
   We will further specify the consequences of this choice of test function in the later paragraphs.
     \subsubsection{Matching for $p\not\in S$} 

  The full arithmetic fundamental lemma (proved in \cite[Proposition 10.2.3]{Qiu})
equates  $$-2\int_{ K^\times_ p/\BQ_ p^\times}\int_{ K^\times_ p}    \sum _{g\in  \GL_2(\BQ_p)/U_ p} f_{ p}(g )    m_p( t_{1 } ^{-1}  \delta t_{2 }  ,   g^{-1} )     \Omega_ p^{-1}( t_{2 })\Omega_ p(t_{1 })   dt_{2 }d t_{1 }\cdot \log p$$
  with the derivative $O'(0,\Phi_{f'_p})$ of  the local orbital integral  for ${f'_p}$. 
Combining it with \eqref{1120}, we have the following lemma. (Note that   the $j(f)_p=0$  here.)
 \begin{lem} \label{111}  For a finite place $p\not \in S$ and nonsplit, we have 
$$2  H(f)_ p=\frac{1}{c_K}O'(0, f')_p,$$
where the right hand side is defined  in Definition \ref{decpuredef}.
  \end{lem} 
  
  \subsubsection{Coherence for $p\in S$ }
   For $p\in S$, we do not need anything like the arithmetic fundamental lemma, as the local terms are ``coherent" (a term borrowed from \cite{YZZ}, whose meaning is indicated in the proof of the lemma below).
  \begin{lem}\label{112}
 For  $p \in S$ and nonsplit,  there exists $f_{(p)} \in \cS_c(B(p)_p)$ such that 
  $$ 2H(f)_ p-\frac{1}{c_K}O'(0, f')_p=  O(f_{S-\{p\}}\otimes f_{(p)} \otimes f^S\otimes f_\infty).$$
  \end{lem}
  \begin{proof} 
  Choose  $e \in  \cS_c(B(p)_p) $  by Lemma \ref{Xue} (2) so that  $O(x,e)=O'(0,x,\Phi_{f'_p})$. 
Let $$f_{p} (\delta)=  -\frac{1}{c_K}\lb 2\sum _{g\in  \GL_2(\BQ_p)/U_ p}f_{ p}(g )  \lb  ( m_p+l_C)( \delta ,   g^{-1} ) \rb
  +e(\delta)\rb$$
  which is  well-defined  on $B(p)_p$  by Assumption \ref{asmp2}. 
    It is automatically in $ \cS_c(B(p)_p)$  (so is ``coherent") by Lemma \ref{mnonv} and Lemma \ref{9123}.
    Then the lemma follows from \eqref{1120} and \eqref{1120j}.
  \end{proof}
 \subsubsection{Local height at $\infty$}\label{locinf}
 The comparison at $\infty$ is  similar to the comparison at $p$ nonsplit in $K$  and not in $S$. Namely, we first write   $H(f)_ \infty $
in a  form like a sum of orbital integrals (with the local orbital integral at $\infty$ replaced by  the Legendre function).
Besides,  we need a truncation process.
 
 Recall the complex uniformization \eqref{cplx} of the noncuspidal locus $X(N)^\circ$ of $X(N)$:
 $$X(N)^\circ(\BC)\cong B( \infty)_{>0} \bsl  \BH\times B(\infty)(\BA^\infty)^\times /U(N)$$
 Here $B(\infty) $ is the matrix algebra  $ \RM_{2,\BQ}$,    $B(\infty)_{>0}\subset B(\infty)$ consists of elements with positive norms (i.e. positive determinants), and $ \BH\subset \BC$ is the upper half plane. 
For $(z_1,g_1),(z_2,g_2)\in X(N)^\circ(\BC)$, let
 \begin{align} G_k((z_1,g_1),(z_2,g_2))=-\sum_{\delta \in B( \infty) _{>0}}Q_{k-1}(d(z_1,\delta z_2))1_{U}(g_1^{-1}\delta g_2),\label{Hdecinf}\end{align}
where $d$ is the hyperbolic distance.  
The sum is absolutely convergent by \eqref{Qinfty}.
If $(z_1,g_1),(z_2,g_2)$ are  distinct CM points, the local height pairing between the CM cycles over them 
 defined   as in \S \ref { and the height pairing}  is $G_k((z_1,g_1),(z_2,zg_2))$, see  \cite[Proposition 3.4.1]{Zha97}.
 
To use  \eqref{Hdecinf}  to compute  $  H(f)_ \infty$,
  we need more details about the terms on the right hand side of \eqref{Hdecinf}. 

 First, the hyperbolic distance has the following formula.  
Recall that $B(\infty) $ is the matrix algebra $ \RM_{2,\BQ}$, and we have fixed an embedding $K\incl   \RM_{2,\BQ}$ of $\BQ$-algebras to define CM points and cycles in \ref{CM cycles}.   Let  $z_0\in \BH$  be the unique  fixed point under the action of $K_\infty^\times\subset B( \infty) _{\infty,>0}$. Let  $\delta\in B( \infty) _{\infty,>0}\cap  B(\infty)_{\infty,\reg}^\times$.
  The hyperbolic distance between   $z_0\in \BH$ and $\delta z_0$, 
 is      $$d(z_0,\delta z_0)=1-2\lambda(\delta),$$
 where $\lambda$ is defined with respect to the embedding  $K_\infty^\times\subset B( \infty) _{\infty}$
 see   \cite[8.1.1]{YZZ}.

Second,  we truncate  $Q_{k-1}$, since $Q _{k-1} $ does not vanish  near infinity (see \eqref{Qinfty}).    Note that $$\lambda\lb B( \infty) _{\infty,>0}\cap  B(\infty)_{\infty,\reg}^\times\rb =(-\infty,0)$$ so that $1-2\lambda$ takes values in $(1,\infty)$.
   Let $\phi$ be a smooth function on $[1,\infty)$  with compact support such that $\phi([1,2])=\{1\}$. Let 
   $Q^\phi_{k-1}=Q_{k-1} \phi$, and  $$G_k^\phi((z_1,g_1),(z_2,g_2)) =-\sum_{\delta \in B( \infty) _{>0}}Q^\phi_{k-1}(d(z_1,\delta z_2))1_{U}(g_1^{-1}\delta g_2).$$
For $Z=\sum _{i=1}^n a_nz_n$ be a divisor on $ X(N)^\circ(\BC)$, let $$D_Z(x)=\sum _{i=1}^n a_n \lb  G_k(x,z_n) -G_k^\phi(x,z_n)\rb.$$
   Then $D_Z$ is a smooth bounded function on $X(N)^\circ(\BC)$ by \eqref{Qinfty}.

   Now we  can use  \eqref{Hdecinf}  to compute  $  H(f)_ \infty$.  
Under the complex uniformization (see  Remark \ref
{opposite}) $$
 X_{\infty}^\circ(\BC)\cong \GL_2(\BQ)_{>0} \bsl  \BH\times \GL_2 (\BA^\infty)^\times,
$$
we may take 
$P_0\in X^\circ_\infty(\ol \BQ)^{K^\times}$ to be $(z_0,1)$. Then 
under the complex uniformization    \eqref{cplx}  of $ X(N)^\circ(\BC)$, we have  $  P_{g}=(z_0,g)$. 
 By     a similar process of unfolding  as in \eqref{11181} 
(now we use the first equation of  \eqref{eq:H1} only)
 and refolding as in \eqref {1120}, we have
  \begin{equation}\label{Hff}
   \begin{split}
  H(f)_ \infty =   &  \frac{ 1}{c_K}     \int_{ K^\times\bsl \BA_K^{\times}/\BA^{ \times}}\int_{  K^\times\bsl  \BA_K^{ \times}/\BR^\times}       \sum _{g\in  \GL_2(\BA^\infty)/U}    f   (g )       G_k( P_{t_1g},P_{t_2})   \Omega^{-1}( t_{2 })\Omega(t_{1 })   dt_{2 }d t_{1 } 
  \\
     = & \frac{-1}{c_K} \sum_{\delta\in K^\times\bsl \lb B( \infty)_\reg^\times\cap B( \infty)_{>0}\rb/K^\times} O(\delta,f )      \Vol(\BC^\times/\BR^\times)^2 Q _{k-1}^\phi\lb d(z_0,\delta (x)z_0)\rb\\
   &+  \frac{1}{ \Vol(K^\times\bsl \BA_K^{ \times } /\BA^{ \times } )\Vol(U) }\int_{ K^\times\bsl \BA_K^{\times}/\BA^{ \times}}     \int _{ \GL_2(\BA^\infty)}    f   (g )      D_{Z_\Omega}  ( P_{t_1g} )  dg 
   \Omega(t_{1 })    d t_{1 } .
  \end{split}     \end{equation}

 Finally, we choose  $f'_\infty $ and compare $   Q _{k-1}^\phi\lb d(z_0,\delta (x)z_0)\rb$ to $O'(0,x,\Phi_{f'_\infty})$.
  Since    $ 1-2\lambda =\frac{ 1+\inv}{1-\inv }   $ 	
 (see \S \ref{Notations}),  
  $\lambda(\delta)\in (-\infty,0)$ corresponds to $\inv(\delta)\in ( 0,1)$.
  So condition (a) in Proposition   \ref{transfer}  (3) holds. 
  By the truncation above,  condition (b) in Proposition   \ref{transfer}  (3) holds. 
  Choose  $f'_\infty $ by Proposition   \ref{transfer} (3), and  let us look at its consequence. 
  Recall that  $f_\infty
  $ satisfies \eqref{oinf}. 
By  \eqref{PQrelation}  and Proposition   \ref{transfer}  (3),
    \begin{align}\label{Hff1}2{\Vol(\BC^\times/\BR^\times)^2}Q _{k-1}^\phi\lb d(z_0,\delta (x)z_0)\rb=-O'(0,x,\Phi_{f'_\infty})+O(x,f_{ (\infty)})   \end{align}
  and  where  $ f_{(\infty)} \in \cS_c(B(\infty)_\infty)$ is     provided by  Proposition   \ref{transfer}  (3).

  By \eqref{Hff} and \eqref{Hff1},  we have  the following lemma.
   \begin{lem}\label{113}
  Let $f$ and $ f' $ be  as in \S \ref{Mps}.  
 Then  there exists $f_{(\infty)} \in \cS_c(B(\infty)_\infty)$ and an automorphic function on $ X_{\infty}^\circ(\BC)$  such that 
  $$  2 H(f )_ \infty=\frac{1}{c_K}O'(0, f')_\infty+  O(f \otimes f_{(\infty)})+\int_{K^\times  \bsl \BA_K^{ \times} / \BA^{ \times}} \rho(f)D(P_t)    \Omega(t)dt.$$ 
  Here, $D$ is obtained from the third line of \eqref{Hff}, i.e.,
  $$D(P)=\frac{2}{ \Vol(K^\times\bsl \BA_K^{ \times } /\BA^{ \times } )\Vol(U) }          D_{Z_\Omega}  ( P )  ,
$$
which is a priori a smooth bounded    function   on $X(N)^\circ(\BC)$,
and understood as an automorphic function on $ X_{\infty}^\circ(\BC)$.
 \end{lem}
 
  \subsubsection{Conclusion} 
By Lemma \ref{decpure},   Lemma \ref{111}, Lemma \ref{112}   and  Lemma \ref{113}, with the notations as in these lemmas, 
 we have
 \begin{align} \begin{split}
 2 H(f)
   &  = \frac{1}{c_K} O'(0, f')  +  \sum_{p\in S_\nspl} O(f_{S-\{p\}}\otimes f_{(p)}\otimes f^S\otimes f_\infty) \\
      &+O(0, f_S  \otimes f^S\otimes f_{(\infty)}) +\int_{K^\times  \bsl \BA_K^{ \times} / \BA^{ \times}} \rho(f)D(t) dt \label{ATI}.    \end{split}     \end{align}
  Here $S_\nspl\subset S$ is the subset of places nonsplit in $K$.
  
\subsection{Proof of Theorem    \ref{stronghtder}}\label{proofs}
Before the proof of Theorem    \ref{stronghtder}, let us first restate it  with Tamagawa measures as in \cite{Qiu,YZZ}. This is also convenient for the proof as we will refer to \cite{Qiu} for some details about local-global  decomposition  of the automorphic distribution $\frac{1}{c_K} O'(0, f')  $, and it is important to keep   the same  measures as loc. cit..

Let  the Hamilton quaternion algebra  $\BD$ and the representation $\rho _{2k}$ of  $\BD^\times/\BR^\times$ be  as in \S \ref{oinfsec}. 
   Choose measures so that   
   \begin{align}  \label{volgl}
   \Vol(\GL_2(\BQ_p)= \zeta_{\BQ_p}(2)^{-1}, 
       \end{align}
and       $$\Vol(\BC^\times/\BR^\times)=2 ,\Vol(\BD^\times/\BR^\times)=  4\varpi ^2.$$ 
Here we use $\varpi$ to denote   the usual pi, 3.1415926....

Let  $\alpha^\sharp_{\rho_{2k}}$ be the normalization of $\alpha_{\rho_{2k}}$ by the local $L$-factors as in \S \ref {Local periods} (with $\BQ_p$ replaced by $\BR$ and so on).
For a local admissible representation $\rho$ of the unit group of a quaternion algebra, at a finite or infinite place, make $\alpha^\sharp_{\rho}$ into a distribution on the local Schwartz functions as in  \eqref{asharp} (with  $\alpha$ replaced by $\alpha^\sharp$).
Then by  \eqref{eq:alinf} and \cite[p 169]{MW} on local $L$-factors, we have
   \begin{align}  
 \zeta_{\BR}(2) \alpha^\sharp_{\rho_{2k}}(f_\infty)=  4  \frac{(2k-2)!  }{  (k-1)!\cdot   (k-1)!}      \label{oinf2}.
    \end{align}
Now by \eqref{volgl} and \eqref{oinf2}, Theorem    \ref{stronghtder} is equivalent to the following.
        \begin{thm}      \label{stronghtderre}     Assume that  Conjecture \ref{modconj} holds for $\pi$.
For $\phi \in\pi^\infty$ and ${\wt\phi} \in \wt \pi^\infty$,  we have    \begin{equation}   \label{stronghtder1re}
 \pair{z_{   \phi \otimes{\wt\phi}},z_{\Omega^{-1} }}=\frac{1}{2} \frac{ \zeta_\BQ(2)  L' (1/2 ,\pi_K\otimes \Omega) }{(2 L(1,\eta) )^2L (1,\pi,\ad)}       \alpha^\sharp_{\rho_{2k}}(f_\infty)\alpha_{\pi^\infty}^\sharp( \phi \otimes{\wt\phi} ).
          \end{equation}  
                   \end{thm} 
                   \begin{rmk}\label{k=1}If we pretend $k=1$,
                \eqref{stronghtder1re} is the same as the formula in \cite[Theorem   3.15]{YZZ}, despite  we have the  extra $\frac{1}{2}$  in \eqref{stronghtder1re}  and $f_\infty$ in local factor  
           at $\infty$ here. The reason is as follows. 
                
  First, here we  take representations $\pi^\infty, \wt\pi^\infty$ of weight $2k$ dual to each other tautologically  while 
we  use Hecke action to define   $\pi^\infty\otimes \wt\pi^\infty\to  \ol{CM} (\Omega)$ by sending $ {  \phi \otimes{\wt\phi}}$ to $z_{  \phi \otimes{\wt\phi}}$.   
However,  in \cite{YZZ}, the authors use a geometric realization of the representations $\pi^\infty, \wt\pi^\infty$ of weight $2$so that
  a similar map  to  $\pi^\infty\otimes \wt\pi^\infty\to \vil  \Jac (X(N)) $   is tautological.
But they have to define a pairing between   the representations. They choose a normalization factor $1/\Vol(X(N))$ (see \cite[3.1.3]{YZZ} for the definition). 
Unwinding the definitions, we see that the difference between the definition of CM points or cycles is by $\frac{1/\Vol(X(N))}{\Vol(U)}=2\varpi^2$.
Second,
  our local factor  
           at $\infty$, if we let $k=1$, is $\Vol(\BD^\times/\BR^\times)=4\varpi^2$ times the one in \cite{YZZ}.
           So  the extra $\frac{1}{2}$    is cancelled by $4\varpi^2/2\varpi^2=2$.

\end{rmk}

\begin{proof}[Proof of Theorem    \ref{stronghtderre}]
Let $$P_{\pi^\infty}: =\Hom_{\BA_{K}^{\infty,\times}\times \BA_{K}^{\infty,\times}}\lb (\pi^\infty \otimes \Omega^\infty)\otimes (\wt \pi^\infty\otimes \Omega^{-1,\infty}),\BC\rb ,$$
and for each $\pi_v$, define $P_{\pi^\infty}$ similarly. Then $$P_{\pi^\infty}=\otimes_{p<\infty} P_{\pi_p}.$$
The theorem of Tunnell \cite{Tun} and  Saito \cite{Sai} implies   that  $\dim P_{\pi_v}\leq 1$ so that   $\dim P_{\pi^\infty}=1$.

Since Conjecture \ref{modconj} holds for $\pi$,
  both sides of \eqref{stronghtder1re} are in $P_{\pi^\infty}.$
We may assume that $P_{\pi^\infty}$ is nonzero, otherwise, \eqref{stronghtder1} becomes trivial.  
So $\dim P_{\pi_p}=1$ for all $p$ so that   $\dim P_{\pi^\infty}=1$. 
  Thus we only need to prove \eqref{stronghtder1} for one element   $\Psi\in \pi^\infty  \otimes \wt \pi^\infty$ that is not annihilated by  $\alpha_{\pi^\infty}^\sharp$.
 (In fact, it is known $\alpha^\sharp _{\pi_p}\neq 0$ in this case, see for example  \cite[6.2]{Qiu}. Moreover,  $\alpha_{\pi^\infty}^\sharp\neq 0$, as its evaluation at a vector is an infinite product with almost all factors $1$.)  
We will choose  this element   $\Psi\in \pi^\infty  \otimes \wt \pi^\infty$ to be  the image of $f^\vee$ in $\pi^\infty\otimes \wt\pi^\infty\subset \End( \wt\pi^\infty)$ for some $f\in \cS_c(\GL_2(\BA^\infty))$. (Then  $\Psi=\sum_{\phi}   \pi (f )\phi\otimes\wt \phi )$ where    the sum is over a   basis $\{\phi\}$ of $\pi ^\infty$, and $\{  \wt \phi \}$ is the dual basis of $\wt \pi^\infty$). Then the  value   $\alpha_{\pi^\infty}^\sharp(\Psi)$  is  
$\prod_{p<\infty} \alpha^\sharp _{\pi_p} (f_p) $ if $f$ is a pure tensor. 
          
                \begin{lem}\label{anonvan} For every finite place $p $,  there exists $f_p$  such that  $\supp f_p\subset \GL_2(\BQ_p)_\reg$,  $q (\supp f_p)\subset  \Nm(K_p^\times)$, and                   $ \alpha_{\pi_p}^\sharp(f_p) \neq 0.$
                
                \end{lem}
                \begin{proof} For $p$ nonsplit in $K$, by \cite[6.2]{Qiu}, there exists a smooth function $\omega$ on $\GL_2(\BQ_p)$ nonzero at 1, such that $$\alpha_{\pi_p}^\sharp(f_p)=\int_{\GL_2(\BQ_p)}f_p(g)\omega(g)dg.$$
      Let $V$ be an open neighborhood   of 1 on which $\omega$ is constant. Let $V^\circ\subset V\cap \GL_2(\BQ_p)_\reg\cap q^{-1}(  \Nm(K_p^\times)) $ be open compact and nonempty.
                Let $f_p=1_{V^\circ}$.
                
                For $p$ split in $K$, it always holds that  $q (\supp f_p)\subset  \Nm(K_p^\times)$. Then we apply \cite[Theorem A.2]{Zha14}. 
                \end{proof}
                
  Let $N$ be large enough such that the subspace of 
 $U(N)$-invariants    $\pi^{U(N)}\neq \{0\}.$
 For   $p\in S $, let $f_p$ be as in 
 Lemma \ref{anonvan}. 
 Choose  $f'^S\in \cH_E^S$ and $f\in \cH^S$ as in \S \ref{Mps}.
 For  $f$ being a pure tensor, we need to establish     \begin{equation}   \label{stronghtder1redist}
 \pair{ z_{\Omega,f_S f^S,\pi},z_{\Omega^{-1} }}=\frac{1}{2} \frac{ \zeta_\BQ(2)  L' (1/2 ,\pi_K\otimes \Omega) }{(2 L(1,\eta) )^2L (1,\pi,\ad)}       \alpha^\sharp_{\rho_{2k}}(f_\infty)\prod_{p  <\infty  }   \alpha_{\pi_p}^\sharp (f_p) .
          \end{equation}  
    Then let $f_p=1_{\GL_2(\BZ_p)}$ for $p\not \in S$ so that  $\alpha_{\pi_p}^\sharp (f_p) =1$ and thus $\prod_{p  <\infty  }   \alpha_{\pi_p}^\sharp (f_p)  \neq 0$.
     Theorem \ref{stronghtderre}  follows.

 Recall      that we use the base change homomorphism $\bc:\cH_E^S\to \cH^S$
to choose $f^S$ to be  a multiple of  $\bc(f'^S)$, 
 (see \S \ref{Mps}).
 We regard \eqref{ATI} as an equation between linear forms on $f'^S\in \cH_E^S$
    By Theorem \ref{strongmodularity} and choosing $S$ large enough, we can  perform the   spectral decomposition on  \eqref{ATI}, i.e.,
 decomposed into a sum of  characters   $\cH^S_E$ (see \cite{JN,Qiu}). 
 Let use deal with the case that $\pi_E$ is cuspidal (equivalently, $\pi\not\cong \pi\otimes\eta$)  and the other case is similar, see  \cite[Proof of Theorem 3.1.9]{Qiu}.
  For a cuspidal automorphic representation $\sigma$ of $\cH^S_E$, we have  its $\cH^S_E$-character $L_{\sigma^S}$.
  Below, we write $L_{\sigma }$ for $L_{\sigma^S}$  and $L_{\pi }$ for $L_{\pi^S}$   to ease the notations.
Then $L_\pi$ and $L_{\pi\otimes\eta}$ are the only two characters of $\cH^S$ whose restrictions to  $\cH_E^S$ are  $L_{\pi_E}$.
 Let us consider the   $L_{\pi_E}$-components of both sides of  \eqref{ATI}.
 
For the     left hand side of \eqref{ATI}, by \eqref{zpif}, its $L_\pi$-component  is $2 \pair{ z_{\Omega,f_S f^S,\pi},z_{\Omega^{-1} }}
$.
By Lemma \ref{lem:eta2} and that $q(\supp  (f  )\subset \Nm(  \BA_K^{ {\infty}, \times})  $ (see \S \ref{Mps}), this also equals  the $L_{\pi\otimes \eta}$-component.
So its the   $L_{\pi_E}$-component is 
   \begin{equation}   \label{stronghtder1redist1}
4  \pair{ z_{\Omega,f_S f^S,\pi},z_{\Omega^{-1} }}.
 \end{equation} 
(Note that here we do not need Conjecture \ref{modconj}   for $\pi\otimes \eta$.)

For the right hand side, we claim  that for each of  the second,   third and fourth  term,  the $L_{\pi_E}$-component, i.e., the sum of the $L_\pi$-component and $L_{\pi\otimes\eta}$-component  vanishes.  Let us consider  the $L_\pi$-component first. 
 For $p\in S_\nspl$, the $L_\pi$-component of $ O(f_{S-\{p\}}\otimes f_{(p)}\otimes f^S\otimes f_\infty)=0$.
 Indeed,  $B(p)_p$ is  the unique division quaternion  algebra over $\BQ_p$ and  let $\pi'_p$ be the Jacquet-Langlands correspondence of $\pi_p$ to $B(p)_p^\times$.
By definition,  the $L_\pi$-component of $ O(f_{S-\{p\}}\otimes f_{(p)}\otimes f^S\otimes f_\infty)=0$ only depends on the image of $f_{(p)}$ in $\Hom(\pi'_p,\pi'_p)$, and then defines an element in $P_{\pi'_p}$.
 As $P_{\pi_p}\neq \{0\}$,    by  the theorem of Tunnell \cite{Tun} and  Saito \cite{Sai},  $P_{\pi'_p}= \{0\}$. So    the $L_\pi$-component of the second    term on the right hand side of \eqref{ATI} is 0.
Similarly,  since
$\Hom_{ {K}_{\infty}^{\times}} (\pi_\infty \otimes \Omega_\infty,\BC)= \{0\}$ (see \cite{Gro1} and recall $\Omega_\infty$ is trivial)      the $L_\pi$-component of the    third and fourth  term on the right hand side of \eqref{ATI} is 0. (Surely, the vanishing of the   fourth  term also follows  by  considering   weights.) 
Now as $P_{\pi_p}=P_{\pi_p\otimes \eta_p}$ by definition,   the $L_{\pi\otimes\eta}$-components vanish too.
So  the $L_{\pi_E}$-component of  the right hand side of \eqref{ATI}  only comes from the first term, the automorphic distribution $\frac{1}{c_K} O'(0, f')  $. 

We want to relate the  $L_{\pi_E}$-component of $\frac{1}{c_K} O'(0, f')  $  to the right hand side of 
\eqref{stronghtderre}. We only give a sketch  and refer to \cite[Proof of Theorem 3.1.9]{Qiu} for more details. 
Take  $f$ to be a pure tensor.
By   the  local-global  decomposition of the  $L_{\pi_E}$-component of $\frac{1}{c_K} O'(0, f')  $ (see \cite[Section 8]{Qiu} or \cite{JN}), 
$L_{\pi_E}$-component of  the right hand side of \eqref{ATI}     is  the product of 
$ \frac{\zeta_\BQ(2)L' (1/2 ,\pi_K\otimes \Omega)}{L (1,\pi,\ad)} $ and   local distributions
 $I_{\pi_{E,v}}(0, f_v),v\in S\cup\{\infty\}$  on $\GL_{2,E}$ (see \S \ref{Local distributions on})  divided the  local $L$-factors
 of $ \frac{\zeta_\BQ(2)L (1/2 ,\pi_K\otimes \Omega)}{L (1,\pi,\ad)} $   at $v$ 
  (compare with our normalization of  $\alpha^\sharp$  in \S \ref {Local periods}). 
Then by Proposition \ref{the proof is more important for us} (and the discussion below it), 
the $L_{\pi_E}$-component of  the right hand side of \eqref{ATI}   finally becomes
   \begin{equation}   \label{stronghtder1redist2}
  \frac{1}{c_K}  \frac{\zeta_\BQ(2)L' (1/2 ,\pi_K\otimes \Omega)}{L (1,\pi,\ad)} \prod_{v }   \alpha_{\pi_v}^\sharp (f_v),    
 \end{equation} 
   where the product  is over the set of places of $\BQ.$
    Since $\Vol(K^\times\bsl \BA_K^{ \times } /\BA^{ \times } )=2L(1,\eta)$ (see \cite[1.6.3]{YZZ}), we have  $$c_K=\lb2L(1,\eta)\rb^2 \Vol(\BQ^\times\bsl \BA^\times/\BR^\times)=2L(1,\eta) ^2.$$
    
 Now    the  equality  between the $L_{\pi_E}$-components of both sides of  \eqref{ATI} gives  the equality  between 
 \eqref{stronghtder1redist1} and \eqref{stronghtder1redist2}. This is \eqref{stronghtder1redist}.    \end{proof}

\section{Theta lifting and modularity} \label{Weak}
  In this section, we prove the   modularity of CM cycles using an arithmetic mixed Siegel--Weil formula. 

The classical Siegel--Weil formula 
relates  theta series and  Eisenstein series, and its arithmetic variant is 
 central in Kudla's program. In fact,
an arithmetic   Siegel-Weil formula already implicitly appeared   in the work of Gross and Zagier \cite[p 233, (9.3)]{GZ}, where the authors use a linear combination of products   of 
 theta  series and
Eisenstein series, which we call a mixed theta-Eisenstein series.

In \S \ref{Analytic kernel function}-\S\ref{Local comparison II}, we first recall some basics of theta series and Eisenstein series, and  form   mixed theta-Eisenstein series  of $\GL_{2,\BQ}$  of  weight $2k$ \cite{YZZ}. 
          Then we  study  Whittaker functions of  mixed theta-Eisenstein series, locally and globally.  
          Finally,  we
  compare the height pairing $H(f)$  with the  value of the Whittaker function at $g=1$.    
  The   conclusion is 
 an arithmetic mixed Siegel--Weil formula
 \eqref{concl}. We refer to our another work \cite{Qiu1} for the systematic treatment of such a formula for unitary groups.

 In \S \ref{remove}-\S\ref{removeer}, we prove the modularity 
    of CM cycles based on the arithmetic mixed Siegel--Weil formula. After that, in \S \ref{aid}, we prove some technical local  results used in the proof of the   modularity.

 \subsection{Theta series and Eisenstein series} \label{Analytic kernel function}
\subsubsection{Weil representataion}
Let $F$ be a local field   with an absolute value $|\cdot|$. 
Let $\psi$ be a nontrivial additive character of $F$ and endow $F$ with the self-dual Haar measure.  
Let $(V,q)$  be a nondegenerate quadratic space over $F$ of even dimension and endow $V$ with the self-dual Haar measure.  
Let   $\nu$ be the similitude character of  similitude  orthogonal group $\GO(V)$.
The   Weil representation of $\GL_2(F)\times \GO(V)$ on $\cS(V\times F^\times)$   is defined as follows \cite[2.1.3]{YZZ}:
\begin{itemize}
\item
$r(1,h)\Phi(x,u)=\Phi(h^{-1}x,\nu(h)u)$ for $h\in \GO(V)$;
\item $r\lb\begin{bmatrix}a&0\\
 0&a^{-1}\end{bmatrix},1\rb \Phi(x,u)=\chi_{(V,uq)}(a)|a|^{\dim V/2} \Phi(ax, u)$ for $a\in F^\times$, where $\chi_{(V,uq)}
$ is the quadratic character  associated to  the quadratic space $(V,uq)$\footnote{As $\dim V$ is even, $\chi_{(V,uq)}=\chi_{(V,q)}$.};
\item $r \lb\begin{bmatrix}1&b\\
 0&1\end{bmatrix},1\rb \Phi(x,u)=\psi (buq(x)) \Phi(x,u) $ for $b\in F$;
\item $r \lb\begin{bmatrix}1&0\\
 0&a\end{bmatrix},1\rb \Phi(x,u)= |a|^{-\dim V/4} \Phi(x,a^{-1}u)$  for $a\in F^\times$;
\item $r\lb \begin{bmatrix}0&1\\
 -1&0\end{bmatrix},1\rb \Phi(x,u)= \gamma_{(V,uq)} ( \cF_{(V,uq)}\Phi)(x,u)$  for $a\in F^\times$,  where $\gamma_{(V,uq)}
$ (resp. $\cF_{(V,uq)}$) is the  Weil index (resp. Fourier transform  on the $x$-variable) for the quadratic space $(V,uq)$.
\end{itemize}

\begin{rmk}\label{extweil}
(1) The action of $\GO(V)$ extends to all functions on  $V\times F^\times$.

(2) The action of $\GL_2(F)\times \GO(V)$ extends to all functions    $\Phi $   such that for  $u_0\in F^\times$, $\Phi(x,u_0)\in \cS(V)$.
 
 \end{rmk}
Let  $  \rO(V)$  act on $\cS(V)$ in the usual way: $h\in \rO(V)$ sends $\phi(g)\in \cS(V)$ to $\phi(h^{-1}g)$.
 \begin{lem} \label{not2} Assume that $F$ is non-archimedean of   residue characteristic  not 2.  Assume  that $V$ has a self-dual lattice $L$.  Let $ \rO(L)\subset \rO(V)$ be the stabilizer of $L$. Let $\phi\in \cS(V)$  be $\rO(L)$-invariant.   Then for every $a\in F^\times$, there exists $f\in \cS(\GL_2(F))$, right $\GL_2(\cO_F)$-invariant,  such that  
 \begin{equation}\label{not21}\phi\otimes 1_{a\cO_F^\times}=\int_{\GL_2(F)}f(g) r(g,1) (1_{L}\otimes 1_{\cO_F^\times}) dg.\end{equation}
  \end{lem}
  \begin{proof} The lemma is a simple corollary of \cite[Theorem 10.2]{Ho}.
  \end{proof}
 \subsubsection{Theta series and Eisenstein series}We follow \cite[4.1, 6.1]{YZZ}.
 Fix  the standard additive character  $\psi$ of $\BQ\bsl \BA$ (see \cite[1.6.1]{YZZ})
 for defining the global Weil representations.
  
 Let $V_1 $  be a nondegenerate quadratic space over $\BQ$ of even dimension.
For $\Phi_1=\Phi_1^\infty\otimes \Phi_{1,\infty}$, where $\Phi_1^\infty\in \cS(V_1(\BA^\infty)\times \BA^{ \infty,\times})$ and  
 $\Phi_{1,\infty}$ is  a smooth function   on $V_1(\BR)\times \BR^\times$  such that $\Phi_{1,\infty}(x,u )\in \cS(V_1(\BR))$ for every $u$,
   the theta series
 $$ 
  \theta(g, u,\Phi_1)=\sum_{x \in V_1(\BQ) }  r(g,1 ) \Phi_1(x,u)$$ 
   is absolutely convergent for  every $g\in \GL_2(\BA) ,u\in \BA^\times$. Note that $ \theta(g, u,\Phi_1)$ in only left $\SL_2(\BQ)$-invariant.
(Later, we will specify  a concrete $\Phi_{1,\infty}$.)

For  $g\in \GL_2(\BA), h\in \GO(V_\BA)$, define 
\begin{equation}\label{theta}
  \theta(g,h,\Phi_1)=\sum_{(x,u)\in V_1(\BQ)\times \BQ^\times}  r(g,h) \Phi_1(x,u),\end{equation}
The  sum   \eqref{theta} only involves finitely many $u$ and converges absolutely. Then  \eqref{theta} defines  a smooth $\GL_2(\BQ)\times \GO(V)$-invariant function on $\GL_2(\BA)\times \GO(V_\BA)$. 

Let $\BV_2 $ be a nondegenerate quadratic space over $\BA$ of even dimension.
For $\Phi_2=\Phi_2^\infty\otimes \Phi_{2,\infty}$ where $\Phi_2^\infty\in \cS(\BV_2^\infty \times \BA^{ \infty,\times})$ and  
 $\Phi_{2,\infty}$ is a smooth function   on $\BV_{2,\infty} \times \BR ^\times$  such that $\Phi_{2,\infty}(x,u)\in \cS(\BV_{2,\infty})$ for every $u$,
   the Eisenstein series 
\begin{equation}\label{Eis}E(s,g,u,\Phi_2)=\sum _{\gamma\in P(\BQ)\bsl \SL_2(\BQ)} \delta(\gamma g)^sr(\gamma g,1)\Phi_2(0,u)\end{equation}
 is absolutely convergent for $\Re s$  large enough and  it has a meromorphic continuation to the whole complex plane,  and is holomorphic at $s=0$. Here $P$ is the standard Borel subgroup of $\SL_2$ and $\delta$ the standard modular character \cite[1.6.6]{YZZ}. 
Note that $E(s,g,u,\Phi_2)$ in only $\SL_2(\BQ)$-invariant.

Define  a function on $V_1(\BA)\times \BV_2\times \BA^\times$ by
 $$\Phi_1\otimes \Phi_2((x,y),u)= \Phi_1(x,u)\Phi_2(y,u).$$
Define the mixed  theta-Eisenstein series  \cite[2.3, 6.1]{YZZ}  
\begin{equation}\label{mix}I(s,g,\Phi_1\otimes \Phi_2)=\sum_{u\in \BQ^\times}  \theta(g, u,\Phi_1) E(s,g,u,\Phi_2),\end{equation}
which is a smooth $\GL_2(\BQ) $-invariant function on $\GL_2(\BA) $.

   \subsubsection{Local Whittaker function}\label{Local Whittaker function}

 We will use  the local Whittaker functions $W_a(s,g,u,\phi)$  of 
 Eisenstein series \cite[p. 186]{YZZ}, indeed only its value at $g=1$. We denote its value at $g=1$ to be  $W_a(s,u,\phi)$, which  is given as follows.
 
 Let $v$ be a place of $\BQ$. For $a\in \BQ_v$, $u\in \BQ_v^\times$ and $\phi\in \cS(\BV_{2,v}\times  \BQ_v^\times)$, 
define a (not necessarily convergent) Whittaker integral
  \begin{equation}W_a(s,u,\phi)=\int_{F_v}\delta\lb  \begin{bmatrix}0&1\\
 -1&0\end{bmatrix} \begin{bmatrix}1&b\\
 0&1\end{bmatrix}\rb ^sr\lb  \begin{bmatrix}0&1\\
 -1&0\end{bmatrix} \begin{bmatrix}1&b\\
 0&1\end{bmatrix}\rb\phi(0,u)\psi_v(-ab)db.\label{wasu}
 \end{equation}

\subsubsection{Gaussian}\label{Gaussian}
We need special Schwartz functions at the infinite place.
 Fix  the   additive character of $\BR$  to  be the archimedean component $\psi_\infty$ of $\psi$, that is $\psi_\infty(x)=e^{2\pi i x}$.
  
By the  Iwasawa decomposition,  any element in $ \GL_2(\BR)$  can be uniquely  written as 
\begin{equation}
 \begin{bmatrix}c&0\\
 0&c\end{bmatrix}
 \begin{bmatrix}y&x\\
 0&1\end{bmatrix}\begin{bmatrix}\cos \theta&\sin \theta\\
 -\sin \theta&\cos \theta\end{bmatrix},
 \label{iwad}\end{equation}
 where $c\in \BR_{>0}$ and $y\in \BR^\times$.

For non-negative  integer $l$, let 
\begin{equation*}p_l(t)=\sum  _{j=0}^l \binom{l}{j} \frac{(-t)^j}{j!} .\end{equation*}
For  $V=\BC$ and $q(z)=   z\bar z$, define  the standard  Gaussian of weight $2l+1$ 
\begin{equation*}  \Psi(z,u)= p_l(4\pi u q(z))e^{-2\pi u q(z)}  1_{\BR_{>0}} (u).\end{equation*}
 By   \cite[p 350, A2]{Xue1},  $r(g,1)\Psi(0,u) $  is of weight $2l+1$.
 Then  a direct computation shows that
\begin{equation}r(g,1)\Psi(z,u)=\sgn(y) |y|^{1/2} p_l(4\pi u q(z) y) e^{2\pi i u q(z)  (x+yi)} e^{(2l+1)i\theta} 1_{\BR_{>0}}({u y}), \ z\neq 0, \label{hui1}\end{equation}
\begin{equation}r(g,1)\Psi(0,u)=   \sgn(y)   |y|^{1/2}  e^{(2l+1) i\theta} 1_{\BR_{>0}}({u y}) , \label{hui2}\end{equation}
 where $g$ has Iwasawa decomposition  as in \eqref{iwad}.
 \subsection{Local Whittaker functions}
   
  We specialize  the above definitions using quaternion algebras.
 To study the Whittaker function of  our  mixed theta-Eisenstein  series  $I(s,g,\Phi)$ on $  \GL_2(\BA)$ defined    in \eqref{mix}, we first study (a function which will be) its (non-archimedean) local components.

   \subsubsection{Quaternion algebra as a quadratic space}\label{Quaternion algebra as quadratic space}
Let   $B$ be a quaternion algebra  over a field $F$ with reduced norm $q$.
Let $(b_1,b_2) \in B^\times\times B^\times$ act on $B$ by $$ x\mapsto b_1xb_2^{-1}.$$
This gives an embedding of $B^\times\times B^\times$, modulo the diagonal embedding of $F^\times$, into the similitude orthogonal group $\GO(B)$ of the quadratic space $(B,q)$.
We use $[b_1,b_2]$ to denote the  image of $(b_1,b_2)$ in $\GO(B)$.
Then  $\GO(B)=B^\times\times B^\times/F^\times \rtimes \{1,\iota\}$ with
$\iota([b_1,b_2])=[b_2,b_1].
$  
And $\SO(B)=\{[b_1,b_2]: q(b_1)=q(b_2)\}$, $\rO(B)=\SO(B)\rtimes \{1,\iota\}$.

  Let $\BB$ be  an    quaternion algebra   over $\BA_F$. Fix an embedding $\BA_K\incl \BB$ and  a decomposition $\BB=\BA_K\bigoplus  \BA_Kj$
as in 
\ref {matchorb}. Then clearly, it is an orthogonal decomposition.

  \subsubsection{Nonsplit case}\label{k1yu}

  Let $p$ be a finite  place of $\BQ$. Let $\ord_p$ be the standard discrete valuation on $\BQ_p$.

Assume that $p$ is nonsplit in $K$. Let $K_p^1\subset K_p$ be the subset of norm 1 elements, with the measure as in \cite[1.6.2]{YZZ}.
 Let $\eta_p$ be the quadratic character of $\BQ_p$ associated to  $K_p$ by the class field theory.

  For $a\in \BQ_p$, $u\in \BQ_p^\times$ and $\phi\in \cS(K_pj_p\times \BQ_p^\times)$, let     $$ W^\circ_a(s,u,\phi)= \gamma_{ (K_pj_p,uq)}^{-1}W_a(s,u,\phi)$$ 
where $\gamma_{(K_pj_p,uq)}$
 is the  Weil index  for the quadratic space $(K_pj_p,uq)$, and $W_a$ is   the Whittaker integral \eqref{wasu}.
\begin{rmk}\label{w00p} We remind the reader that in \cite[6.1.1]{YZZ}, the normalization of the Whittaker integral for $a=0$ differs  from the one for $a\neq 0$. Our normalization here is uniform   for $a=0$ and $a\neq 0$.  This uniformity will be useful in \S \ref{lokk}.
\end{rmk}
 By  \cite[Proposition 6.10 (1)]{YZZ}, we have  \begin{equation}\label{Wcirc}W^\circ_{a}(s,u,\phi)=(1-p^{-s})\sum_{n=0}^\infty p^{-ns+n} \int_{D_n(a)} \phi(t,u) d_u t,
 \end{equation}
 where $d_u t$ is the self-dual measure on the quadratic space $(K_pj_p, uq)$ and 
 $$D_{n}(a)=\{t\in K_pj_p: uq(t)\in a+p^n  \BZ_p\}.$$

 \begin{lem}\label{conhol}If 
 $ a\neq 0$ or $\phi(0,u)= 0$, then \eqref{Wcirc} converges absolutely and defines a holomorphic function on $s$. 
 \end{lem}
 \begin{proof}If $a\neq 0$, then \eqref{Wcirc} is a finite sum. For the case  $\phi(0)= 0$, see Lemma \ref{goodlem} (2).
 \end{proof}
 
  \begin{lem}\label{goodlem}Fix $u\in \BQ_p^\times$.  Let  $W(a)=\frac{d}{ds}|_{s=0}W^\circ_{a}(s,u,\phi)$, as a function on $\BQ_p^\times$.
  
    (1) Regard  $W(a)$ as a function on $\BQ_p$ with singularity at 0. Then it  has compact support.
  
  (2) If $\phi(0,u)= 0$, then $W(a)$   extends to a Schwartz function function on $\BQ_p$.
  
  (3)  Let $u\in \BZ_p^\times$.  If   $\phi(0,u)=1$, then $\frac{L(1,\eta_p)}{\Vol(K_p^1)}  W(a)-\frac{\ord_p (a)}{2}\cdot 1_{\BZ_p}\cdot \log p$    extends to a Schwartz function on $\BQ_p$. (One may replace $\BZ_p$ by any open compact neighborhood of $0$.)
\end{lem}
  \begin{proof} We use \eqref{Wcirc}.
  
 (1)  For  $a$ large enough, $\phi(\cdot, u)$ is 0 on $D_{n}(a)$.

 (2) Without loss of generality, assume that $ uq(\supp \phi(\cdot ,u))\subset p^A\BZ_p^\times$ where $A$ is an integer.   
   If $\ord_p(a)\geq A+1$, then  the  nonzero summands in \eqref{Wcirc} are $n=0,...,A$ (if $A<0$, then every summands in \eqref{Wcirc}  is 0). And for these $n$'s, $D_n(a)$ does not depend on $a$ for $a$ . So
  $W_{a}(s,1,u)$ does not depend on $a$.   Thus  the extension   follows.

  (3)   follows from a direct computation. For example, see \cite[p. 195]{YZZ} and \cite[p. 598]{YuanZ}. Or one may bypass  the computation as follows: by  \cite[p. 195]{YZZ} and \cite[p. 598]{YuanZ}, (3) holds   if $\phi(\cdot, u)=1_{\cO_{K_p}j_p}$.  The general case follows by applying (2).
  \end{proof}  

   For $\phi_1\in \cS(K_p\times \BQ_p^\times)$, $\phi_2\in \cS(K_pj_p\times \BQ_p^\times)$, and  $(y_1,x_2)\in  \BB_p=K_p\oplus K_pj_p $,
  let $$\phi_1\otimes \phi_2((y_1,x_2),u)= \phi_1(y_1,u)\phi_2(x_2,u).$$ 
Let $B(p)_p$ be the   quaternion algebra over $\BQ_p$ non-isomorphism to $\BB_p$. 
Fix an embedding $K_p\incl B(p)_p$ and  a decomposition $B(p)_p= K_p\bigoplus   K_p j(p)$ as in \S \ref {matchorb}.

First, let $p$ be nonsplit in $K$.
For \begin{equation*}(y_1,y_2)\in B(p)_p-K_p =K_p\oplus   (K_p j(p) -\{0\}),\end{equation*} 
and $u\in \BQ_p^\times$,  define
\begin{equation}\label{kphi}k_{\phi_1\otimes \phi_2} ((y_1,y_2),u) = \frac{L(1,\eta_p)}{\Vol(K_p^1)} \cdot  \phi_1(y_1,u) \cdot \frac{d}{ds}|_{s=0}W^\circ_{uq(y_2),p}(s,u,\phi_2) .\end{equation}
For a general  $\phi\in \cS(\BB_p\times \BQ_p^\times)$, we define $k_{\phi}$ by linear extension. Then $k_{\phi}$ is a smooth function on $ (B(p) -K_p )\times \BQ_p^\times$.

Lemma \ref{goodlem} (2) implies the following lemma.
 \begin{lem}\label{morek2} Assume that $\phi $ vanishes on $K_p \times \BQ_p^\times$. Then  $k_{\phi}$  extends to a Schwartz function on $B(p)_p \times  \BQ_p^\times$.
 \end{lem}

 \subsubsection{Split case}\label{k1yu}  
    
Assume that $p$ is split in $K$.  
  For  $u\in \BQ_p^\times$ and $\phi\in \cS(K_pj_p\times \BQ_p^\times)$, let     $$ W^\circ_0(s,u,\phi)= \gamma_{ (K_pj_p,uq)}^{-1}\frac{\zeta_{\BQ_p}(s+1)}{\zeta_{\BQ_p}(s)}W_0(s,u,\phi).$$ 
    \begin{rmk}  
Here our normalization of the Whittaker integral   is the same with the one in \cite[6.1.1]{YZZ}, and  differs  from the one for  in the nonsplit case, see Remark \ref{w00p}.
\end{rmk}
  Note that $\frac{\zeta_{\BQ_p}(s+1)}{\zeta_{\BQ_p}(s)}$ has a zero at $s=0$, and $ W^\circ_0(s,u,\phi)$ has analytic continuation to  $s=0$. 
  In fact, later we will only consider $\phi$ such that $\phi(0,u)=0$. In this case, we  have  an analog of Lemma    \ref{conhol}  so that $ W _0(s,u,\phi)$ is automatically holomorphic at $s=0$.

For $y\in K_p $ 
and $u\in \BQ_p^\times$,  define
\begin{equation}\label{kphispl}c_{\phi_1\otimes \phi_2} (y,u) =   \phi_1(y,u) \cdot \frac{d}{ds}|_{s=0}W^\circ_{0,p}(s,u,\phi_2) .\end{equation}
For a general  $\phi\in \cS(\BB_p\times \BQ_p^\times)$, we define $c_{\phi}$ by linear extension.  

  We have the following  analog of
Lemma \ref{morek2}, which is deduced from an analog of  Lemma \ref{goodlem} (2).
\begin{lem}\label {aidlem20} 

 Assume that $\phi $ vanishes on $K_p \times \BQ_p^\times$. Then  $c_{\phi}$  extends to a Schwartz function on $K_p \times  \BQ_p^\times$.
 \end{lem}

  
  \subsubsection{A local analytic kernel $k  _p(y,x)$}\label{lokk}
  
  Let $p$ be a  finite  place of $\BQ$  nonsplit in $K$.
 Let $U_p$ be an open compact subgroup of $\BB_p^\times$.
Let $t\in K_p^\times $ act on $B(p)_p^\times$ by  right multiplication  by $t^{-1}$, and act on $\BB_p^\times$  via  left multiplication by $t$.
 Define a function $k_p $ on $  B(p)_p^\times  \times ^{K_p^\times}\BB_p^\times/U_p -\{(1,1)\}$
 by \begin{equation*}
 k_p(y,x )=k_{1_{x^{-1}U_p \times q(x )q(U_p ) }}(y ,q(y^{-1}) ). 
   \end{equation*} 
 (If $y\in K_p^\times$ and $x\not\in K_p^\times$, we use Lemma \ref{morek2}.)
The invariance by $K_p^\times$ can be checked directly using  \eqref{wasu} and  \eqref{kphi}.
Note that $k_p$ in fact depends on $U_p$ though we do not indicate this dependence in the notation. 
Later, we will fix a $U_p$, and 
compare $k_p$ with the multiplicity function $m_p$ defined in \S \ref{ipartm}, see Corollary \ref{arithmatch} and Corollary \ref{finally}.

Now we   study  properties of $k_p$. 
First, we consider its support.
\begin{lem} \label{morek0} 
  If $k_p(y,x)\neq 0$, then $q(y)  q(x ) \in q(U_p)$.

\end{lem}
\begin{proof}
Express  $k_{\Phi_p}$ using \eqref{Wcirc} and  \eqref{kphi}, 
 consider  the $\BQ_p^\times$-components of $\BB_p^\times\times\BQ_p^\times $ and  $B(p)_p^\times \times\BQ_p^\times$.
 The lemma follows directly.
  \end{proof}

Second, we consider  $k_p(y,x)$ with $x\in K^\times_p$. By the $K_p^\times$-action, we only need to consider the case $x=1$.

\begin{lem}\label{morek1}There  is an open compact subgroup  $U'$ of $B(p)_p^\times$ such that:  

 (1)  $U'\cap K_p^\times=U_p\cap K_p^\times$ as a subgroup of $K_p^\times$;

(2) the function   $k_p(y,1)-\frac{\ord_p(\lambda(y))}{2}\cdot 1_{U'}\cdot \log p $  on $B(p)_p^\times  -{K_p^\times}$
extends to a smooth function on $B(p)_p^\times$.  Here $\lambda$ is the invariant defined in \S \ref{Notations}.
\end{lem}
\begin{proof}  Let $U_0\subset U_p$ be an open compact subgroup of the form  \begin{equation}U_0= (1+U)\times Uj\subset \BB_p=K_p\oplus K_pj_p\label{equa}
\end{equation}where   $U $ is an open compact subgroup    of $K_p^\times$ (small enough so that $U_0$ is indeed a  group).
Let $\{t_1,...,t_n\}$ be a set  of representatives of $(K_p^\times U_0\cap U_p)/U_0$. 
Let \begin{equation*} 
 k _i(y )=k_{ t_i  1_{ U_0 \times  q(U_p ) }}(y ,q(y^{-1}) ) 
   \end{equation*} 
   and  express  $k_i(y)$ using \eqref{Wcirc} and  \eqref{kphi}.  
 By  Lemma \ref{goodlem} (3),  $k_i(y)$  satisfies the conditions in the lemma.  
Let $k(y)=\sum_{i=1}^n k_i(y) .$ By the $K_p^\times$-invariance of $\lambda$,  $k(y)$ also satisfies the conditions in the lemma.  
Note that  $$1_{ U_p \times  q(U_p )} -\sum_{i=1}^n  1_{ t_i U_0 \times  q(U_p ) }$$ vanishes on $K_p \times \BQ_p^\times$.
By Lemma \ref{morek2}, $k_p(y,1)-k(y)$ extends to a Schwartz function on $B(p)_p \times  \BQ_p^\times$.
Then the lemma follows.  \end{proof}

By Lemma \ref{morek2}  and Lemma \ref{morek1}, we have the following corollary.
\begin{cor}\label{morek20} The function $ k_p $  on  $B(p)_p^\times  \times ^{K_p^\times}\BB_p^\times/U_p-\{(1,1)\}$ is smooth.\end{cor}

Finally,  we consider $k_p(y,x)$ with $x$ far away from $K^\times_p$, in the sense that $\lambda(x)$ is far away from 0. 

\begin{lem}\label{morek}Assume that $\BB_p $ is the matrix algebra.
 Let $V$ be an open compact neighborhood of $0$ in $\BQ_p$. For  $U_p$ small enough and $V$ large enough, 
 $ k_p(y,x )=0$ for all $y\in B(p)_p^\times$ and  $x\in \BB_p^\times$ with $\lambda(x)\not \in V $.
 
 \end{lem}

\begin{proof} Fix a compact subset $C$ of  $  B(p)_p^\times $ such that  $$C\times \BB_p^\times \to B(p)_p^\times  \times ^{K_p^\times}\BB_p^\times $$ is surjective (here we use the fact that $B(p)_p$ is the division algebra, which follows from the assumption that  $\BB_p $ is the matrix algebra).  Since $\lambda$ is $K_p^\times$-invariant, we may and we do    assume that $y\in C$.

By Lemma \ref{morek0},  $ k_p(y,x )= 0$ unless 
$q(x^{-1})\in q(C)q(U_p)$. Now assume $q(x^{-1})\in q(C)q(U_p)$.
 Write $x =(x_1,x_2)$ under  $\BB_p=K_p\oplus K_pj_p$.   Then $\lambda(x)=q(x_2)/q(x)\not \in V $   with $V$ large enough 
 is equivalent to that $\ord_p (q(x_2))$  is small enough. Since $q(x)$ is bounded, $\ord_p (q(x_2))=\ord_p (q(x_1))$. Choose $U_p$ to be as in \eqref{equa}.
  For $u_1\in U,u_2\in Uj$, $$x ( u_1+u_2)=  (x_1u_1+x_2u_2)+(x_2u_1+x_1u_2) $$
Choose $U $ small enough  such that $\ord_p (q(x_2u_1+x_1u_2))=\ord_p (q(x_2))$  for every pair of $u_1\in 1+U ,u_2\in Uj$.  
Thus for the integrand $\phi_2(t,1)$ in \eqref{Wcirc} to be nonzero, we necessarily have $\ord_p (q(t))=\ord_p (q(x_2))$.
 Now we check the condition $ q(t)\in q(y_2)+p^n  \BZ_p$ defining the integration domain $D_{n}(q(y_2))$ of \eqref{Wcirc}, where $n\geq 0$.  
Let $V$ be large enough such that $\ord_p (q(t))=\ord_p (q(x_2))<0$. Since $B(p)_p $ is the division algebra, $q(y_2)\not\in \Nm(K_p^\times)$. Thus 
$D_{n}(q(y_2))$ is empty, and $ k_p(y,x )=0$.  \end{proof}

\subsection{Global Whittaker functions}
 Now we study  the  Whittaker functions of the mixed theta-Eisenstein series, in particular, their holomorphic projections.
   
 \subsubsection{Schwartz functions}

   We further require    $\BB_\infty$ to be division (so the Hamilton quaternion algebra). 
Then in the   decomposition $\BB=\BA_K\bigoplus  \BA_Kj$,   $j_\infty^2>0$. 
Recall that we fixed  an integer $k>1$.
Let  \begin{equation}\Phi_\infty=\Phi_{\infty,1}\otimes \Phi_{\infty,2}\label{phiinf}\end{equation} where  $\Phi_{\infty,1} $ is   the standard  Gaussian on $ K_\infty $ of weight $1$  
and $\Phi_{\infty,2} $ is the    standard  Gaussian on $ K_\infty j $ of weight $2k-1$   defined  in \S \ref{Gaussian}.
For  $\Phi^\infty\in \cS(\BB^{\infty }\times   \BA^{ \infty,\times})$,
let $\Phi=\Phi^\infty\otimes \Phi_{\infty}$. 
 
 \subsubsection{Holomorphic projection}
 For an  automorphic form $f$ on $\GL_2(\BA)$, let $ \Pr f$ be the orthogonal projection of $f$ to the space of holomorphic cusp forms on $\GL_2(\BA)$
of weight $2k$. I.e. for every holomorphic cusp form $\varphi$ on $\GL_2(\BA)$
of weight $2k$, the Petersson inner product of $f$ and $\varphi$ equals the one of $ \Pr f$ and $\varphi$.
Recall that we have fixed the standard additive character $\psi$ of $\BA$ (see \cite[1.6.1]{YZZ}). Denote the $\psi$-Whittaker function of $f$  by $  f_\psi$.
Then   $\Pr f_\psi$
  can be explicitly computed as follows. 

For   a positive integer   $m$, the standard holomorphic   Whittaker function on $\GL_2(\BR)$ of weight $m$ is 
 $$ W^{(m)}(h)=|y|^{m/2} e^{2\pi i   (x+yi)} e^{mi\theta} 1_{\BR_{>0}}({ y})$$ 
where $h$ has Iwasawa decomposition  as in \eqref{iwad}. 

 \begin{prop} \label{holop}  Assume  $$f \lb \begin{bmatrix}a&0\\
 0&b\end{bmatrix} g \rb=  O_{g }\lb (\|a \|\|b\|)^{k-\ep}\rb$$ for some $\ep>0$.
 Then $$f_{\psi}(g)=W^{(2k)}(g_\infty)\int_{Z N \bsl \GL_2(\BR)}  f_\psi(gh)\ol{W^{(2k)}(h)}dh,$$ where $Z$ is the center of $\GL_2(\BR)$ and  $N$ is the upper triangular unipotent  subgroup of $\GL_2(\BR)$.

  \end{prop}

\begin{proof} 
 This is the adelic version of \cite[p. 288, (5.1)]{GZ}. One may also prove it in a way similar to  \cite[Proposition 6.12]{YZZ}   which is in the weight 2 case.  
  \end{proof} 
\subsubsection{Theta  series}\label{Theta  series}
Before the holomorphic projection of  the  Whittaker function of the mixed theta-Eisenstein series, we study the one of a  theta series, which is simpler and will also be used.

Assume that $\BB=B_\BA$ where $B$ is a quaternion algebra over  $\BQ$.
Then we have the theta series  $\theta(g,h,\Phi)$  defined  in \eqref{theta}.

\begin{lem} For   $a,b\in \BR^\times,$    we have  
  $$\theta\lb \begin{bmatrix}a&0\\
 0&b\end{bmatrix} g  ,h,\Phi\rb=O_{g,h}(\|a \|\|b\|)$$
 where $\|a\|=\max\{|a|,|a^{-1}|\}$.
\end{lem}
\begin{proof} 
Only need to check the constant term of $\theta(g,h,\Phi)$, which is the finite sum 
$$\sum_{u\in \BQ^\times} r(g,h)\Phi(0,u).$$  The growth is checked directly by definition.
\end{proof}

Then the $\psi$-Whittaker function $\Pr \theta(g,h,\Phi )_\psi$ of the holomorphic projection of    $\theta(g,h,\Phi)$  can be  computed using  Proposition \ref{holop}  and \eqref{hui1}.  Then we have   
\begin{equation} \label{Prtheta} \Pr \theta(g,h,\Phi )_\psi= \frac{(4\pi)^{k-1}  (k-1)!}{(2k-2)!} \sum_{y \in B^\times }  r(g^\infty,h^\infty) \Phi^\infty (y,q(y)^{-1})  W^{(2k)}_{1}(g_\infty)   \cdot  P_{k-1}(1-2\lambda (y) ) 
\end{equation} 
 if  $h_\infty$ fixes $\Phi_\infty$.
 Here
$W^{(2k)}_{1}$ is the  standard holomorphic  Whittaker function (see   \S \ref{Gaussian}),  
$P_{k-1}(t)$ is the $(k-1)$-th Legendre polynomial (see \S \ref{Notations}) and $\lambda$ is the invariant defined in \S \ref{Notations}.
\begin{rmk}
The computation is  essentially the first integral in the second displayed formula on \cite[p. 293]{GZ}, except now we are the adelic setting. 
  \end{rmk}


\subsubsection{Mixed theta-Eisenstein series}\label{mtes}
Assume that $\BB$ is  an  incoherent quaternion algebra   over $\BA_F$ (see \S \ref{Decomposition under the pure  matching condition}). 
In other words, for an odd number of places $v$ of $\BQ$  (including $\infty$),  $\BB_v$  is the division quaternion algebra over $\BQ_v$. 
For a place $v$ of $\BQ$ nonsplit in $K$, let $B(v)$ be the unique quaternion algebra over $\BQ$ such that $B(v)(\BA_F^v)\cong \BB^v$, and fix an embedding $K\incl B(v)$ so that  $\BA_K^v\subset B(v)(\BA_F^v)$ is identified with the $\BA_K^v\subset  \BB^v$.
 Let $B(p)_p$ be the   quaternion algebra over $\BQ_p$ non-isomorphism to $\BB_p$. 
Fix a decomposition $B(v)= K\bigoplus   K j(v)$ as in \S \ref {matchorb}.

Recall the  mixed theta-Eisenstein  series  $I(s,g,\Phi)$ on $  \GL_2(\BA)$   defined   in \eqref{mix}.
Let $I'(0,g,\Phi)$ be  the  derivative  at $s=0$.
\begin{lem} For   $a,b\in \BR^\times,$  we have   
$$I'\lb 0,\begin{bmatrix}a&0\\
 0&b\end{bmatrix} g  ,\Phi\rb=O_{g}\lb \|a \|(\log \|a\|)\|b\|(\log \|b\|)\rb$$
 where $\|a\|=\max\{|a|,|a^{-1}|\}$.
\end{lem}
\begin{proof} The growth is computed in \cite[Lemma 6.13]{YZZ}  and  \cite[Proposition  6.14]{YZZ} in the weight 2 case, and the same computation can be done in our weight $2k$ setting.  Let  $W(g)$ be the derivative   at 0 of the  $0$-th Whittaker function of $\Phi_{\infty,2}$ (normalized or not). Then we only need   to show that  $$W\lb \begin{bmatrix}a&0\\
 0&b\end{bmatrix} g   \rb  =O_g\lb \|a \|^{1/2}(\log \|a\|)\|b\|^{1/2}(\log \|b\|)\rb.$$ This fact follows  from \cite[Lemma 2.3 (1)]{Yan}.
\end{proof}

We want to give an expression of $ \Pr I'(0,g,\Phi)_\psi $ that is similar to \eqref{Prtheta}. To lighten notations, we were only consider $ \Pr I'(0,1,\Phi)_\psi $, which will be enough for purpose.
We further
assume the following assumption.
\begin{asmp}\label{asmp20}
There  is a finite places $p_0$ nonsplit in $K$ such that $\Phi=\Phi^{p_0}\otimes \Phi_{p_0}$ with $\Phi_{p_0}$ vanishes on $ K_{p_0}\times  \BQ_{p_0}^\times$.
 \end{asmp}
One effect of Assumption \ref{asmp20} is   the following lemma.
\begin{lem} \label{wcl}Let $\phi=\phi^\infty\otimes \phi_{\infty}$ where $\phi^\infty\in \cS(\BV_2^\infty \times \BA^{ \infty,\times})$ is a pure tensor and  
 $\phi_{\infty}$ is a smooth function   on $\BV_{2,\infty} \times \BR ^\times$  such that $\phi_{\infty}(x,u)\in \cS(\BV_{2,\infty})$ for every $u$.
Let $E_0'(0,1,u)$ be the value at $g=1$ of the constant term of  $\frac{d}{ds}|_{s=0}E(s,g,u,\phi)$ (see
 \ref{Eis}).
  Assume that $\phi=\phi_{p^0}\otimes \phi_{p_0}$ such that
  $\phi_{p_0}(0,u)=0$. Then
  $$E_0'(0,1,u)=2   \frac{L(1,\eta_{p_0})}{\Vol(K_{p_0}^1)} \cdot \frac{d}{ds}|_{s=0}  W_{0,p_0}^\circ(s,u,\phi_{p_0})\cdot  \phi_0^{p_0}(0,u).$$
 
    \end{lem}
\begin{proof} The lemma follows from the same reasoning as in  the proof of   \cite[Proposition 6.7]{YZZ} (with different normalization, see Remark \ref{w00p}).   \end{proof}

  For a finite place $p $ nonsplit in $K$,   
let  
\begin{equation} \label{Prmix2} \cK^{(p)}_{\Phi}   = \sum_{y \in  B(p) ^\times   }   \lb \Phi^{p,\infty}  \cdot  k _{   \Phi_p}\rb (y,q(y)^{-1})    \cdot   P_{k-1}(1-2\lambda (y) )  \cdot  W^{(2k)}_{1}(1) 
\end{equation}  where $W^{(2k)}_1$,  $P_{k-1}$, and $\lambda$  are the ones  in
\eqref{Prtheta}. This is a finite sum.
If $p=p_0$, $k _{   \Phi_p}(y,q(y)^{-1}) $ is well-defined for $y\in K^\times$ by Assumption  \ref{asmp20} and Lemma \ref{morek2}.
 If $p\neq p_0 $, the summand  in \eqref{Prmix2} corresponding to $y\in K^\times$  is  understood as 0 by Assumption  \ref{asmp20}. 
 For the infinite place, let  
\begin{equation} \label{Prmix3}\cK^{(\infty)}_{\Phi} =   \sum_{ y  \in  B( \infty) _{>0}-K^\times }   \Phi^\infty (y, q(y)^{-1})    \cdot    Q_{k-1}(1-2\lambda (y)  )   \cdot  W^{(2k)}_{1}(1) ,
\end{equation} where $Q_{k-1}$ is the Legendre function  of the second kind (see \S \ref{Notations}).

The $\psi$-Whittaker function $\Pr I'(0,g,\Phi)_\psi $ of  the (cuspidal) holomorphic projection $\Pr I'(0,g,\Phi) $ of    $I'(0,g,\Phi)$  is  computed using  Proposition \ref{holop}. 
 By \eqref{hui2},  \cite[Lemma 2.3 (2) (3)]{Yan} and the proof of (which is the discussion above) \cite[p. 294, (5.8)]{GZ},  
 we have   
 \begin{equation} \label{Prmix1} \Pr I'(0,1,\Phi)_\psi = - 2 \sum_{v  } \int_{\BA^\times K^\times\bsl\BA_K^\times } \cK^{(v)}_{r(1, [t,t]) \Phi}   dt 
\end{equation}
where the sum is over all nonsplit places of $\BQ$ and   the volume of  $\BA^\times K^\times\bsl\BA_K^\times$ is chosen to be 1.  
Note that $r(1, [t,t]) \Phi$ also satisfies the condition in Assumption \ref{asmp20}.


 \begin{rmk}\label{formally} Formally, the above result is the same as  \cite[Proposition 6.5, Proposition 6.15]{YZZ}, which is in the weight 2 case.    \end{rmk}

 \subsection{Local comparison II}\label {Local comparison II}
  We compare the    height  pairing  $H(f )  = \pair{  Z_{\Omega,f },z_{\Omega^{-1}}} $
 with  $ \Pr I'(0,1,\Phi)_\psi$, the derivative of the mixed Eisenstein series (more precisely,  the value at $g=1$ of  its $\psi$-Whittaker function) 
 computed in \eqref {Prmix1}. The main conclusion is 
 an arithmetic mixed Siegel--Weil formula
 \eqref{concl}.
 
 The proof of \eqref{concl} consists of 9 steps, each is a subsubsection. 
We briefly sketch them as follows:
 \begin{itemize}
 
  \item[(1)] recall/set up notations;
   \item[(2)]  review local heights $H(f )_p ,p<\infty $, especially when $p$ is nonsplit in $K$; 
   
   \item[(3)] for  $p$ is nonsplit in $K$, rewrite local height  $H(f )_p $ using  a Schwartz function $\Phi$ on $ \BB^\infty$;
   \item[(4$\&$5$\&$6$\&$7)]  finish the comparison for $p$ nonsplit in $K$ (we will give instructions there in the beginning);
     \item[(8$\&$9)]    finish the comparison for $p$ split in $K$ and  at $\infty$. 

\end{itemize}
 
\subsubsection{Notations} \label {Local comparison IImn}
Let $\BB$  be  the  
  quaternion algebra $\BB=\BD\times \RM_2(\BA^\infty)$ over $\BA$ where $\BD$ is the unique division quaternon algebra over $\BR$. We consistently use  $ \BB^{\infty,\times} $ to replace   $\GL_2(\BA^\infty)$ used before (see  \S \ref{Height distribution}).
  Thus we do not confuse this  $\GL_2(\BA^\infty)$  with the $\GL_2$ used in the Weil representation in \S \ref{Analytic kernel function}.

With the given 
$\BB$, let $B(v)$ be the quaternion algebra over $\BQ$ as in \S \ref{mtes}, i.e.   for a  finite place $p$ of $\BQ$, let $B(p)$ be the unique   quaternion algebra over $\BQ$ such that $B(p)_v$ is division only for $v=p$ and $v=\infty$, and 
 $B(\infty) $ is the matrix algebra. 
Then $B(v)$'s are the same as in \S \ref{Notations}.

For a quaternion algebra $B$ over a local field, 
 let $ h\in B^\times$ act on $\cS(B\times F^\times)$ by  $r(1,[h,1])$, where $r$ is the Weil representation and  $[h,1]$ is defined in \S \ref{Quaternion algebra as quadratic space}. Then the induced action of $f\in \cS(B^\times  )$ on $\phi\in \cS(B\times F^\times)$ is  given by 
 \begin{equation}\label{rfr}r(f)\phi(x,u): =\int_{h\in B^\times}\phi(h^{-1}x,q(h)u) f(h)dh.\end{equation}

  We    choose   measures so that $$\Vol(  K^\times\bsl \BA_K^{\infty,\times}/\BA^{\infty,\times})=\Vol(  K^\times\bsl  \BA_K^{\infty,\times})=1.$$

\subsubsection{Review height decomposition}\label{rewrite}

 Let ${S }$ be a finite set of  finite places of $\BQ$ which contains 2 and  all  finite  places of $\BQ$   ramified in $K$ or  ramified for $\Omega$.
Recall that  $S_\nspl\subset S$ is the subset of  places nonsplit in $K$.
Let $N$ be  the product of two relatively prime integers which are $\geq 3$ such that   the  prime factors of $N$  are contained in ${S }$. 
Let $ U=U(N)$, the corresponding principal congruence subgroup of $\GL_2(\BA^\infty)$.

We  consider  right $U $-invariant Schwartz functions of the form $f=f_{S }\otimes f^{S }$. (Note that we do not require the bi-$U^S$-invariance of $f^S$ as in \S \ref{Height distribution}.)
  Assume that    $f_{S}=\otimes_{p\in S} f_p$ is a pure tensor for simplicity.
    Also assume the following assumption until \ref{remove}  where we prove Theorem \ref{strongmodularity} and Theorem  \ref{strongermodularity}.
    \begin{asmp}\label{asmp2'}
There exists $p_0\in {S}$ nonsplit in $K$ such that $f_{p_0}$  vanishes on $K_{p_0}^\times   $.
    \end{asmp}

Though Assumption \ref{asmp2'}  is weaker than Assumption \ref{asmp2} (and we relaxed the invariance on $f^{S}$),
  the decompositions of $H(f)$  in \S \ref{Height distribution}  still hold after suitable modifications, which are given below.
  By  Assumption \ref{asmp2'}, $Z_{\Omega,f_S} $ and $Z_{\Omega^{-1}} $ do not intersect. So we have the decomposition   of  $H(f)$  into local heights 
  as in \eqref{Hdec} and \eqref{Hdec2}. We continue to use the notations in \eqref{Hdec} and \eqref{Hdec2}.

We treat the case that  $p$ is split in $K$ in \S \ref{leave}.

Now let $p$ be  a finite place of $\BQ$ nonsplit  in $K$. 
Let $m_p$ be the multiplicity function  on $B(p)_p^\times\times ^{K_p^\times}\BB_p^\times-\{(1,1)\}$  defined in \S \ref{ipartm}.
For $t_1,t_2\in \BA_K^{\infty,\times}$, let 
   \begin{align} \begin{split} 
          h(f, t_1,t_2)_p = &  \sum_{\delta\in  B(p) ^\times  }
    P_{k-1}(1-2\lambda(\delta ))\\
 &    \sum _{x\in \BB^{\infty,\times} /U}   f (x_p )   m_p( t_{1,p } ^{-1}  \delta t_{2,p }  ,   x_p^{-1} )  f(x^p) 1_{ U^p}((t_1^p x ^p)^{-1}  \delta t_2 ^p)     \label{111819}.
  \end{split}     \end{align} 
  Here
$P_{k-1}(t)$ is the $(k-1)$-th Legendre polynomial (see \S \ref{Notations}) and $\lambda$ is the invariant defined in \S \ref{Notations}.

\begin{rmk} (1)  Note that  $m_p( \cdot   ,   x_p^{-1} ) $  only depends on $x_p U_p$ by  Lemma \ref{mnonv} (2). So the sum  over $x$ is well-defined.

(2)
   If $p=p_0$ and  $x_p\in K_p^\times$,   $ f (x _p) =0$  by Assumption \ref{asmp2'}. Then
  $f (x _p)   m_p( t_{1,p } ^{-1}  \delta t_{2,p }  ,   x_p^{-1} ) $  is understood as 0.
If $p\neq p_0 $ and  $\delta\in K^\times$,   $f(x^p) 1_{ U^p}((t_1^p x ^p)^{-1}  \delta t_2 ^p) =0$ by Assumption \ref{asmp2'}.
Then the summand  in \eqref{111819} corresponding to $\delta\in K^\times$  is understood as 0.
  \end{rmk}
  
  Then   \eqref{11181} is rewritten as follows:
   \begin{align} \begin{split} 
   i(f)_ p
   &  =      - \int_{ K^\times\bsl \BA_K^{\infty,\times}/\BA^{\infty,\times}}\int_{  K^\times\bsl  \BA_K^{\infty,\times}}    h(f, t_1,t_2) _p\log p \cdot  \Omega^{-1}( t_{2 })\Omega(t_{1 })   dt_{2 }d t_{1 } \label{11183}.
  \end{split}     \end{align}

\subsubsection{Rewrite the local height following \cite{YZZ}}

Consider the following basic Schwartz function on $ \BB^\infty$: 
 \begin{equation} \label{Testfunction1}
\Phi_0=  1_{U_S \times q(U_S)} \bigotimes  1_{\RM_2(  \wh \BZ^S)\times \wh \BZ^{S,\times}}\bigotimes \Phi_\infty, \end{equation}
where      $\Phi_\infty$ is the Gaussian on $\BB_\infty$ as in \eqref{phiinf}.
      Let  \begin{equation*} \Phi  = r\lb   f\rb \Phi_0 .\label{Testfunction}
\end{equation*}  
Then $\Phi$ is right invariant by $U$, and 
 \begin{equation} \Phi^\infty(x,q(x)^{-1})=   f(x) .\label {simple}\end{equation}
Moreover,  if  $u_Sq(x_S)\not\in q(U_S)$,    then \begin{equation} \Phi_S(x_S,u_S)=0 .\label {simple1}\end{equation}
In particular, Assumption \ref{asmp20} holds for $\Phi$ by Assumption \ref{asmp2'}.

For  $\phi \in \cS(\BB_p  \times \BQ_p^\times)$ right invariant by $U_p$, define a function $m_{\phi} $ on $( B(p)_{p}^ \times-K_p^\times)\times \BQ_p^\times $ as in \cite[Notation 8.3]{YZZ}:
  \begin{equation}\label{meqi}
m_{\phi}(y,u)=\sum_{x\in  \BB_p^\times/U_p}m_p(y ,x^{-1} )\phi(x,uq(y)/q(x)).  \end{equation}  
 The sum in \eqref{meqi} only involves finitely many nonzero terms.
 Moreover, if $\phi(x)=0$ for $x\in K_p \times\BQ_p^\times$, then $m_{\phi}(y,u)$ extends to  $ B(p)_{p}^ \times\times \BQ_p^\times $ by the same formula.



By    Lemma \ref{mnonv} (2),  \eqref{111819} , \eqref{simple} and a direct computation,  we have 
    \begin{align} \begin{split} 
 h(f, t_1,t_2)_p =       &    \sum_{y\in  B(p) ^\times }
    P_{k-1}(1-2\lambda(y ))\\
 &   r(1, [t_1,t_2] ) \Phi^{p,\infty}(y, q(y)^{-1})m_{r(1, [t_1,t_2] ) \Phi_p}(y,q(y)^{-1})  \label{heqi},
  \end{split}   .  \end{align}

       \subsubsection{Comparison  for $p$ nonsplit in $K$}\label{Matp}
 Below in \S \ref{Matp}-\ref{jpart},
 we compare the local component of      $H(f )$ and $ \Pr I'(0,1,\Phi)_\psi$ for a finite place $p$ nonsplit in $K$.
      By \eqref{Prmix1} and  \eqref{11183},  up to  the $j$-part which will be dealt in \S \ref{jpart}, 
      we need to compare   $h(f,t_1,t_2)_p$ and $ \cK^{(p)}_{r(1,[t_1,t_2])\Phi} $
      defined in   \eqref {Prmix2}. 
        Then by  \eqref{heqi}, we are led to compare 
      $k_{\phi} $ and $m_{\phi} \cdot \log p$ for $\phi\in \cS(\BB_p  \times \BQ_p^\times)$.  This is the content of \ref{Matp}
and  \ref{content}.
      
     \subsubsection{Matching for $p\not\in {S}$}\label{Matp}
Assume that $p\not\in {S }$ and nonsplit in $K$.
\begin{lem}
For $y\in B(p)_{p}^ \times$, we
  have 
$$k_{r(1,[t_1,t_2])\Phi_p}(y,q(y)^{-1})=m_{r(1,[t_1,t_2])\Phi_p}(y,q(y)^{-1})\cdot \log p .$$\end{lem}
\begin{proof} 
By the descriptions of $\GO(\BB_p)$ and $\rO(\BB_p)$ in \S \ref{Quaternion algebra as quadratic space} and a direct computaion, we see that $\Phi_p$ is a linear combination of functions in the form of the right hand side of \eqref{not21} in
  Lemma \ref{not2}. By Lemma \ref{not2},  we may assume that $\Phi_p=r(g,1)\Phi_{0,p} $ for some $g$  in the standard Borel subgroup of $\GL_2(\BQ_p)$.
  By definition,  for $a\in F^\times$, we have  \begin{equation*}  k_{\phi}(ay,a^{-1}q(y)^{-1})=k_{ r \lb a,1\rb\phi}(y, q(y)^{-1}).  \end{equation*}  
 Then the lemma follows  
\cite[Lemma 6.6, Proposition 8.8]{YZZ}. 
\end{proof}
   Then  by  \eqref {Prmix2} and \eqref{heqi}, we  have
    \begin{equation*} 
  W^{(2k)}_{1}(1)  h(f,t_1,t_2)_p\log p= \cK^{(p)}_{r(1,[t_1,t_2])\Phi} .\end{equation*}
Since $j(f)_p=0$,  by  \eqref{11183}, 
we  have   \begin{align} \begin{split} 
   W^{(2k)}_{1}(1)   H(f)_ p
   &  = -   \int_{ K^\times\bsl \BA_K^{\infty,\times}/\BA^{\infty,\times}}\int_{  K^\times\bsl  \BA_K^{\infty,\times}}   \cK^{(p)}_{r(1,[t_1,t_2])\Phi}    \Omega^{-1}( t_{2 })\Omega(t_{1 })   dt_{2 }d t_{1 } .\label{hck}
  \end{split}     \end{align}
  
  \subsubsection{Coherence for $p\in S_\nspl$: difference between $i$-part and the analytic kernel}
\label{content}

  Let $p\in S_\nspl$.
  
  \begin{lem}
[{\cite[Lemma 9.3.16]{Qiu}}]\label{9313}
There  is an open compact subgroup  $U'$ of $B(p)_p^\times$ such that:  

 (1)  $U'\cap K_p^\times=U\cap K_p^\times$ as a subgroup of $K_p^\times$;

(2) $m_p(y,1)-\frac{\ord_p(\lambda(y))}{2}\cdot 1_{U'}$ extends to a smooth function on $B(p)_p^\times$. Here $\lambda$ is the invariant defined in \S \ref{Notations}.
\end{lem}
By Lemma \ref{mnonv} (1), Lemma \ref{morek0} (2), Lemma \ref{morek1} and Lemma \ref{9313}, we have 
the following corollary. (This corollary could be regarded as an analog of the ``arithmetic smooth matching" \cite[Proposition 10.4.1]{Qiu}.)
\begin{cor}  The function   $k_p(y,1)- m_p(y,1 )\cdot \log p$  on $y\in  B(p)_{p}^ \times-K_p^\times$ extends to a smooth function on 
$  B(p)_{p}^ \times  $.
\end{cor}

Combined with  Corollary  \ref{morek20}, we have the following corollary.
\begin{cor}\label{arithmatch}  The function   $k_p(y,x)- m_p(y,x )\cdot \log p$    extends to a smooth function on 
$B(p)_p^\times\times ^{K_p^\times}\BB_p^\times$.
\end{cor}
Let $d_p$ be  such a smooth extension. Then $d_p$ is  supported on the union of the supports of $k_p$ and $m_p$. For $\phi\in \cS(\BB_p  \times \BQ_p^\times)$, 
define a smooth function $d_{\phi} $ on $ B(p)_{p}^ \times \times \BQ_p^\times $ by
     \begin{equation}\label{dphi}
d_{\phi}(y,u)=\sum_{x\in  \BB_p^\times/U_p}d_p(y,x^{-1})\phi(x,uq(y)/q(x)).  \end{equation}  
The sum in \eqref{dphi} only involves finitely many nonzero terms.
As we can recover  (part of) $k_{\phi}$ from $k_p$ by the following tautological formula
 \begin{equation*} k_{\phi}(y,q(y)^{-1})=\sum_{x\in  \BB_p^\times/U_p}k_p(y,x^{-1})\phi(x,1/q(x)),  \end{equation*}  
 we  have     $$d_{\phi}(y,q(y)^{-1})=k_{\phi}(y,q(y)^{-1})-m_{\phi}(y,q(y)^{-1}).$$
 
 Moreover, by   Lemma \ref{mnonv} (1),  Lemma \ref{morek0},   Corollary \ref{arithmatch}, and  that $\Phi_p\in \cS(\BB_p^\times  \times \BQ_p^\times)$ (see \eqref{simple} and \eqref{simple1}), 
 we have  $$d_{r(1,[t_1,t_2]) \Phi_p}\in \cS(B(p)_p^\times  \times \BQ_p^\times).$$
 (In particular, its extension to $B(p)_p  \times \BQ_p^\times$ by 0 is a Schwartz function.)
By   \cite[Lemma 6.6, Lemma 8.5]{YZZ},  we have 
$k_{r(1,[t_1,t_2]) \Phi_p} =r(1,[t_1,t_2])k_{  \Phi_p} $ and 
$m_{r(1,[t_1,t_2])\Phi_p} =r(1,[t_1,t_2])m_{  \Phi_p} $
 (see Remark \ref{extweil}).
Thus $d_{r(1,[t_1,t_2]) \Phi_p} =r(1,[t_1,t_2])d_{  \Phi_p} $. Then 
by  \eqref{Prtheta}, \eqref {Prmix2},    \eqref{11183}, \eqref{heqi}, 
 we have 
 \begin{align} \begin{split} 
 &  { W^{(2k)}_{1}(1)}   i(f)_ p
 \\
   =&   -   \int_{ K^\times\bsl \BA_K^{\infty,\times}/\BA^{\infty,\times}}\int_{  K^\times\bsl  \BA_K^{\infty,\times}}   \cK^{(p)}_{r(1,[t_1,t_2])\Phi}   \Omega^{-1}( t_{2 })\Omega(t_{1 })   dt_{2 }d t_{1 } 
   \\
  &  +    \int_{ K^\times\bsl \BA_K^{\infty,\times}/\BA^{\infty,\times}}\int_{  K^\times\bsl  \BA_K^{\infty,\times}}   \Pr \theta(1, [t_1,t_2 ],d_{\Phi_p}\otimes \Phi ^p)_\psi\Omega^{-1}( t_{2 })\Omega(t_{1 })   dt_{2 }d t_{1 } . \label{Cohp2}   \end{split}     \end{align}

     \subsubsection{Coherence for $p\in S_\nspl$: $j(f)_p$}\label{jpart}
  For   $j(f)_p$, the treatment is similar, and we use   the discussion in \S \ref{nonj}. 
  Let $l_C$ be the smooth function on $B(p)_p^\times\times ^{K_p^\times}\GL_2(\BQ_p)$ as in \eqref{1120j}. Here 
$C$  is a   vertical divisor   of  the desingularized   deformation space $  \cM_{U_p}'$.
  
  For $\phi\in \cS(\BB_p  \times \BQ_p^\times)$, define a smooth function $l_\phi$ on $B(p)_p^\times\times   \BQ_p^\times $:
     \begin{equation*} 
l_{\phi}(y,u)=\sum_{x\in  \BB_p^\times/U_p}l_C(y,x^{-1})\phi(x,uq(y)/q(x)) \cdot \log p.  \end{equation*}  
This sum  only involves finitely many nonzero terms.
By Lemma \ref{9123}  and  that $\Phi_p\in \cS(\BB_p^\times  \times \BQ_p^\times)$ (see \eqref{simple} and \eqref{simple1}), 
  it is easy to check that
     \begin{equation}\label{lPp} l_{r(1,[t_1,t_2]) \Phi_p}\in \cS(B(p)_p^\times\times  \BQ_p^\times).\end{equation}
 Then  similar to \eqref{Cohp2}, we have  \begin{align}  
  W^{(2k)}_{1}(1)    j(f)_ p
    =     -   \int_{ K^\times\bsl \BA_K^{\infty,\times}/\BA^{\infty,\times}}\int_{  K^\times\bsl  \BA_K^{\infty,\times}}   \Pr \theta(1,[t_1,t_2 ],l_{\Phi_p}\otimes \Phi ^p)_\psi\Omega^{-1}( t_{2 })\Omega(t_{1 })   dt_{2 }d t_{1 } . \label{Cohp2'}       \end{align}

\begin{lem}\label{finallyl} There   exists an open compact subgroup $U'_p$ of $B(p)_p^\times$ such that 
 $l_{C}$ is left $U'_p$-invariant. In particular, 
  $l_{\phi}$ is left $U'_p$-invariant for all  $\phi\in \cS(\BB_p  \times \BQ_p^\times)$.

         \end{lem}
         \begin{proof} 
      Let $B(p)_p^\times$ act on     the desingularized   deformation space $  \cM_{U_p}'$ by the functoriality of the minimal desingularization.          Let  $U_p'$ be the stabilizer of  $C$.   Then the lemma follows from Lemma \ref{9123}  (2).      \end{proof}

        \subsubsection{Split $p$}\label{leave}
      
Let   $p$ be split in $K$. Then $i(f)_p=0$, $j(f)_p=0$.   The proof is the same as the proof in \S \ref{split0}.  
   
    \subsubsection{Local height at $\infty$} 
    By \eqref{Hdecinf}
     and a similar (and easier, since we do not need holomorphic projection,) process as in   \cite[8.1]{YZZ}  (which contains a sign mistake due to the sign mistake in the definition of the Green function in  \cite[7.1.3]{YZZ}), we have   \begin{align} \begin{split} 
   W^{(2k)}_{1}(1)   H(f)_ \infty
   &  =    -   \int_{ K^\times\bsl \BA_K^{\infty,\times}/\BA^{\infty,\times}}\int_{  K^\times\bsl  \BA_K^{\infty,\times}}   \cK^{(\infty)}_{r(1,[t_1,t_2])\Phi}  \Omega^{-1}( t_{2 })\Omega(t_{1 })   dt_{2 }d t_{1 } .\label{111821}
  \end{split}     \end{align}
See \eqref {Prmix3} for the definition of $\cK^{(\infty)}_{\Phi} $.

 \subsubsection{Conclusion}\label{Conclusion}
 Define  the following holomorphic cusp forms on $\GL_2(\BA)$
of weight $2k$     and central character   $\Omega |_{\BA^\times}$:    \begin{align} \label{intrep}   I(g,\Phi )=\int_{K^\times\bsl \BA_K^{\infty,\times}} \Pr I'(0,g,r(1,[t,1])\Phi)\Omega(t) dt 
  ; \end{align} 
  \begin{align}\label{intrep1}\theta_{d,p}(g,\Phi)=\int_{K^\times\bsl \BA_K^{\infty,\times}/\BA^{\infty,\times}} \int_{K^\times\bsl \BA_K^{\infty,\times}}  \Pr \theta(g, [t_1,t_2], d _{\Phi_p} \otimes \Phi ^p)   \Omega^{-1}( t_{2 })\Omega(t_{1 })   dt_{2 }d t_{1 } ;\end{align}
  \begin{align}\label{intrep2}
\theta_{l,p}(g,\Phi):  =   \int_{ K^\times\bsl \BA_K^{\infty,\times}/\BA^{\infty,\times}}\int_{  \BQ^\times\bsl  \BA^{\infty,\times}}  \Pr \theta(g,[t_1,t_2],l_{\Phi_p}\otimes \Phi ^p)  \Omega^{-1}( t_{2 })\Omega(t_{1 })    dzd t_{1 }  .  \end{align}

Assume Assumption \ref{asmp2'}. Combining     \eqref{Prmix1},  \eqref{hck},  \eqref{Cohp2}, \eqref{Cohp2'} and  \eqref{111821},  we have 
\begin{equation}2   W^{(2k)}_{1}(1)  H(f)= I(1,\Phi)_\psi+\sum_{p\in S_\nspl}(\theta_{d,p}(1,\Phi)_\psi-\theta_{l,p}(1,\Phi)_\psi). \label{concl}\end{equation}
Here each term on the right hand is  the value of  the corresponding $\psi$-Whittaker function at $g=1$.
\begin{rmk}\label{clear}
  In the current set-up,   \textit{modularity of generating series of heights}  means 
a generalization  of \eqref{concl} where on the right hand side, the variable $1$ is replaced by  a general $g\in \GL_2(\BA)$, and the left hand side is modified accordingly. 
 Such a  generalization  is not hard to prove for $g\in 1_{\GL_2(\BA_S)} \GL_2(\BA^S)$,   see \cite[1.5.10, 7.4.3]{YZZ} for  the weight 2 case. 
Moreover, in the weight 2 case, such a  generalization  is proved  in \cite[1.5.10, 7.4.3]{YZZ} for 
 general $g\in \GL_2(\BA)$ by using another  modularity result, proved in a separate work of them \cite{YZZ1}.
In higher weights, S. Zhang \cite{Zha97} established  such a generalization for some $\Phi$, even without Assumption \ref{asmp2'}. From  his result, we deduce an 
analog of \eqref{concl}  for some $\Phi$, see \eqref{conc1}. 
 
  \end{rmk}

\subsection{Proof of  Theorem \ref{strongmodularity}}\label{remove} 
The proof of  Theorem \ref{strongmodularity} is done in \S\ref{remove1}.
 Our strategy to prove the modularity  is  to show that 
 $H(f)$ vanishes if $f$ acts as 0 on some representations, precisely Proposition \ref{strongmodularity1} below. This is done by showing the vanishing of the right hand side of
 \eqref{concl}.
 The proof of this proposition is the focus of this subsection.

 We continue to use the notations in \S \ref{Local comparison II}. 
Recall that $\BB$  is  the  
  quaternion algebra $\BB=\BD\times \RM_2(\BA^\infty)$ over $\BA$ where $\BD$ is the unique division quaternon algebra over $\BR$ and $\RM_2$ denotes  the matrix algebra over $\BZ$. So  $ \BB^{\infty,\times} \cong \GL_2(\BA^\infty)$, but we use the former consistently (unless we need to specify the isomorphism). 
    Thus we do not confuse this  $\GL_2(\BA^\infty)$  with the $\GL_2$ used in the Weil representation in \S \ref{Analytic kernel function}.

Also recall  that $\cA$ is our set  of cuspidal automorphic representations of $\GL_2(\BQ)$ defined in \S \ref{Modularitysec} of weight $2k$. In view our convention on  the  use of the
   quaternion algebra $\BB$, we consistently consider $\cA$ as a set of admissible representations of $\BB^\times$  by the Jacquet-Langlands correspondence to $\BB^\times$.  
   if $\pi\in \cA$, then it is simply viewed as $1_{\BB_\infty^\times}\otimes\pi^\infty$ as a representation of $\BB^\times$ (under $ \BB^{\infty,\times} \cong \GL_2(\BA^\infty)$).
      For an open compact subgroup $V \subset \BB^{\infty,\times}$, let $\cA^{V}\subset \cA$ be the finite subset of   representations with nonzero  $V$-invariant vectors.

 Let $S$ be as in Theorem \ref{strongmodularity}. 
    Let  $N$  be a  product of two relatively prime integers which are $\geq 3$ and assume   that the  prime factors of $N$  are contained in ${S }$.  Let $ U=U(N)$, the corresponding principal congruence subgroup of $\BB^{\infty,\times}= \GL_2(\BA^\infty)$.   Then  $\cH^S$ is the Hecke algebra of bi-$U^S$-invariant Schwartz functions on $(\BB^{S,\infty})^{\times}.$

Recall that $S_\nspl\subset S$ is the subset of places nonsplit in $K$.

  \begin{prop} \label{strongmodularity1}  
   There is an open compact subgroup $V_{S_\nspl}\subset U_{S_\nspl}$ 
  such that if $f ^S\in \cH^S$ acts as 0 on all representations in $\cA^{V_{S_\nspl}U^{S_\nspl}}$,  then for every    pure tensor $f_S=\otimes_{p\in S} f_p$   right invariant by $ U_S$, we have $ H( f_S\otimes f^S)=0$.
\end{prop}

  We will prove Proposition \ref{strongmodularity1} in \S \ref{mod2'}.
  We will first prove it 
 under Assumption \ref{asmp2'} using the comparison \eqref{concl} in \S \ref{mod2}, and then remove the assumption.
  Before the proof, we need some  preparations on the analytic kernels (i.e. the right hand side) in \eqref{concl}.
   Below we always let $f=f_S\otimes f^S$. Assume that    $f_S=\otimes_{p\in S} f_p$  and  $f^S\in \cH^S$. 
The right invariance of $f_S$ will be imposed in \S \ref{mod2}.

  \subsubsection{Integral representation of $L$-function}\label{mod3}  
We study the first term $I(g,\Phi)_\psi$
on the right hand side of the comparison \eqref{concl}, using the fact that it is a kernel function for  integral representations of the twisted base change $L$-functions.
The main results are Lemma \ref{unramShimizu}-\ref{jequishim}. 
  
Let $\Phi_{0}^\infty$ be  a Schwartz function on $ \BB^\infty\times \BA^{\infty,\times}$   invariant by  $r(1,[1,h])$
for $h\in U$. (Recall that $r$ is the Weil representation and  $[1,h]$ is defined in \S \ref{Quaternion algebra as quadratic space}.)
Let  \begin{equation*}  
\Phi_{0}=  \bigotimes_{p\in S} \Phi_{0,p} \bigotimes\Phi_0^{S ,\infty}\bigotimes \Phi_\infty, \end{equation*}
where      $\Phi_\infty$ is the Gaussian as in \eqref{phiinf} (generalizing \eqref{Testfunction1}). 
Let   \begin{equation}\label{phip} \Phi  = r\lb   f\rb \Phi_0.\end{equation} Here  $r(f)$ is defined in \eqref{rfr}.  Then $\Phi$  is  still invariant by  $r(1,[1,h])$
for $h\in U$.

Let  $  I(g,\Phi)$ be the  holomorphic cusp forms on $\GL_2(\BA)$     of weight $2k$ and central character   $\Omega |_{\BA^\times}$
 as in
 \eqref{intrep} (with the current $\Phi$ as the input).
 Let $\sigma$ be a cuspidal automorphic representation of $\GL_{2,\BQ}$ and  $\phi\in \sigma$.  
Consider   the Petersson pairing (not taking complex conjugate of the second term) $\pair{I(g,\Phi),\phi}$ of  
$  I (g,\Phi)$  and $\phi$.
 We assume that the central character of $\sigma$ is $\Omega^{-1}|_{\BA^\times}$ and 
$\sigma_\infty$ is holomorphic discrete   of weight $2k$. Otherwise,   $\pair{I(g,\Phi),\phi}=0$.
Let $\pi$ be the Jacquet-langlands correspondence of $\sigma$ to $\BB$.

Let $W$ be the $\psi^{-1}$-Whittaker function of $\phi$. Without loss of generality,  assume that $W$ is a pure tensor.
Also assume that $\Phi$ is  a pure tensor.
For a place $v$ of $\BQ$, define a local integral  
$$P_v(s,\Omega_v,\Phi_v,W_v)=  \int_{K_v^\times/\BQ_v^\times}\Omega_v(t)dt\int_{N_v\bsl \GL_2(\BQ_v)}\delta(g)^sW_v(g)r((g,1))\Phi_v(t^{-1},q(t)) dg.$$ 
Here $N_v$ is the upper triangular  unipotent subgroup and $\delta$ the standard modular character \cite[1.6.6]{YZZ}. 
Let $$P^\circ_v(s,\Omega_v,\Phi_v,W_v) =\frac{L((s+1)/2,\pi_{K_v}\otimes\Omega_v)}{L(s+1,\eta_v)}P_v(s,\Omega_v,\Phi_v,W_v),$$
where   $\eta$ is the Hecke character of $\BQ^\times $ associated to $K$ by the class field theory.
 Then (up to normalizing the measures) $$P^\circ_v(0,\Omega_v,\Phi_v,W_v) =\alpha^\sharp_{\pi_v}(\Theta_v(\Phi_v\otimes W_v)),$$
 where $\alpha^\sharp_{\pi_v}$ 
is defined in \S \ref {Local periods} (and $\alpha^\sharp_{\pi_\infty}$ is defined similarly), and $\Theta_v(\Phi_v\otimes W_v)\in \pi_v\boxtimes \wt \pi_v$ is the normalized local Shimizu lifting   \cite[2.2.2]{YZZ}.
Let $\Theta=\prod_v \Theta_v$ where the product is over all places of $\BQ$.

 By  \cite[2.3]{YZZ},  (up to normalization of the measures) we have
$$ \pair{I(\cdot,\Phi),\phi}=\frac{d}{ds} |_{s=0} \lb \prod_{v} P_v (s,\Omega,\Phi,W)\rb.$$
where the product is over all places of $\BQ$.

 If $L(1/2,\pi_K\otimes\Omega) =0$, then
\begin{equation}\label{use2} \pair{I(\cdot,\Phi),\phi} =\frac{L'(1/2,\pi_K\otimes\Omega)}{L(1,\eta)} \alpha^\sharp_{\pi}(\Theta(\Phi\otimes W)). \end{equation}

Assume that $L(1/2,\pi_K\otimes\Omega)\neq 0$. Then $\ep(1/2,\pi_K\otimes\Omega)=1$. Since 
$\ep(1/2,\pi_{K_\infty}\otimes\Omega_\infty)\Omega_\infty(-1)=-1$ (indeed, $\Omega_\infty$ is trivial),
   there exists $p\in S $ (necessary nonsplit in $K$) such that  $\ep(1/2,\pi_{K_p}\otimes\Omega_p)\Omega_p(-1)=-1$.
   Let $\pi_p'$ be the Jacquet-Langlands correspondence of  $\sigma_p$ to $B(p)_p^\times $.
  By the    theorem of Tunnell \cite{Tun} and  Saito \cite{Sai},   $\Hom_{K_{p}} (\pi_p\otimes \Omega_p,\BC)=0$ and 
  $\Hom_{K_{p}} (\pi_p'\otimes \Omega_p,\BC)\neq 0$. 
  In particular,  $\alpha^\sharp_{\pi_p}(\Theta_p(\Phi_p\otimes W_p))=0$ and  $\pi_p'\neq \{0\}$.
From $\alpha^\sharp_{\pi_p}(\Theta_p(\Phi_p\otimes W_p))=0$, we have 
  \begin{equation} \label{use1}\pair{I(\cdot,\Phi),\phi} =\frac{L(1/2,\pi_K\otimes\Omega)}{L(1,\eta)}  (P^\circ_p)'(0,\Omega_p,\Phi_p,W_p) \alpha^\sharp_{\pi^p}(\Theta^p(\Phi^p\otimes W^p)).
 \end{equation} 
 We need to control $(P^\circ_p)'(0,\Omega_p,\Phi_p,W_p) $ using $\pi_p'$, as follows.
 \begin{lem}\label{aidlem}   Assume that   $p$ is nonsplit in $K$ such that  $\ep(1/2,\pi_{K_p}\otimes\Omega_p)\Omega_p(-1)=-1$. There is an  open  compact subgroup $U_p'$ of $B(p)_p^\times $  such that   $(P^\circ_p)'(0,\Omega_p,\Phi_p,W_p) = 0$ unless  $\pi_p'$ has a nonzero $U_p'$-invariant vector.  Here $U_p'$ is  uniform for all such $\pi_p$ and  all $f_p\in \cS(\GL_2(\BQ_p))$.  (Note that $ \Phi _p = r\lb   f_p\rb \Phi_{0,p}$, see \eqref{phip}.)
  \end{lem}
 We will prove  Lemma \ref{aidlem} in \S \ref{aid}.

\begin{lem} \label{unramShimizu} There is an open compact subgroup $V_{S_\nspl}\subset U_{S_\nspl}$ 
  such that if    $f ^S\in \cH^S$ acts as 0 on all representations in $\cA^{V_{S_\nspl}U^{S_\nspl}}$,  then    $  I(g,\Phi)=0$. 
\end{lem} 
\begin{proof} 
We show that  there exists $V_{S_\nspl}$ such that for every $f^S$ as in the lemma and 
 every  cusp form $\phi$ of $\GL_{2,\BQ}$,  $\pair{I(g,\Phi),\phi}=0$. Then the  lemma follows.  
 We may and we do assume that $\phi\in \sigma$ as above. Let $\pi$ and $W$ be as above.

Since $\Theta$ is $\BB^\times\times \BB^\times$-equivariant  \cite[2.2.2]{YZZ}  (we only use the second $\BB^\times$ here) and  $\Phi$ is right $U$-invariant,
$\Theta(\Phi\otimes W)\in  \pi\boxtimes \wt \pi$ is $U$-invariant, where $U$ acts on   $\wt\pi$. By  \eqref{use2} and  \eqref{use1}, 
a necessary condition for $\pair{I(\cdot,\Phi),\phi} $ to be nonzero is:
\begin{itemize}
\item [(1)] if $L(1/2,\pi_K\otimes\Omega) =0$, then $\pi\in \cA^U$;

\item[(2)]   if $L(1/2,\pi_K\otimes\Omega) \neq 0$, then there exists $p\in S_\nspl$ such that   the Jacquet-Langlands correspondence of $\sigma$ to $B(p) ^\times $ has a nonzero   $U_p'U^p$-invariant  vector, where $U_p'$ is as in Lemma \ref{aidlem}. There are   only finitely many such $\sigma$ (independent of $\pi$). 
Let $V_p$ an open compact subgroup of $U_p $ such that the Jacquet-Langlands correspondence of  $\sigma_p$, for every such $\sigma$, to $\BB_p^\times $ has a nonzero  $V_p$-invariant vector.  Thus $\pi\in\cA^{V_{S_\nspl}U^{S_\nspl}}$.

\end{itemize}

Let $f^S$ act as 0 on $\pi$, we only need to show that   $\pair{I(\cdot,\Phi),\phi}= 0$. 
Note that  by the $\BB^\times\times \BB^\times$-equivariance of $\Theta$ (we only use the first $\BB^\times$ here), we have 
\begin{equation*}\label{rho1}\Theta^{S,\infty}(\Phi^{S,\infty}\otimes W^{S,\infty})=\pi (f^{S}) \Theta^{S,\infty}\lb  \Phi_{0}^{{S,\infty} }\otimes W^{S,\infty}\rb.\end{equation*}
   Then $\pair{I(\cdot,\Phi),\phi}=0$ by \eqref{use2} and \eqref{use1}. 
\end{proof}

\begin{lem}  \label{equishim} Assume that $\Phi_0^{S,\infty}$ is invariant by the standard maximal compact subgroups of $ \GL_2(\BA^{S,\infty})$ and $ (\BB^{S,\infty})^{\times}$, where $g\in  \GL_2(\BA^{S,\infty})$ acts by  $r(g,1)$ and 
$h\in (\BB^{S,\infty})^{\times}$  acts by   $r(1,[h,1])$.
We have    $ I(g, \Phi)= \rho_{\GL_2}(f^{S,\vee})I(g,\Phi_0)$. 
Here  $\rho_{\GL_2}(f^S)$ is the usual Hecke action of $f^S$, regarded as a Schwartz function on $ \GL_2(\BA^{S,\infty})$ via $(\BB^{S,\infty})^{\times}= \GL_2(\BA^{S,\infty})$.
\end{lem} 
\begin{proof}
It is enough to show that for every  cusp form $\phi$ of $\GL_{2,\BQ}$, we have $$\pair{I(g,\Phi),  \phi}=\pair{I(g,\Phi_0), \rho_{\GL_2}(f^S)\phi}$$
Again, we assume that $\phi\in \sigma$ as above, and  let $\pi$ and $W$ be as above.
 By \eqref{use2} and \eqref{use1},
the lemma follows from the 
following equations  for $p\not \in S$:
 $$\Theta_p( \Phi_{p}\otimes W_p)   =\pi_p(f_p) \Theta_p( \Phi_{0,p}\otimes W_p)= \Theta_p( \Phi_{0,p}\otimes \sigma_p(f_p) W_p) ,$$
 where the first equation is the $\BB^\times_p\times \BB^\times_p$-equivariance of $\Theta_p$, and the second equation is the unramified Shimizu lifting.
 \end{proof} 

Define a ``special value" version of   $   I(g,\Phi )$ by  \begin{align*}   J(g,\Phi )=\int_{K^\times\bsl \BA_K^{\infty,\times}} \Pr I(0,g,r(1,[t,1])\Phi)\Omega(t) dt 
  . \end{align*}  
 \begin{lem}  \label{jequishim}  In  Lemma \ref{unramShimizu} and  Lemma \ref{equishim}, replacing $I(g,\Phi )$ by $J(g,\Phi )$, the same results hold.
 
 \end{lem} 
\begin{proof} 
The proofs are similar to the proofs of Lemma \ref{unramShimizu} and  Lemma \ref{equishim}. It is in fact easier  since we only need to  
use $ \pair{J(\cdot,\Phi),\phi} =\frac{L(1/2,\pi_K\otimes\Omega)}{L(1,\eta)} \alpha^\sharp_{\pi}(\Theta(\Phi\otimes W)) $ instead of \eqref{use2}
and \eqref{use1}. In particular, we do not need Lemma \ref{aidlem}.
\end{proof} 
  
  \subsubsection{Invariance of $d_{\phi} $}We study the second term $\theta_{d,p}(g,\Phi)_\psi$
  on the right hand side of the comparison \eqref{concl}.
The main result is Corollary \ref{finally}, about  its local test function $d_{\phi} $ (see \eqref{dphi}).
The same result for the  local test function $l_{\phi}$ of the  third term $\theta_{l,p}(g,\Phi)_\psi$ was already done in Lemma \ref
 {finallyl}.

 Let  $\lambda$ be the invariant defined in \S \ref{Notations}. We abuse notation and  use $\lambda$ to denote this invariant on both $\BB_p$ and $B(p)_p$.
\begin{lem}\label{morem}
 Let $V$ be an open compact neighborhood of $0$ in $\BQ_p$. For $V$ large enough, there exists an open compact subgroup $U'_p$ of $B(p)_p^\times$ such that 
for all $x\in  \BB_p ^\times$ with $\lambda(x)\not \in V $, $ m_p(y,x )$ is left $U'_p$-invariant  as a function on 
$y$. 
\end{lem}
\begin{proof} We need some  notations  defined in \S \ref{ipartm}.
Let $\cM_{U_p}$ be the base change to  $\cO_{F_p^\ur}$ of the supersingular   formal deformation space of level $U_p$. (This notation is slightly different from \ref{ipartm} for convenience.)
 Let $\cM_{U_p}'$ be  the minimal desingularization of   $\cM_{U_p}$. For  $(y,x)\in B(p)_p^\times\times ^{K_p^\times}\BB^\times_p$, let $(y,x)$
denote its image in $\cM_{U_p}$ and 
 let $(y,x)'$
denote its image in $\cM'_{U_p}$. 
  Then $ m_p(y,x )$ is the intersection number of   $(y,x)'$ and $(1,1)'$ on $\cM_{U_p}'$, and we let $\mu_p(y,x)$ be the  intersection number of   $(y,x)$ and $(1,1)$ on   $\cM_{U_p}$.

We  have  the following explicit formula of $\mu_p(y,x)$
 \cite[Theorem 1.3]{Li}:  $$\mu_p(y,x)= c \int_{g\in U_px} | (1-\lambda(y))- (1-\lambda(g))|_p dg,$$
 where $c$ is a nonzero constant independent of $y,g$.
Recall    $ 1-\lambda(y )=\frac{ 1}{1-\inv(y)}   $  (see \S \ref{Notations}) where  $\inv(y)\in \ep \Nm   (K_p^\times) \cup\{0,\infty\}  $ 
with $\ep\not \in \Nm   (K_p^\times)$ (see \S \ref{matchorb}). Since  $|1-\inv(y)|_p\geq 1$, $|1-\lambda(y)|_p\leq 1$. Thus $\mu_p(y,x)$ does not depend on $y$  for $\lambda(x)$ large enough.

By the projection formula,  $ m_p(y,x )-\mu_p(y,x )$ 
 is  the  intersection number  of $(y,x)'\in \cM_{U_p}'$  with the vertical part $C$ of the preimage of $(1,1)\in \cM_{U_p}$  by the morphism $\cM'_{U_p}\to \cM_{U_p}$. Let $U_p'$   be the stabilizer of   $C$.
 Then the lemma follows from Lemma \ref{9123}  (2).  
\end{proof}

 \begin{rmk}The condition that $\BB_p$  is a matrix algebra is essential in Lemma \ref{morem}. If $\BB_p$ is a division algebra, an explicit formula of the corresponding $\mu_p(y,x)$   is not known yet. 

 \end{rmk}
 
\begin{cor} \label{finally}For $U_p$ small enough, there   exists an open compact subgroup $U'_p$ of $B(p)_p^\times$ such that 
 $ d_p=k_p-m_p$ is left $U'_p$-invariant. In particular, $d_{\phi} $  is left $U'_p$-invariant for all $\phi\in \cS(\BB_p  \times \BQ_p^\times)$.
  \end{cor}
  \begin{proof}Let $V$ be an open compact neighborhood of $0$ in $\BQ_p$ such that both Lemma \ref{morek} and  Lemma \ref{morem}  hold.
  We only need to show that for $x\in  \lambda^{-1}(V)\subset \BB_p^\times$, $d_p(y,x)$ is  left $U'_p$-invariant for some $U'_p$.
   By Lemma \ref{mnonv} (1), Lemma \ref{morek0} and Corollary \ref{arithmatch}, we only need to show that 
  $$\{(y,x)\in B(p)_p^\times\times ^{K_p^\times}\BB_p^\times:   \lambda(x)\in V, \ q(y)q(x)\in q(U_p)\}$$ is compact.
 Up to the  $K_p^\times$-action, we may assume that $y$ is  in a fixed compact subset of $B(p)_p^\times$ (here we use that $B(p)_p$ is a division algebra), so that 
 $q(x)$ is  in a fixed compact subset   $U_1\subset \BZ_p^\times$.  Write $x=a+bj_p$ as in \S \ref{matchorb}. Then  $b$ is  in a fixed compact subset  $U_2\subset K ^\times$ by  $\lambda(x)\in V$. By $q(x)\in U_1$,   $a$ is  in a fixed compact subset  $U_3\subset K $. 
Let  $U_3=U_{3,1}\cup  U_{3,2}$ where $ U_{3,1}$ is a compact subset of  $  K_p^\times $ and $U_{3,2}$ is a small compact neighborhood of  $0$. If $a\in U_{3,1}$, $a+bj\in U_{3,1}(1+ U_{3,1} U_2j)$, which is compact.  
If $a\in U_{3,1}$, then by $q(x)\in U_1$,   $b$ is  in a compact subset  $U_{2,1}\subset K_p^\times $. 
Then  $a+bj\in U_{2,1}(U_{2,1} U_3+  j)$, which is compact.  
  \end{proof}
 
\subsubsection{Proposition \ref{strongmodularity1} under Assumption \ref{asmp2'}}\label{mod2}
 
 Still,  let $\Phi  =  r\lb   f\rb \Phi_0 $ be as in \eqref{phip}.
 Assume  that $f_S=\otimes_{p\in S} f_p$ is  right invariant by $ U_S$.
  For  $p\in S_\nspl$, let the  holomorphic cusp forms   
   $\theta_{d,p}(g,\Phi) $ and $\theta_{l,p}(g,\Phi)$ of $\GL_{2,\BQ}$ on $\GL_2(\BA)$     of weight $2k$ and central character   $\Omega |_{\BA^\times}$     be as in
 \eqref{intrep1} and  \eqref{intrep2}.
 
We have the following analog of Lemma \ref{unramShimizu}.
\begin{lem}\label{444} There is an open compact subgroup $V_{S_\nspl}\subset U_{S_\nspl}$ 
  such that if $f ^S\in \cH^S$ acts as 0 on all representations in $\cA^{V_{S_\nspl}U^{S_\nspl}}$,  then for every $f_S$ satisfying Assumption \ref{asmp2'},
   $\theta_{d,p}(g,\Phi) =\theta_{l,p}(g,\Phi)=0 $ for all  $p\in S_\nspl$.

\end{lem}
\begin{proof}

Let $U_{p}'\subset B(p)^\times$ be such that both Lemma \ref{finallyl} and Corollary \ref{finally} hold.
  There are only finitely many  automorphic representations  $\pi$ of $B(p)^\times$ with  nonzero $U_p'U^p$-invariant vectors such that 
   the central character of $\pi$ is $\Omega^{-1}|_{\BA^\times}$ and 
    the Jacquet-Langlands correspondence of $\pi_\infty$  to $\GL_2(\BR)$ is holomorphic discrete   of weight $2k$.   
 Choose $V_p\subset U_p$  such that  for every such $\pi$,  the Jacquet-Langlands correspondence of $\pi _p$  to $\BB^\times_p$  has   nonzero $V_p$-invariant vectors. Let $V_{S_\nspl}=\prod_{p\in {S_\nspl}} V_p$.
  
  We show that   for every $f=f_S\otimes f^S$ as in the lemma and 
 every  cusp form $\phi$ of $\GL_{2,\BQ}$, the  Petersson pairings   $\pair{\theta_{d,p}(g,\Phi),\phi} $ and $\pair{ \theta_{l,p}(g,\Phi),\phi}  $  are 0, where $p\in S_\nspl$.
Let $\phi$ be in a cuspidal automorphic representation $\sigma$.  
We assume that the central character of $\sigma$ is $\Omega^{-1}|_{\BA^\times}$ and 
$\sigma_\infty$ is holomorphic discrete   of weight $2k$. (Otherwise,   $\pair{\theta_{d,p}(g,\Phi),\phi}=\pair{\theta_{l,p}(\cdot ,\Phi),\phi}=0$ already holds.) 
Consider $\pair{ \Pr  \theta(\cdot,h,d_{\Phi_p}\otimes \Phi_0 ^p) ,\phi}$ as  a function on $h \in \lb B(p)^\times\times B(p)^\times\rb(\BA)$.
 Let us  show that $$\pair{ \Pr  \theta(\cdot,h,d_{\Phi_p}\otimes \Phi ^p) ,\phi}=0.$$  Then  $\pair{\theta_{d,p}(g,\Phi),\phi}=0$ by definition.

 Note that   $\pair{ \Pr  \theta(\cdot,h,d_{\Phi_p}\otimes \Phi_0 ^p) ,\phi}$ lies in the global Shimizu lifting $\pi'\boxtimes\wt\pi'$   of $\sigma$ \cite[2.2]{YZZ}. Here $\pi'$ is the Jacquet--Langalnds correspondence of $\sigma$ to $B(p)^\times$.
 By the $ \lb B(p)^\times\times B(p)^\times\rb(\BA)$-equivariance  of the formation of the theta series, 
  $\pair{ \Pr  \theta(\cdot,h,d_{\Phi_p}\otimes \Phi_0 ^p) ,\phi}$ is  invariant by $U_p'\times U^p\subset \lb B(p)^\times\times B(p)^\times\rb(\BA)$. Thus $ \pi'$ has a nonzero $U_p'U^p$-invariant vector.
By the $ \lb B(p)^\times\times B(p)^\times\rb(\BA)$-equivariance again, 
we have 
  \begin{equation}\label{haoln}\pair{ \Pr  \theta(\cdot,h,d_{\Phi_p}\otimes \Phi ^p) ,\phi}=\pi'(f^p)\pair{ \Pr  \theta(\cdot,h,d_{\Phi_p}\otimes \Phi_0 ^p) ,\phi}.\end{equation}  
Since $f ^S\in \cH^S$ acts as 0 on $\pi' $ by our choice of $V_{S_\nspl}$,  
   $\pair{ \Pr  \theta(\cdot,h,d_{\Phi_p}\otimes \Phi ^p) ,\phi}=0$  by \eqref{haoln}.

By the same argument, we have $\pair{\theta_{l,p}(\cdot ,\Phi),\phi}=0$. The lemma is proved.
 \end{proof} 
 
 Combining \eqref{concl}, Lemma \ref{unramShimizu}  and  Lemma \ref{444}, we have the following corollary.
 \begin{cor}\label{445}  Proposition \ref{strongmodularity1}  holds under Assumption \ref{asmp2'} on $f_S$.  
\end{cor}

 \subsubsection{Proof of Proposition \ref{strongmodularity1}}\label{mod2'}
For a positive integer $N_0$ with prime factors in $S$,  let $U_0(N_0) $ be the  open compact subgroup of $\BB^{\infty,\times}= \GL_2(\BA^\infty) $ corresponding to the usual  congruence subgroup $\Gamma_0(N_0)$ of $\SL_2(\BZ)$.

\begin{thm}\label{deyt}  
 Assume that  $N_0$ is the product of  two relatively prime integers which are $\geq 3$ whose prime factors  are all split in $K$.
Then Proposition \ref{strongmodularity1}   holds  for   $f_S=1_{U_0(N_0)_{ S}}$.
\end{thm} 
\begin{proof}
   The theorem  follows from a result of S. Zhang  \cite[Theorem 0.2.1]{Zha97}, explained as follows.

Let $S_{2k}(N_0)$ be the space of $U_0(N_0)$-invariant automorphic forms of $\GL_2(\BQ)$ holomorphic of weight $2k$.
 For a positive integer $n$ with prime factors outside $S$, let $A_n\in \cH^S$ be as in Remark \ref{yongshangl}.  
Let $\psi_n=\psi(n\cdot)$. 
By  \cite[Theorem 0.2.1]{Zha97}, there exists $ \Phi_{N_0,S}\in \cS(\BB_S^\times)$ such that  for 
\begin{equation*}   \Phi_{N_0}=  \Phi_{N_0,S} \bigotimes \bigotimes  1_{\RM_2(  \wh \BZ^S)\times \wh \BZ^{S,\times}}\bigotimes \Phi_\infty, \end{equation*}
where      $\Phi_\infty$ is the Gaussian as in \eqref{phiinf},
 we have 
$I(g,   \Phi_{N_0}), J(g,   \Phi_{N_0})\in S_{2k}(N_0)$, 
 and 
$$H(1_{U_0(N)_{ S}} \otimes A_n)= a \cdot I(1,   \Phi_{N_0})_{\psi_n}+b\cdot J(1,   \Phi_{N_0})_{\psi_n}. $$
Here  $a,b$ are constants independent of $n$ (indeed, $a,b$ are  multiples of $L(1,\eta^{\infty})$ and $L'(1,\eta^{\infty})$ respectively, see \cite[p. 271]{GZ}). 
Regard  $A_n$ as a Schwartz function on $ \GL_2(\BA^{S,\infty})$ via $(\BB^{S,\infty})^{\times}= \GL_2(\BA^{S,\infty})$.
Let  $\rho_{\GL_2}$ denote the usual Hecke action  by Schwartz functions.
 Then $$I(1,   \Phi_{N_0})_{\psi_n}=\lb\rho_{\GL_2}(A_n) I(1,   \Phi_{N_0})\rb_{\psi}, $$
 and $$J(1,   \Phi_{N_0})_{\psi_n}=\lb\rho_{\GL_2}(A_n) J(1,   \Phi_{N_0})\rb_{\psi}. $$
 Indeed, span $I(g,   \Phi_{N_0})$ (resp. $J(g,   \Phi_{N_0})$) as a linear combination of Hecke eigenforms (for all Hecke operators $A_n$'s).
  Then the corresponding equation follows from the classical relation  between Fourier coefficients and Hecke eigenvalues.
 On 
   $S_{2k}(N_0)$ (as , $ \rho_{\GL_2}(A_n) =\rho_{\GL_2}(A_n^\vee)$.
   Then by    Lemma \ref{equishim} and  Lemma \ref{jequishim}, we have 
\begin{equation}\label{conc1}H(1_{U_0(N)_{ S}} \otimes A_n)= a \cdot I(1,   \Phi )_{\psi}+b\cdot J(1,   \Phi )_{\psi}, \end{equation}
where $\Phi  = r\lb   1_{U_0(N)_{ S}} \otimes A_n\rb \Phi_0$, see \eqref{phip}.
As $A_n$'s span  $\cH^S$  modulo the center  of $\GL_2$,
  the theorem is  implied by  Lemma \ref{unramShimizu} and  Lemma \ref{jequishim}.
\end{proof}

\begin{lem}
 
   Functions satisfying Assumption \ref{asmp2'} and   left translations  of    $1_{U_0(N_0)_{ S}}$ by   $\BA_{K,S}^\times $ span $\cS(\BB_S^\times)$.
   \end{lem}
   \begin{proof}  Let $W\subset \cS(\BB_S^\times)$ be the span.
    Let $p\in S$. Then  $\{f_p\in \cS(\BB_p^\times-K_p^\times)\}$ and  left translations  of    $1_{U_0(N_0)_{ p}}$ by   $ {K_p}^\times $ 
   span $\cS(\BB_p^\times)$.
   Thus  $W$ contains   $\cS(\BB_p^\times)\otimes \{1_{U_0(N_0)_{S-\{p\}}}\}$ and its  left translations  by   $\BA_{K,S}^\times $.
    Thus the lemma is reduced to the same statement for $S-\{p\}$. So the lemma is proved by  induction.
   \end{proof}
   Note that the truth of $ H( f_S\otimes f^S)=0$ and the right invariance of $f_S$  do not change  if we   left translate   $f_S$ by $\BA_{K,S}^\times $.
   Thus Proposition \ref{strongmodularity1}  follows from
 Corollary  \ref{445} and Theorem \ref{deyt}.

 \subsubsection{Proof of  Theorem \ref{strongmodularity}}\label{remove1}       
  We shall prove   Theorem \ref{strongmodularity}  for $ \ol{CM}(\wt \pi)$,  for which our previous set-up is more handy to use.

 For $f_{S}\in \GL_2(\BA_S)$, 
define a subspace of $CM(\Omega^{-1} ,\cH^S)$ by $$ CM(\Omega ^{-1},f_{S}\cH^S) =\{  Z_{\Omega^{-1},f_{S}\otimes f^S } :  f ^S\in \cH^S\}.$$
       Let $ \ol {CM}(\Omega^{-1} ,f_{S}\cH^S) $ be the image of $ CM(\Omega^{-1} ,f_{S}\cH^S) $ in $\ol {CM}(\Omega^{-1})$. 
Recall that
 by     the strong multiplicity one theorem \cite{PS}, $L_{\pi^S}$'s are pairwise non-isomorphic.  
 Claim: there is an open compact subgroup $V_S\subset  \GL_2(\BA_S)$ such that  $ \ol{CM}(\Omega^{-1},f_{S}\cH^S)$,  as an $\cH ^S$-module,
       is  a direct sum  of  $L_{ \pi^S}$'s of multiplicity  at most one, where $\pi\in \cA^{V_S U^S  }$.       
     Then  Theorem \ref{strongmodularity}  follows. 

For the claim, we want to show that  the    $ \cH^S$
action on $ \ol{CM}(\Omega^{-1},f_{S}\cH^S)$ factors through  its action 
on  a direct sum  of  $L_{ \pi^S}$'s. By  \eqref{ZTZ0} and \eqref{ydy},            it is enough to show that if $f^S$ acts  on  $ \pi$ as 0  for  all   $\pi\in \cA^{V_SU^S}$, then     $H( f_0*(f_S^\vee\otimes f^{S}))=0$  for all $f_0\in \cS(\BB^{\infty,\times})$.
This is  Proposition \ref{strongmodularity1}.

  \subsection{Proofs of  Theorem \ref{strongermodularity} and Theorem \ref{Vani}}\label{removeer} 
Similar to \S \ref{remove}, our strategy to prove the modularity   is  to show that 
 $H(f)$ vanishes under suitable conditions, precisely Proposition \ref{strongmodularity100} below. This is again done by showing the vanishing of the right hand side of
 \eqref{concl}.  
       \begin{prop} \label{strongmodularity100} Let  $f_S \in \cS(\GL_2(\BA_S)) $ be right $U_S$-invariant such that Assumption \ref{asmp2'} holds. Let $V_{S_\nspl}\subset U_{S_\nspl}$ be as in   Lemma \ref{unramShimizu} and  Lemma \ref{444}.   
    Assume that  $f ^S\in \cH^S$ acts as 0 on all representations in $\cA^{V_{S_\nspl}U^{S_\nspl}}$ other than $\pi$.     Then $ H( f_S\otimes f^S)=0$ under condition  (1) or (2):
    \begin{itemize}  
 \item[(1)]    $L (1/2 ,\pi_K\otimes \Omega) =0$ and  $f_S$  acts as 0 on $\pi$;
  
 \item[(2)]   $\ep (1/2 ,\pi_{K_v}\otimes \Omega_v)\Omega_v(-1) = -1$ for more than one  $v<\infty$.
\end{itemize}
\end{prop}
We use the proofs  of Lemma \ref{unramShimizu} and  Lemma \ref{444}. 
The new ingredient is  to apply  the Waldspurger formula under the condition $L (1/2 ,\pi_K\otimes \Omega) =0$.

 \begin{proof} 
Let  $\Phi  =  r\lb   f_S\otimes f^S\rb \Phi_0 $ be as in \eqref{phip}. 
By \eqref{concl}, we only need to
  show that for every  cusp form $\phi$ of $\GL_{2,\BQ}$, the  Petersson pairings  $\pair{ I(g,\Phi) ,\phi}$,
  $\pair{\theta_{d,p}(g,\Phi),\phi} $ and $\pair{ \theta_{l,p}(g,\Phi),\phi}  $ for $p\in S_\nspl$  are 0.
   As in the proof of Lemma \ref{unramShimizu} and Lemma \ref{444}, let $\sigma$  be a cuspidal automorphic representation of $\GL_{2,\BQ}$ 
with  central character   $\Omega^{-1}|_{\BA^\times}$ and 
$\sigma_\infty$ holomorphic discrete   of weight $2k$. And we  may assume that $\phi\in \sigma$.  

First, we  claim that  $\pair{\theta_{d,p}(g,\Phi),\phi} $ and $\pair{ \theta_{l,p}(g,\Phi),\phi}  $  are 0  for all  $p\in S_\nspl$.  
The proof is as follows.
Let $\pi'$ be the Jacquet--Langalnds correspondence of $\sigma$ to $B(p)^\times$. 
By the  reasoning in the paragraph of  \eqref{haoln} in the proof of  Lemma \ref{444} and our assumption on $f ^S\in \cH^S$, $\pair{ \Pr  \theta(\cdot,h,d_{\Phi_p}\otimes \Phi ^p) ,\phi}=0$  unless   $\pi'^{S,\infty}=\pi^{S,\infty}$ as representations of $B(p)^\times(\BA^{S,\infty})\cong 
(\BB^{S,\infty})^\times$. 
Assume this is the case. Then $\pi$ is the Jacquet-Langlands correspondence of $\sigma$ to $\BB^\times$.
In particular,
 $L (1/2 ,\pi'_K\otimes \Omega)=L (1/2 ,\pi_K\otimes \Omega)  $, and $\pi'^{p}=\pi^{p}$ as representations of $B(p)^\times(\BA^{p})\cong 
(\BB^{p})^\times$. 
So under condition (1), as
 $L (1/2 ,\pi_K\otimes \Omega) =0$, by   the Waldspurger formula \cite[1.4.2]{YZZ} (or more directly, the fourth equation on \cite[p 44]{YZZ}), the claim follows. 
Now assume condition (2),
  $\ep (1/2 ,\pi_{K_v}\otimes \Omega_v)\Omega_v(-1) = -1$ for more than one  $v<\infty$.  (Such $v$ is necessary in $S_\nspl$.)
 In particular, there is $v\in  S_\nspl-\{p\}$ such that   $\ep (1/2 ,\pi'_{K_v}\otimes \Omega_v)\Omega_v(-1) = -1$.
 By  the theorem of Tunnell \cite{Tun} and  Saito \cite{Sai},
  $\Hom_{K_{v}} (\pi_v'\otimes \Omega_v,\BC)=0$.
  Recall that as  a function on $h \in \lb B(p)^\times\times B(p)^\times\rb(\BA)$,  $\pair{ \Pr  \theta(\cdot,h,d_{\Phi_p}\otimes \Phi_0 ^p) ,\phi}$ lies in the global Shimizu lifting $\pi'\boxtimes\wt\pi'$   of $\sigma$ \cite[2.2]{YZZ}.    
So by the definition \eqref{intrep1} of $ \theta_{d,p}(g,\Phi)  $, $\pair{\theta_{d,p}(g,\Phi),\phi} =0$. Similarly, $\pair{ \theta_{l,p}(g,\Phi),\phi}  =0$.  Thus the claim follows.

Now,  we prove that $\pair{ I(g,\Phi) ,\phi}=0$.
By the reasoning in the proof of Lemma \ref{unramShimizu} and   the assumption on $f^S$, the Petersson pairing $\pair{I(g,\Phi),\phi}$ is 0 unless   the Jacquet-langlands correspondence of $\sigma$ to $\BB$ is $\pi$.
    Assume (1) so that   we are in the case of  \eqref{use2}. Then $ I(g,\Phi )=0$ by the assumption on $f_S$.
    Assume (2). There is $v\in  S_\nspl-\{p\}$ such that   $\ep (1/2 ,\pi_{K_v}\otimes \Omega_v)\Omega_v(-1) = -1$.
Then
     both  \eqref{use2} and  \eqref{use1} are 0 by  the    theorem of Tunnell \cite{Tun} and  Saito \cite{Sai}.
  \end{proof}
\begin{rmk}In condition (1), the part for $f_S$ is only used in proving  that  $\pair{ I(g,\Phi) ,\phi}=0$.
 \end{rmk}

    We shall prove   Theorem \ref{strongermodularity} and Theorem \ref{Vani}  for $ \ol{CM}(\wt \pi)$,  for which our previous set-up is more handy to use.

    Let ${T }$ be a finite set of  finite places of $\BQ$ as in Theorem \ref{strongmodularity}. (We save the  notation $S$   for later use.)
 For  $f\in \cS(\GL_2(\BA_T))\otimes \cH^T$, 
let $z_{\Omega^{-1},f }$ be      the image of $Z_{\Omega^{-1},f}$ in $\ol {CM}(\Omega^{-1})$.            
For $\pi\in \cA^{U^T}$, let $z_{\Omega^{-1},f,\wt\pi}$ be the   $L_{\pi^T} $-component of $z_{\Omega^{-1},f}$ via Theorem \ref{strongmodularity}.  
Then the conclusion in            Theorem \ref{strongermodularity} is then that the    $ \cH^T$
action on $\ol{CM} (\wt \pi)$ factors through  its action 
on $\pi_T$. 
  By  \eqref{ZTZ0},            Theorem \ref{strongermodularity} is equivalent to (1) of the following theorem. Similarly,    Theorem \ref{Vani} is equivalent to (2) of the following theorem.       \begin{thm}  \label{strongermodularitylem0}

      (1)    Assume that   $L (1/2 ,\pi_K\otimes \Omega) =0$.  If    $ f_{T}\in \GL_2(\BA_T)$ acts as 0 on $\pi_T$, then $z_{\Omega^{-1},f_T^\vee\otimes f^T,\wt\pi}=0$ for all 
  $f^T\in\cH^T$.  
  
  (2)  Assume that     $\ep (1/2 ,\pi_{K_p}\otimes \Omega_p)\Omega_p(-1) = -1$ for more than one  $p<\infty$. Then $z_{\Omega^{-1},f_T^\vee\otimes f^T,\wt\pi}=0$ for all 
 $f_{T}\in \GL_2(\BA_T)$ and  $f^T\in\cH^T$. \end{thm}  
     
   \begin{proof}
          
      By definition (see the beginning \ref{Height and derivatives1}), $z_{\Omega^{-1},f_T^\vee\otimes f^T  ,\wt\pi}  $ is given as follows. Let $\cA_1$
      be a finite subset  of $\cA $ containing $\pi$ such that  of $z_{ \Omega ^{-1}, f_T^\vee\otimes f^T}$    lies in
 the sum of the $L_{\wt\pi_1^T}$-components of $\ol {CM}(\Omega^{-1},\cH^T)$ over all $\pi_1\in \cA_1$. Let $f_1^T\in \cH^T$   act   as identity  on   $ \pi_1$   if $\pi_1=\pi$,
 and as 0  otherwise. 
 Then $z_{\Omega^{-1},f_T^\vee\otimes f^T  ,\wt \pi}  =z_{\Omega^{-1},(f_T^\vee\otimes f^T )* f_1^{T,\vee}  }  $.

We want  choose  suitable $\cA_1$ and $f_1^T$ so that we can apply Proposition \ref{strongmodularity100} to prove the theorem. Let $p_0\not \in T$ be a finite place of $\BQ$ inert in $K$ such that $f_{p_0}=1_{\GL_2(\BZ_p)}$.
Let $S=T\cup\{p_0\}$ and let $\cA_1$  be $\cA^{V_{S_\nspl}U^{S_\nspl}}$ in Proposition \ref{strongmodularity100}. 
Let $f_1^S\in \cH^S$  act   as identity  on   $ \pi_1\in \cA_1$   if $\pi_1=\pi$,
 and as 0  otherwise (the same requirement as $f^T$, replacing $T$ by $S$). Let 
 $f_{1,p_0}\in \cH_{p_0}$, the unramified Hecke algebra at $p_0$,  vanish  on $K_{p_0}^\times$ and act as   identity  on   $  \pi$ so that Assumption \ref{asmp2'} holds. (The existence of such $f_{1,p_0}$ is easy to prove.)  Let $f_1^T=f_{1,p_0}\otimes f_1^S$
    By \eqref{ydy},  the theorem follows from Proposition \ref{strongmodularity100}.
  \end{proof}

 \subsection{Aid from the weight 2 case}\label{aid} 
  We prove Lemma \ref{aidlem} in this subsection. Though it is purely local and purely analytic, we prove it via a global method and  use the work of Yuan, S. Zhang and W. Zhang \cite{YZZ} essentially.
 We expect that there  should be a local proof of  Lemma \ref{aidlem}.
 
 First,  we recall   \cite[(1.5.6)]{YZZ} (see also  \cite[7.4.3]{YZZ}, the second equation on page 227).
Let $\BB$ be as in \S \ref{Local comparison IImn}.
 Let    $\Phi^{(2)}=\Phi^{(2)}_p\otimes \Phi^{(2)}_\infty\otimes\Phi^{(2),p,\infty}\in \cS(\BB^\times)$ 
  such that  
   \begin{itemize}
\item[(1)]   $\Phi^{(2)}_p= \Phi_p$ where $\Phi_p$ is the one in Lemma \ref{aidlem};
\item[(2)]   
 $\Phi^{(2)}_\infty$ is the 
 standard   Gaussian used in \cite[4.1.1]{YZZ}  (same with our definition \eqref{phiinf} if we 
let $k=1$ in  \eqref{phiinf});
\item[(3)] $ \Phi^{(2),p,\infty}$  satisfies the degeneracy conditions in \cite[5.2]{YZZ} such that 
$ \alpha^\sharp_{\sigma^{(2),p}}(\Theta^p(\Phi^{(2),p}\otimes W^{(2),p}))\neq 0.$   \end{itemize}
 
 The existence of  such  $W_p^{(2)}$ and $  \Phi^{(2),p,\infty}$
is established in \cite[Proposition 5.8]{YZZ}.
Note that at $p$, we do not require the degeneracy condition  \cite[Assumption 5.3]{YZZ}. 
However,    the   analog of \eqref
{Cohp2} in this  setting still holds by  the same proof. Then  \cite[(1.5.6)]{YZZ}  still holds, 
which is an equation between automorphic forms of $\GL_2(\BQ)$ of weight 2:  \begin{equation} c Z(g,\Omega,\Phi^{(2)})= I(g,\Phi^{(2)}) +\sum_{v\in S_\nspl}(\theta_{d,v}(g,\Phi^{(2)}) -\theta_{l,v}(g,\Phi^{(2)}) ). \label{concl1}\end{equation}
Here $c$ is a nonzero constant, $Z(g,\Omega,\Phi^{(2)})$ is the height pairing between  a CM divisor and its  translation by the generating series of Hecke operators (which is an automorphic form!), and the terms on the  right hand side are defined as  in \S \ref {Conclusion}, with holomorphic projection of weight 2.

 \begin{rmk}  (1) For  $\theta_{d,v}(g,\Phi^{(2)})_\psi$  and $\theta_{l,v}(g,\Phi^{(2)})_\psi$, the holomorphic projection in \S \ref {Conclusion}  is  redundant.
 
 (2) To be precise, in the definition of $\theta_{l,v}(g,\Phi^{(2)})_\psi$,
  the  $l_{\Phi_v}$    is the one   in \cite[YZZ]{YZZ}. It is formed in the same way as the one in
    \eqref{lPp}, with a different    vertical divisor   of  the desingularized   deformation space $  \cM_{U_v}'$.
 So  it has  the same   properties with the one in
    \eqref{lPp}. (Note that we do not need to know the   vertical divisor explicitly.)

  \end{rmk}

 Let the notations be as in Lemma \ref{aidlem}. Then  
 \begin{itemize} \item $B(p)$ is the quaternion algebra over $\BQ$ only division  at $v=p$ and $v=\infty$;     
\item  $\pi_p'$ is an irreducible admissible representation of $B(p)_p^\times$ such that $\Hom_{K_{p}} (\pi_p'\otimes \Omega_p,\BC)\neq 0$;
\item $\pi_p$  is the Jacquet-Langlands correspondence of $\pi_p'$ to $\BB_p^\times$ with   $\Hom_{K_{p}} (\pi_p\otimes \Omega_p,\BC)=0$;
\item $W_p$ is in the $\psi_p^{-1}$-Whittaker model  of  the Jacquet-Langlands correspondence of $\pi_p'$ to $\GL_2(\BQ_p)$.
\end{itemize}
By Lemma \ref{Globalization}, there exists an automorphic representation $\pi^{(2)}$ of $B(p)^\times$ such that \begin{itemize}
\item 
  $\pi^{(2)}_p= \pi'_p$, and  $\pi^{(2)}_\infty$ is the trivial representation;
\item  $\Hom_{K_{v}} (\pi^{(2)}_v\otimes \Omega_v,\BC)\neq 0$ for every  place $v$ of $\BQ$.
 \end{itemize}
 Let $\sigma^{(2)}$ be the Jacquet-Langlands correspondence of $\pi^{(2)}$ to $\GL_2(\BQ)$. 
 Then  $W$ is in the $\psi_p^{-1}$-Whittaker model of $\sigma^{(2)}_p$.
 Let  $\phi^{(2)}\in \sigma^{(2)}$,
 and $\Phi^{(2)}\in \cS(\BB^\times)$ 
  such that  
 the       $\psi^{-1}$-Whittaker function $W^{(2)}$  of  $\phi^{(2)}$ satisfies $W^{(2)}=W_p^{(2)}\otimes W^{(2),p}$  with   $W_p^{(2)}=W_p$. 
 Consider the Petersson inner product of each term of \eqref{concl1} with $\phi^{(2)}$. 
\begin{itemize} \item Let $v\in S_\nspl$.  
Recall that as  a function on $h \in \lb B(v)^\times\times B(v)^\times\rb(\BA)$,  $$\pair{ \Pr  \theta(\cdot,h,d_{\Phi^{(2)}_v}\otimes \Phi_0 ^{(2),v}) ,\phi^{(2)}}$$ lies in the global Shimizu lifting $\pi(v)\boxtimes\wt\pi(v)$   of $\sigma^{(2)}$ \cite[2.2]{YZZ}. Here   $\pi(v)$ is the Jacquet--Langalnds correspondence of $\sigma^{(2)}$ to $B(v)^\times$. If $v\neq p$, then $\pi(v)_p=\pi_p$ as  representations of 
$B(v)_p^\times \cong 
\BB_p^\times$. 
Since $\Hom_{K_{p}} (\pi_p\otimes \Omega_p,\BC)=0$,
by   the definition \eqref{intrep1} of $ \theta_{d,v}(g,\Phi^{(2)})  $, $\pair{\theta_{d,v}(g,\Phi^{(2)}),\phi^{(2)}} =0$. Similarly, $\pair{ \theta_{l,v}(g,\Phi^{(2)}),\phi^{(2)}}  =0$.   

\item  By \cite[Theorem 3.22]{YZZ} and $\Hom_{K_{p}} (\pi_p\otimes \Omega_p,\BC)=0$, $\pair{Z(g,\Omega,\Phi^{(2)}),\phi^{(2)}} =0$. 

   \end{itemize}Thus \eqref{concl1}  implies 
   $$\pair{I(\cdot,\Phi^{(2)}),\phi^{(2)}} +\pair{ \theta_{d,p}(\cdot ,\Phi^{(2)}),\phi^{(2)}}-\pair{ \theta_{l,p}(\cdot ,\Phi^{(2)}),\phi^{(2)}} =0.$$
  Then by \eqref{use1} and the  condition (3) on $\Phi$, Lemma \ref{aidlem} is implied by the following lemma
 \begin{lem} 
  There is an  open  compact subgroup $U_p'$ of $B(p)_p^\times $ such that  
  the Petersson pairings      $\pair{   \theta(\cdot,h,d _{\Phi_p^{(2)}} \otimes  \Phi^{(2),p}) ,\phi^{(2)}}$ and 
   $\pair{   \theta(\cdot,h,l_{\Phi_p^{(2)}}\otimes  \Phi^{(2),p}) ,\phi^{(2)}}$  are 0  for all  $h \in \lb B(p)^\times\times B(p)^\times\rb(\BA)$ unless   $\pi^{(2)}_p$  has a nonzero $U_p'$-invariant vector.   
\end{lem}
\begin{proof} The proof is similar to (part of) the proof of Lemma \ref{444}.
As   functions on $h \in \lb B(p)^\times\times B(p)^\times\rb(\BA)$,    $\pair{   \theta(\cdot,h,d _{\Phi_p} \otimes \Phi^{(2),p}) ,\phi^{(2)}}$ and 
   $\pair{   \theta(\cdot,h,l_{\Phi_p}\otimes \Phi^{(2),p}) ,\phi^{(2)}}$   lie in the global Shimizu lifting \cite[2.2]{YZZ}  of $\sigma^{(2)}$, which is  
   $\pi^{(2)}\boxtimes\wt\pi^{(2)}$  \cite[2.2]{YZZ}.  Let $U_p' $ be as in Lemma \ref{finallyl} and Lemma \ref{morem}.
   By the $ \lb B(p)^\times\times B(p)^\times\rb(\BA)$-equivariance of the formation of the theta series,  $\pair{   \theta(\cdot,h,d _{\Phi_p} \otimes \Phi^{(2),p}) ,\phi^{(2)}}$ and 
   $\pair{   \theta(\cdot,h,l_{\Phi_p}\otimes \Phi^{(2),p}) ,\phi^{(2)}}$ 
are  invariant by $U_p' $, where  $U_p' $ acts on the $\pi^{(2)}$ -component of $\pi^{(2)}\boxtimes\wt\pi^{(2)}$.  
 Then the lemma follows. 
           \end{proof}

\end{document}